\documentclass[11pt]{article}

\usepackage[margin=1in]{geometry}
\usepackage{amsthm}
\usepackage{amssymb}
\usepackage{amsmath}
\usepackage{mathrsfs}
\usepackage{graphicx}
\usepackage{url}
\usepackage{color}
\usepackage{tikz}
\usepackage{graphicx}

\newcommand{\hookuparrow}{\mathrel{\rotatebox[origin=c]{90}{$\hookrightarrow$}}}
\newcommand{\dashuparrow}{\mathrel{\rotatebox[origin=c]{90}{$\dasharrow$}}}
\newcommand{\twoheaddownarrow}{\mathrel{\rotatebox[origin=c]{90}{$\twoheadleftarrow$}}}

\DeclareFontEncoding{LS1}{}{}
\DeclareFontSubstitution{LS1}{stix}{m}{n}
\newcommand*\kay{%
  \text{%
  \fontencoding{LS1}%
  \fontfamily{stixscr}%
  \fontseries{\textmathversion}%
  \fontshape{n}%
  \selectfont\symbol{"6B}}}
\makeatletter
  \newcommand*\textmathversion{\csname textmv@\math@version\endcsname}
  \newcommand*\textmv@normal{m}
  \newcommand*\textmv@bold{b}
\makeatother

\newtheorem{theorem}{Theorem}[section]
\newtheorem{corollary}[theorem]{Corollary}
\newtheorem{lemma}[theorem]{Lemma}
\newtheorem{proposition}[theorem]{Proposition}

\theoremstyle{definition}

\newtheorem{example}[theorem]{Example}
\newtheorem{definition}[theorem]{Definition}
\newtheorem{remark}[theorem]{Remark}

\title{Doubly transitive lines II: Almost simple symmetries}

\author{Joseph W.\ Iverson\footnote{Department of Mathematics, Iowa State University, Ames, IA}\qquad Dustin~G.~Mixon\footnote{Department of Mathematics, The Ohio State University, Columbus, OH}}
\date{}

\begin{document}
\maketitle

\begin{abstract}
We study lines through the origin of finite-dimensional complex vector spaces that enjoy a doubly transitive automorphism group.
This paper classifies those lines that exhibit almost simple symmetries.
We introduce a general recipe involving Schur covers to recover doubly transitive lines from their automorphism group.
Combining our results with recent work on the affine case by Dempwolff and Kantor~\cite{DempwolffK:22}, we deduce a classification of all linearly dependent doubly transitive lines in real or complex space.
\end{abstract}

\section{Introduction}

Given a sequence $\mathscr{L}$ of $n$ lines through the origin of $\mathbb{C}^d$, consider the \textbf{automorphism group} $\operatorname{Aut}\mathscr{L}$ consisting of permutations of $\mathscr{L}$ that can be realized by applying a unitary operator.
We say $\mathscr{L}$ is \textbf{doubly transitive} if $\operatorname{Aut} \mathscr{L}$ acts doubly transitively on $\mathscr{L}=\{\ell_j\}_{j\in[n]}$, that is, for every $i,j,k,l\in[n]$ with $i\neq j$ and $k\neq l$, there exists $Q\in\operatorname{U}(d)$ such that $Q\ell_i=\ell_k$ and $Q\ell_j=\ell_l$.
Doubly transitive lines are highly symmetric and therefore worthy of study in their own right.
Case in point, G.~Higman first investigated doubly transitive lines in \textit{real} spaces (via their combinatorial \textit{two-graph} representation) as a means of studying Conway's sporadic simple group $\operatorname{Co}_3$~\cite{Taylor:77}.
At the time, one might have studied two-graphs with the hope of discovering a new simple group, since the unique minimal normal subgroup of every finite doubly transitive group is either simple or elementary abelian, as detailed further below.

Today, we have another reason to study doubly transitive lines.
By virtue of their high degree of symmetry, under mild conditions, doubly transitive lines are necessarily optimizers of a fundamental problem in discrete geometry, specifically, packing lines $\{\ell_j\}_{j\in[n]}$ through the origin of $\mathbb{C}^d$ so as to maximize the minimum chordal distance:
\[
\min_{\substack{1\leq i<j\leq n}}\sqrt{1-\operatorname{tr}(\Pi_i\Pi_j)},
\]
where $\Pi_j$ denotes orthogonal projection onto $\ell_j$.
The optimal line packings that are most frequently studied in the literature are spanned by \textbf{equiangular tight frames (ETFs)}, which are sequences of unit vectors $\{\varphi_j\}_{j\in[n]}$ in $\mathbb{C}^d$ with the property that there exist constants $A,\mu$ such that
\[
\sum_{j\in[n]}\varphi_j\varphi_j^*= A I,
\qquad
\qquad
|\langle \varphi_i,\varphi_j\rangle|= \mu
\quad
\text{for every}
\quad
i,j\in[n],~i\neq j.
\]
See \cite{FickusMixon:table} for a survey.
Indeed, the vectors in an ETF span lines that maximize the minimum chordal distance by achieving equality in the Welch bound \cite{Welch:74} (also known as the simplex bound \cite{ConwayHardinSloane:96}). 
In general, optimal line packings find applications in compressed sensing \cite{BandeiraFickusMixonWong:13}, multiple description coding \cite{StrohmerHeath:03}, digital fingerprinting \cite{MixonQuinnKiyavashFickus:13}, and quantum state tomography \cite{RenesBlumeKohoutScottCaves:04}.

The proposition below makes explicit when doubly transitive lines correspond to optimal line packings.
Here and throughout, the \textbf{span} of a sequence of lines is the smallest subspace containing those lines, and a sequence of lines is \textbf{linearly dependent} if one of the lines is contained in the span of the others.
Observe that a sequence of $n$ lines is linearly dependent if and only if it spans a space of dimension $d < n$.

\begin{proposition}[Lemma 1.1 in \cite{IversonM:18}]
\label{prop: 2tran implies ETF}
Given $n >d$ doubly transitive lines with span $\mathbb{C}^d$, select unit-norm representatives $\Phi = \{\varphi_j\}_{j\in[n]}$.
Then $\Phi$ is an equiangular tight frame.
\end{proposition}

This paper is the second in a series that studies doubly transitive lines.
In the previous installment~\cite{IversonM:18}, we focused on developing the theory of \textit{Higman pairs} and \textit{roux}.
The former is a special type of Gelfand pair $(G,H)$ in which, as a consequence of part~(b) of the Higman Pair Theorem~\cite{IversonM:18} (see also Theorem~\ref{thm.Higman Pair Theorem} below), each of the irreducible constituents of the permutation representation of $G$ on the cosets of $H$ determines a sequence of linearly dependent doubly transitive lines.
In fact, every sequence of linearly dependent doubly transitive lines is determined by such a Gelfand pair (by part~(a) of the Higman Pair Theorem, which we prove in this paper).
Meanwhile, roux provide a convenient combinatorial generalization of doubly transitive lines in complex spaces, much like how G.~Higman's two-graphs generalize doubly transitive lines in real spaces.

Both Higman pairs and roux play a pivotal role in the present installment, where we provide a general recipe involving Schur coverings to find all complex line sequences whose automorphism group contains a given doubly transitive group.
We apply this machine to classify ``half'' of the doubly transitive lines.
Specifically, a theorem of Burnside gives two possibilities for a doubly transitive permutation group $S$ according to its \textbf{socle} $\operatorname{soc} S$, which for a doubly transitive group is its unique minimal nontrivial normal subgroup.

\begin{proposition}[Burnside~\cite{Burnside:97}]
\label{prop: Burnside}
If $S$ is a doubly transitive permutation group, then exactly one of the following holds:
\begin{itemize}
\item[(I)]
$\operatorname{soc} S$ is a regular elementary abelian subgroup that may be identified with $\mathbb{F}_p^m$ in such a way that $S = \mathbb{F}_p^m \rtimes G_0 \leq \operatorname{AGL}(m,p)$, where $G_0 \leq \operatorname{GL}(m,p)$ is the stabilizer of $0 \in \mathbb{F}_p^m$, or
\item[(II)]
$\operatorname{soc} S = G$ is a nonabelian simple group, and $G\unlhd S\leq\operatorname{Aut}(G)$.
\end{itemize}
\end{proposition}

Groups of type (I) are said to have \textbf{affine type}, while those of type (II) are called \textbf{almost simple}.
The former are characterized by having a normal elementary abelian subgroup.
(For further details, see Chapter~4 of~\cite{DixonMortimer:96}.)
In the time since Burnside, the doubly transitive permutation groups have been classified as a consequence of the classification of finite simple groups~\cite{Cameron:81}.
Since the Schur covers of simple groups were already determined in order to perform this classification, our general recipe is most readily applied to the almost simple doubly transitive groups.
We handle that case in this installment.
The affine case was recently handled by Dempwolff and Kantor~\cite{DempwolffK:22}, and by combining our results with theirs, we are able to deduce the complete classification of linearly dependent doubly transitive lines.

In the theorem below, two line sequences $\mathscr{L} = \{ \ell_j \}_{j\in [n]}$ and $\mathscr{L}' = \{ \ell_j' \}_{j\in [n]}$ in $\mathbb{C}^d$ are called \textbf{unitarily equivalent} if there is a unitary $U \in \operatorname{U}(d)$ and a permutation $\sigma \in S_n$ such that $U \ell_j = \ell_{\sigma(j)}'$ for every $j \in [n]$; then we also say $\sigma$ \textbf{induces} a unitary equivalence $\mathscr{L} \to \mathscr{L}'$.
Meanwhile, a line sequence $\mathscr{L} = \{ \ell_j \}_{j\in [n]}$ in $\mathbb{C}^d$ is \textbf{real} if there is a unitary $U \in \operatorname{U}(d)$ and a choice of unit-norm representatives $\{ \varphi_j \}_{j\in [n]}$ such that $U \varphi_j \in \mathbb{R}^d \subset \mathbb{C}^d$ for every $j \in [n]$.
(This is not the definition of real lines given in~\cite{IversonM:18}, but the notions are equivalent.
See the comments following Proposition~\ref{prop:real lines are real}.)

\begin{theorem}[Classification of linearly dependent doubly transitive lines]
\label{thm:classification}
Let $\mathscr{L}$ be a sequence of $n\geq 2d > 2$ lines with span $\mathbb{C}^d$ such that $S := \operatorname{Aut} \mathscr{L}$ is doubly transitive.
Then one of the following holds:
\begin{itemize}
\item[(I)]
$S$ has a regular normal subgroup $G$ that is elementary abelian of order $n = p^{2m}$, where $p$ is prime.
Following bijections $\mathscr{L} \cong G \cong \mathbb{F}_p^{2m}$, then $S$ is isomorphic to a subgroup of $\operatorname{AGL}(2m,p)$, the stabilizer $S_0 \leq S$ of the line corresponding to $0 \in \mathbb{F}_p^{2m}$ can be identified with
a subgroup of $\operatorname{GL}(2m,p)$, and one of the following holds:

	\begin{itemize}
	\item[(A)]
	The lines are real, and
		\begin{itemize}
		\item[(i)]
		$d = 2^{m-1}(2^m - 1)$, $n = 2^{2m}$, $S_0 = \operatorname{Sp}(2m,2)$, $m > 1$.
		\end{itemize}
	
	\item[(B)]
	The lines are not real, and $(d,n,S_0)$ satisfy one of the following:
		\begin{itemize}
		\item[(ii)]
		$d = p^m(p^m-1)/2$, $n=p^{2m}$, $S_0=\operatorname{Sp}(2m,p)$, $p>2$ is prime, $m \geq 1$.
		\item[(iii)]
		$d = 2$, $n=4$, $|S_0|=3$.
		\item[(iv)]
		$d = 8$, $n = 64$, $S_0=G_2(2)'\cong\operatorname{PSU}(3,3)$.
		\end{itemize}
	
	\end{itemize}
	
\item[(II)]
$S$ has a simple normal subgroup $G \unlhd S \leq \operatorname{Aut} G$ such that one of the following holds:

	\begin{itemize}
	\item[(C)]
	The lines are real, and $(d,n,G)$ satisfy one of the following:
		\begin{itemize}
		\item[(v)]
		$d=(q+1)/2$, $n=q+1$, $G=\operatorname{PSL}(2,q)$, $q\equiv1\bmod4$ is a prime power.
		\item[(vi)]
		$d=q^2-q+1$, $n=q^3+1$, $G=\operatorname{PSU}(3,q)$, $q$ is an odd prime power. 
		\item[(vii)]
		$d=q^2-q+1$, $n=q^3+1$, $G={}^2G_2(q)$, $q=3^{2m+1}$, $m \geq 0$.
		\item[(viii)] 
		$d=(2^{2m-1}-3\cdot2^{m-1}+1)/3$, $n=2^{2m-1}-2^{m-1}$, $G=\operatorname{Sp}(2m,2)$, $m\geq3$.
		\item[(ix)]
		$d=(2^{2m-1}+3\cdot2^{m-1}+1)/3$, $n=2^{2m-1}+2^{m-1}$, $G=\operatorname{Sp}(2m,2)$, $m\geq3$.
		\item[(x)]
		$d=22$, $n=176$, $G=\operatorname{HS}$ (the Higman--Sims group).
		\item[(xi)]
		$d=23$, $n=276$, $G=\operatorname{Co}_3$ (the third Conway group).
		\end{itemize}
	
	\item[(D)]
	The lines are not real, and $(d,n,G)$ satisfy one of the following:
		\begin{itemize}
		\item[(xii)]
		$d=(q+1)/2$, $n=q+1$, $G=\operatorname{PSL}(2,q)$, $q\equiv3\bmod4$ is a prime power, $q > 3$.
		\item[(xiii)]
		$d=q^2-q+1$, $n=q^3+1$, $G=\operatorname{PSU}(3,q)$, $q$ is a prime power, $q > 2$.
		\end{itemize}
	\end{itemize}
\end{itemize}
Conversely, there exist lines satisfying each of (i)--(xiii).
For (i)--(xii), the lines are unique up to unitary equivalence.
For (xiii), the lines are unitarily equivalent to some of those constructed in Theorem~\ref{thm.updated unitary groups}(b).
\end{theorem}

\begin{remark}
\label{rem: coincidences}
The lines in case (xiii) of Theorem~\ref{thm:classification} are not unique, but all equivalences are detailed in Theorem~\ref{thm:unitary equivalences}.
There is exactly one other equivalence among (i)--(xiii).
Namely, there is just one sequence of $n=28$ equiangular real lines in dimension $d=7$ up to unitary equivalence, and it is repeated in (vi), (vii), and (viii).
There are no other coincidences.
All the other lines can be distinguished by their realness and their automorphism groups.
As described in Section~\ref{sec:real}, the real lines of (A) and (C) are equivalent to two-graphs, and their automorphism groups appear in~\cite{Taylor:92}.
They are distinct.
The automorphism groups of (B) are listed above from~\cite{DempwolffK:22}, and those of (D) appear in Section~\ref{sec: equivalence}.
They are also distinct.
\end{remark}

In Theorem~\ref{thm:classification}, the global hypothesis $n\geq 2d$ is imposed only to simplify exposition.
Lines spanned by the unit-norm columns of a tight frame $\Phi \in \mathbb{C}^{d \times n}$ naturally pair with those spanned by columns of a \textbf{Naimark complement} $\Psi \in \mathbb{C}^{(n-d) \times n}$ with 
$\Psi^* \Psi = \frac{n}{n-d}(I - \frac{d}{n} \Phi^* \Phi)$.
The lines spanned by Naimark complements exhibit identical automorphism groups, as an easy consequence of Proposition~\ref{prop:switching equiv is unitary equiv} below.
Hence, each of the line sequences described in Theorem~\ref{thm:classification} with $n>2d$ pairs with another sequence of $n$ doubly transitive lines spanning a space of dimension $n-d$.

Real equiangular lines are equivalent to two-graphs (see Proposition~\ref{prop:real lines are two-graphs}), and constructions for all the examples of (A) and (C) appear in Section~6 of~\cite{Taylor:77}.
We provide additional descriptions of (v) and (vi) in Example~\ref{ex.Paley conference} and Theorem~\ref{thm.updated unitary groups}(b), respectively.
See Example~5.10 of~\cite{IversonM:18} for another description of~(i), and also for (ii) and (iii).
A construction of (iv) appears in Example~5.11 of~\cite{IversonM:18}.
For (xii), see Example~\ref{ex.Paley conference}; another description involving Higman pairs appears in Example~5.8 of~\cite{IversonM:18}.
In addition to Theorem~\ref{thm.updated unitary groups}(b), the lines in (xiii) can be described using Higman pairs as in Example~5.9 of~\cite{IversonM:18}.

As explained in Section~\ref{sec:real}, the real cases (A) and (C) follow Taylor's classification of doubly transitive two-graphs~\cite{Taylor:92}.
The cases in which $n=d^2$ (of which three exist) were classified by Zhu~\cite{Zhu:15}.
The affine case (I) is due to Dempwolff and Kantor~\cite{DempwolffK:22}.
Our contribution to the classification is the almost simple case (II).

\begin{theorem}[Main result]
\label{thm:main result}
Let $\mathscr{L}$ be a sequence of $n\geq 2d > 2$ lines with span $\mathbb{C}^d$ such that $S := \operatorname{Aut} \mathscr{L}$ is doubly transitive and almost simple.
Then $\mathscr{L}$ takes the form described in Theorem~\ref{thm:classification}(II).
Conversely, there exist lines satisfying each of Theorem~\ref{thm:classification}(v)--(xiii). 
For (v)--(xii), the lines are unique up to unitary equivalence.
For (xiii), the lines are unitarily equivalent to some of those constructed in Theorem~\ref{thm.updated unitary groups}(b).
\end{theorem}

In the next section, we verify that doubly transitive real lines amount to doubly transitive two-graphs.
Section~3 then reviews necessary results from the previous installment~\cite{IversonM:18}.
We prove part~(a) of the Higman Pair Theorem in Section~4, and our proof suggests a program for classifying doubly transitive lines.
Section~5 builds up the machinery needed to accomplish such a classification, and then Section~6 performs this classification in the almost simple case, ultimately proving Theorems~\ref{thm:classification} and~\ref{thm:main result}.
Finally, Section~7 finds the automorphism groups in case (D) and uses them to sort out unitary equivalences for (xiii).

\section{
Doubly transitive real lines and two-graphs}
\label{sec:real}

Taylor~\cite{Taylor:92} has completely classified doubly transitive two-graphs, which are known to be equivalent to doubly transitive lines in $\mathbb{R}^d$.
In contrast, this series concerns doubly transitive lines in $\mathbb{C}^d$.
Every line $\ell \subset \mathbb{R}^d$ embeds into a unique line $\ell' \subset \mathbb{C}^d$ through the inclusion $\mathbb{R}^d \subset \mathbb{C}^d$: if $\ell = \operatorname{span}_{\mathbb{R}} \{ \varphi \}$, then $\ell' = \operatorname{span}_{\mathbb{C}} \{ \varphi \}$.
In this section, we show that this embedding preserves automorphism groups.
Consequently, Taylor's results apply for real lines in the complex setting.

Throughout, we abuse notation by letting $\Phi=\{\varphi_j\}_{j\in[n]}$ denote both a sequence in $\mathbb{C}^d$ and the $d\times n$ matrix whose $j$th column is $\varphi_j$.
Its \textbf{Gram matrix} is $\Phi^* \Phi = [ \langle \varphi_j, \varphi_i \rangle ]_{i,j \in [n]} \in \mathbb{C}^{n\times n}$.

\begin{proposition}
\label{prop:real lines are real}
A sequence of lines in $\mathbb{C}^d$ is real if and only if it has unit-norm representatives $\Phi = \{ \varphi_j \}_{j\in [n]}$ whose Gram matrix $\Phi^* \Phi$ resides in $\mathbb{R}^{n\times n} \subset \mathbb{C}^{n\times n}$.
\end{proposition}

\begin{proof}
Let $\mathscr{L} = \{ \ell_j \}_{j\in [n]}$ be a sequence of lines in $\mathbb{C}^d$.
In the forward direction, suppose $\mathscr{L}$ is real, and choose unit-norm representatives $\Phi = \{ \varphi_j \}_{j\in [n]}$ and a unitary $U \in \operatorname{U}(d)$ such that $U \varphi_j \in \mathbb{R}^d \subset \mathbb{C}^d$ for every $j \in [n]$.
Then $[\Phi^* \Phi]_{ij} = \varphi_i^* \varphi_j = (U \varphi_i)^* (U \varphi_j) \in \mathbb{R}$ for every $i,j \in [n]$.
In the reverse direction, suppose $\Phi = \{ \varphi_j \}_{j\in[n]}$ is a choice of unit-norm representatives for $\mathscr{L}$ such that $\Phi^* \Phi$ is real-valued.
Since $\Phi^* \Phi$ is positive semidefinite with rank at most $d$, there exists $\Psi = \{ \psi_j \}_{j \in [n]}$ in $\mathbb{R}^d \subset \mathbb{C}^d$ such that $\Psi^* \Psi = \Phi^*\Phi$.
It is well known that vectors in $\mathbb{C}^d$ are determined up to unitary equivalence by their Gram matrices, so there exists a unitary $U \in \operatorname{U}(d)$ such that $U \varphi_j = \psi_j \in \mathbb{R}^d$ for every $j \in [n]$.
Hence, $\mathscr{L}$ is real.
\end{proof}

The previous installment of this series defined real lines in a third equivalent way using \textit{signature matrices}, which play a significant role in the theory of equiangular lines.
Given any sequence $\mathscr{L}$ of linearly dependent equiangular lines, choose unit norm representatives $\Phi = \{ \varphi_j \}_{j\in [n]}$ satisfying $ | \langle \varphi_i, \varphi_j \rangle | = \mu \neq 0$ for every $i \neq j \in [n]$, and consider $\mathcal{S} := \mu^{-1}(\Phi^*\Phi - I)$.
We call $\mathcal{S} \in \mathbb{C}^{n \times n}$ a \textbf{signature matrix} since it satisfies all of the following:
\begin{itemize}
\item[(i)]
$\mathcal{S}_{ii} = 0$ for every $i \in [n]$,
\item[(ii)]
$| \mathcal{S}_{ij} | = 1$ whenever $i \neq j \in [n]$, and
\item[(iii)]
$\mathcal{S}^* = \mathcal{S}$.
\end{itemize}
Property (ii) says that each off-diagonal element of $\mathcal{S}$ is a \textbf{phase}, i.e., a member of the multiplicative group $\mathbb{T}:=\{ z \in \mathbb{C} : |z| = 1\}$.
(Similarly, the phase of a nonzero complex number $z \in \mathbb{C}^\times$ is defined as $z/|z| \in \mathbb{T}$.)
For any choice of phases $\{ \omega_j \}_{j \in [n]}$, the matrix $D := \operatorname{diag}(\omega_1,\dotsc,\omega_n)$ determines another signature matrix $\mathcal{S}' := D^{-1} \mathcal{S} D$.
In terms of $\mathscr{L}$, $\mathcal{S}'$ corresponds to the different choice of unit-norm representatives $\{ \omega_j \varphi_j \}_{j \in [n]}$.
We say that $\mathcal{S}$ and $\mathcal{S}'$ are \textbf{switching equivalent}, and write $\mathcal{S} \sim \mathcal{S}'$.
The \textbf{normalized signature matrix} of $\mathscr{L}$ is the unique $\mathcal{S}' \sim \mathcal{S}$ such that $\mathcal{S}'_{1j} =  \mathcal{S}'_{j1} = 1$ for every $j > 1$.
Conversely, any signature matrix $\mathcal{S}$ gives rise to a unique unitary equivalence class of linearly dependent equiangular lines.
Indeed, the least eigenvalue $\lambda$ of $\mathcal{S}$ is strictly negative since $\operatorname{tr} \mathcal{S} = 0$, and so $-\lambda^{-1} \mathcal{S} + I \geq 0$ can be factored as $\Phi^*\Phi$ for some unit vectors $\Phi$, which span linearly dependent equiangular lines having $\mathcal{S}$ as a signature matrix.
As a consequence of Proposition~\ref{prop:real lines are real}, a sequence of linearly dependent equiangular lines is real if and only if its normalized signature matrix is real-valued.
In~\cite{IversonM:18}, the latter notion was used as the definition of real lines.

The following result relates switching equivalence with unitary equivalence.
It follows easily from the fact that Gram matrices determine vector sequences up to unitary equivalence.

\begin{proposition}
\label{prop:switching equiv is unitary equiv}
Let $\mathcal{S},\mathcal{S'} \in \mathbb{C}^{n\times n}$ be signature matrices for equiangular line sequences $\mathscr{L}$ and $\mathscr{L}'$ in $\mathbb{C}^d$, and let $\sigma \in S_n$ be a permutation with matrix representation $P = [ \delta_{i,\sigma(j)} ]_{i,j \in [n]} \in \mathbb{C}^{n\times n}$.
Then $\mathcal{S}'$ is switching equivalent to $P \mathcal{S} P^{-1}$ if and only if $\sigma$ induces a unitary equivalence $\mathscr{L} \to \mathscr{L}'$.
In particular, $\mathcal{S}$ is switching equivalent to $P \mathcal{S} P^{-1}$ if and only if $\sigma \in \operatorname{Aut} \mathscr{L}$.
\end{proposition}

\begin{definition}
\label{def:autRC}
Let $\mathscr{L} = \{ \ell_j \}_{j\in [n]}$ be a sequence of lines in $\mathbb{R}^d$, and let $\mathscr{L}' = \{ \ell_j ' \}_{j\in [n]}$ be its embedding in $\mathbb{C}^d$.
We define $\operatorname{Aut}_{\mathbb{C}} \mathscr{L} := \operatorname{Aut} \mathscr{L}'$ and
\[ \operatorname{Aut}_{\mathbb{R}} \mathscr{L} := \{ \sigma \in S_n : \text{there exists } Q \in \operatorname{O}(d) \text{ such that } Q \ell_j = \ell_{\sigma(j)} \text{ for every } j \in [n] \}, \]
where $\operatorname{O}(d) \leq \mathbb{R}^{d\times d}$ is the orthogonal group.
\end{definition}

\begin{lemma}
\label{lem:orthogonal equiv}
Let $\mathscr{L} = \{ \ell_j \}_{j \in [n]}$ and $\mathscr{K} = \{ \kay_j \}_{j\in [n]}$ be sequences of lines in $\mathbb{R}^d$, and let $\mathscr{L}' = \{ \ell_j' \}_{j \in [n]}$ and $\mathscr{K}' = \{ \kay_j' \}_{j \in [n]}$ be their embeddings in $\mathbb{C}^d$.
Then the following are equivalent:
\begin{itemize}
\item[(a)]
there exists $Q \in \operatorname{O}(d)$ with $Q \ell_j = \kay_j$ for every $j \in [n]$,
\item[(b)]
there exists $U \in \operatorname{U}(d)$ with $U \ell_j' = \kay_j'$ for every $j \in [n]$.
\end{itemize}
In particular, $\operatorname{Aut}_{\mathbb{R}} \mathscr{L} = \operatorname{Aut}_{\mathbb{C}} \mathscr{L}$.
\end{lemma}

\begin{proof}
Choose unit-norm representatives $\varphi_j \in \ell_j \subset \mathbb{R}^d$ and $\psi_j \in \kay_j \subset \mathbb{R}^d$ for every $j \in [n]$.
Then (a) holds if and only if there exists $Q \in \operatorname{O}(d)$ and scalars $a_j \in \{ \pm 1 \}$ such that $Q \varphi_j = a_j \psi_j$ for every $j \in [n]$.
It is well known that vector sequences in $\mathbb{R}^d$ are determined up to orthogonal equivalence by their Gram matrices.
Hence, (a) holds if and only if there exist scalars $a_j \in \{ \pm 1 \}$ such that
\begin{equation}
\label{eq:orthequiv}
\langle \varphi_j, \varphi_i \rangle = a_i a_j \langle \psi_j, \psi_i \rangle \quad \text{for every }i,j \in [n].
\end{equation}
Similarly, (b) holds if and only if there are scalars $c_j \in \mathbb{T}$ such that 
\begin{equation}
\label{eq:unitequiv}
\langle \varphi_j, \varphi_i \rangle = \overline{c_i} c_j \langle \psi_j, \psi_i \rangle \quad \text{for every }i,j \in [n].
\end{equation}

It is now obvious that (a) implies (b).
For the converse, assume (b) holds and choose scalars $c_j \in \mathbb{T}$ to satisfy \eqref{eq:unitequiv}.
Define $\mathscr{G}$ to be the graph on the vertex set $[n]$ with an edge joining $i \neq j$ whenever $\langle \varphi_j, \varphi_i \rangle \neq 0$.
(This is the \emph{frame graph} of~\cite{Strawn:07,AN:18}.)
Equivalently, vertices $i \neq j$ are adjacent when $\langle \psi_j, \psi_i \rangle \neq 0$.

Choose representatives $i_1,\dotsc,i_r \in [n]$ for each of the connected components of $\mathscr{G}$, and define new scalars $\{ a_j \}_{j\in [n]}$ by setting $a_j = \overline{c_{i_k}} c_j$ whenever there is a path from $j$ to $i_k$.
Then $\overline{a_i} a_j = \overline{c_i} c_j$ whenever $i,j \in [n]$ lie in the same connected component of $\mathscr{G}$.
For every $i,j \in [n]$, \eqref{eq:unitequiv} gives $\langle \varphi_j, \varphi_i \rangle = \overline{a_i} a_j \langle \psi_j, \psi_i \rangle$, both sides of the equation being $0$ when there is no path from $i$ to $j$.

It remains to show that $a_j \in \{\pm 1\}$ for every $j \in [n]$.
Given adjacent vertices $i,j \in [n]$, we can divide the relation $\langle \varphi_j, \varphi_i \rangle = \overline{a_i} a_j \langle \psi_j, \psi_i \rangle$ to deduce that $\overline{a_i} a_j \in \mathbb{R} \cap \mathbb{T} =  \{ \pm 1\}$.
If one of $a_i,a_j$ belongs to $\{\pm 1\}$, the other does as well.
Since $a_{i_k} = \overline{c_{i_k}} c_{i_k} = 1$ for every $k \in [r]$, we quickly see that $a_j \in \{\pm 1\}$ for every $j \in [n]$.
This proves (a).

Finally, we conclude that $\operatorname{Aut}_{\mathbb{R}} \mathscr{L} = \operatorname{Aut}_{\mathbb{C}} \mathscr{L}$ by taking any $\sigma \in S_n$ and setting $\mathscr{K} = \{ \ell_{\sigma(j)} \}_{j \in [n]}$.
Then $\sigma \in \operatorname{Aut}_{\mathbb{R}} \mathscr{L}$ if and only if (a) holds, while $\sigma \in \operatorname{Aut}_{\mathbb{C}} \mathscr{L}$ if and only if (b) holds.
\end{proof}

\begin{definition}
A \textbf{two-graph} is a pair $(\Omega, \mathcal{T})$ where $\Omega$ is a finite set of vertices and $\mathcal{T}$ is a collection of 3-subsets of $\Omega$, such that every 4-subset of $\Omega$ contains an even number of 3-subsets from $\mathcal{T}$.
The \textbf{automorphism group} of $(\Omega,\mathcal{T})$ consists of all permutations of $\Omega$ that preserve $\mathcal{T}$.
\end{definition}

\begin{proposition}
\label{prop:real lines are two-graphs}
Given a sequence $\mathscr{L}$ of $n$ real, linearly dependent, equiangular lines in $\mathbb{C}^d$, choose unit-norm representatives $\{ \varphi_j \}_{j\in [n]}$ and put
\[ \mathcal{T}_{\mathscr{L}} := \Bigl\{ \{ i,j,k\} \subset [n] : i \neq j \neq k \neq i \text{ and } \mu^{-3} \langle \varphi_i, \varphi_j \rangle \langle \varphi_j, \varphi_k \rangle \langle \varphi_k, \varphi_i \rangle = -1 \Bigr\}, \]
where $\mu := | \langle \varphi_i, \varphi_j \rangle |$ for $i \neq j$.
Then $([n], \mathcal{T}_{\mathscr{L}})$ is a two-graph whose automorphism group equals $\operatorname{Aut}(\mathscr{L})$.
Moreover, the mapping $\mathscr{L} \mapsto \mathcal{T}_{\mathscr{L}}$ induces a bijection between unitary equivalence classes of real, linearly dependent, equiangular lines and isomorphism equivalence classes of two-graphs.
In this correspondence, doubly transitive real lines map to doubly transitive two-graphs, and vice versa.
\end{proposition}

\begin{proof}
For the first part, we may apply a global unitary to assume $\mathscr{L}$ is a sequence of lines in real space.
Then it is well known (cf.\ \cite{Seidel:91}) that $T_{\mathscr{L}}$ is a two-graph, and (cf.\ \cite{Cameron:77}) that $\operatorname{Aut} T_{\mathscr{L}} = \operatorname{Aut}_{\mathbb{R}} \mathscr{L}$, which coincides with $\operatorname{Aut}_{\mathbb{C}} \mathscr{L}$ by Lemma~\ref{lem:orthogonal equiv}.
In particular, $\mathscr{L}$ is doubly transitive if and only if $\mathcal{T}_{\mathscr{L}}$ is, too.

By Lemma~\ref{lem:orthogonal equiv}, the embedding $\mathbb{R}^d \subset \mathbb{C}^d$ induces a bijection between linearly dependent equiangular line sequences in $\mathbb{R}^d$ up to global orthogonal transformation, and linearly dependent equiangular real line sequences in $\mathbb{C}^d$ up to global unitary.
It is also well known (cf.\ \cite{Seidel:91} again) that the mapping $\mathscr{L} \to \mathcal{T}_{\mathscr{L}}$ gives a bijection between linearly independent equiangular line sequences in real space up to global orthogonal transformation, and two-graphs on vertex set $[n]$.
The proof is completed by composing these bijections.
\end{proof}

With each two-graph $(\mathcal{T},[n])$ we associate a normalized signature matrix $\mathcal{S} \in \mathbb{C}^{n\times n}$, defined as follows.
Set $\mathcal{S}_{ii} = 0$ for every $i \in [n]$, and $\mathcal{S}_{i1} = \mathcal{S}_{1i} = 1$ for every $i \neq 1$.
When $i \neq j$ are both different from $1$, put $\mathcal{S}_{ij} = -1$ if $\{ 1 ,i ,j \} \in \mathcal{T}$ and $\mathcal{S}_{ij} = 1$ otherwise.
Then $\mathcal{S}$ is the normalized signature matrix of a sequence $\mathscr{L}$ of linearly dependent, real, equiangular lines such that $\mathcal{T} = \mathcal{T}_{\mathscr{L}}$.
The two-graph $(\mathcal{T},[n])$ is called \textbf{regular} if $\mathcal{S}$ has exactly two eigenvalues $\lambda_1 > 0 > \lambda_2$.
In that case, $\mathscr{L}$ spans a space of dimension $d = \frac{n\lambda_1}{\lambda_1 - \lambda_2}$, while $n-d = \frac{-n\lambda_2}{\lambda_1 - \lambda_2}$.
Every doubly transitive two-graph is regular, as an easy consequence of Theorem~2.2 in~\cite{Taylor:77}.

In light of Proposition~\ref{prop:real lines are two-graphs}, doubly transitive real lines are either linearly independent or else they follow Taylor's classification of doubly transitive two-graphs~\cite{Taylor:92}.
We summarize the classification of doubly transitive real lines in Proposition~\ref{prop:real2tran} below.
Information about the automorphism groups is taken from \cite[Theorem~1]{Taylor:92}, while existence and uniqueness are provided by \cite[Theorem~2]{Taylor:92}.
For each of the cases (i)--(viii) below, a corresponding two-graph is described in~\cite[Section~6]{Taylor:77}.
There one also finds eigenvalues for a normalized signature matrix.
We used those eigenvalues to compute the dimensions~$d$ below.
Finally, we remark that \cite[Theorem~1]{Taylor:92} lists additional symmetry groups of affine type (I).
However, each of those groups is contained in an affine symplectic group, and~\cite[Section~3]{Taylor:92} establishes that the corresponding two-graphs are instances of case~(i) below.

\begin{proposition}
\label{prop:real2tran}
If $\mathscr{L}$ is a sequence of $n\geq 2d > 2$ real lines with span $\mathbb{C}^d$ for which $S:=\operatorname{Aut} \mathscr{L}$ is doubly transitive, then one of the following holds:
\begin{itemize}
\item[(I)]
There is a normal elementary abelian subgroup of $S$ that acts regularly on $\mathscr{L}$.
Furthermore, $S$ has a subgroup $G$, where
\begin{itemize}
\item[(i)]
$d = 2^{m-1}(2^m - 1)$, $n = 2^{2m}$, $G = \mathbb{F}_2^{2m} \rtimes \operatorname{Sp}(2m,2)$, $m \neq 1$.
\end{itemize}
\item[(II)]
There is a nonabelian simple group $G$ such that $G\unlhd S\leq\operatorname{Aut}(G)$, where $(d,n,G)$ is one of the following:
\begin{itemize}
\item[(ii)]
$d=(q+1)/2$, $n=q+1$, $G=\operatorname{PSL}(2,q)$, $q$ is a prime power, $q\equiv1\bmod4$.
\item[(iii)]
$d=q^2-q+1$, $n=q^3+1$, $G=\operatorname{PSU}(3,q)$, $q$ is an odd prime power.
\item[(iv)]
$d=q^2-q+1$, $n=q^3+1$, $G={}^2G_2(q)$, $q=3^{2e+1}$.
\item[(v)]
$d=(2^{2m-1}-3\cdot2^{m-1}+1)/3$, $n=2^{2m-1}-2^{m-1}$, $G=\operatorname{Sp}(2m,2)$, $m\geq3$.
\item[(vi)]
$d=(2^{2m-1}+3\cdot2^{m-1}+1)/3$, $n=2^{2m-1}+2^{m-1}$, $G=\operatorname{Sp}(2m,2)$, $m\geq3$.
\item[(vii)]
$d=22$, $n=176$, $G=\operatorname{HS}$ (the Higman--Sims group).
\item[(viii)]
$d=23$, $n=276$, $G=\operatorname{Co}_3$ (the third Conway group).
\end{itemize}
\end{itemize}
Conversely, for each $(n,G)$ of (i)--(viii) there exists a unique dimension $d'$ and a unique unitary equivalence class of $n \geq 2d' > 2$ real lines spanning $\mathbb{C}^{d'}$ whose automorphism group contains $G$.
In particular, $d' = d$ is the dimension listed above.
\end{proposition}

\section{Review of Higman pairs and roux}

In $\mathbb{R}^d$, every sequence of real doubly transitive lines can be viewed as a two-graph.
Similarly in the complex case, we will show that doubly transitive lines derive from a combinatorial object known as a \textbf{roux}.
The authors introduced roux and the related notion of Higman pairs in the first paper of this series~\cite{IversonM:18}.
In this section, we recall some necessary results from~\cite{IversonM:18} and provide a brief summary of association schemes, roux, and Higman pairs.

\subsection{Association schemes}

Recall that a \textbf{$*$-algebra} is a complex algebra $\mathscr{A}$ equipped with an involution $A \mapsto A^*$ such that $(cA+B)^* =\overline{c}A^* + B^*$ and $(AB)^* = B^* A^*$ for every $A,B \in \mathscr{A}$ and $c \in \mathbb{C}$.
An \textbf{association scheme} is a sequence $\{A_i\}_{i\in[k]}$ in $\mathbb{C}^{n\times n}$ with entries in $\{0,1\}$ such that
\begin{itemize}
\item[(A1)]
$A_1=I$,
\item[(A2)]
$\sum_{i\in[k]}A_i=J$ (the matrix of all ones), and
\item[(A3)]
$\mathscr{A}:=\operatorname{span}\{A_i\}_{i\in[k]}$ is a $*$-algebra under matrix multiplication.
\end{itemize}
We refer to $\mathscr{A}$ as the \textbf{adjacency algebra} (or \textit{Bose--Mesner algebra}) of $\{ A_i \}_{i \in [k]}$.
Any matrix $M \in \mathscr{A}$ has a unique expansion $M = \sum_{i \in [k]} c_i A_i$, and $M$ is said to \textbf{carry the scheme} when the coefficients $c_i$ are distinct.
(Then the ``level sets'' of $M$ carve out the adjacency matrices $\{A_i\}_{i\in [k]}$.)
An association scheme is said to be \textbf{commutative} when its adjacency algebra is commutative.
In that case, the spectral theorem provides $\mathscr{A}$ with an alternative basis of orthogonal projections $\{ E_j \}_{j\in [k]}$.
These are called the \textbf{primitive idempotents} of $\mathscr{A}$, since every other projection in $\mathscr{A}$ can be expressed as a sum of primitive idempotents.
Whether or not the scheme is commutative, every orthogonal projection $\mathcal{G} \in \mathscr{A}$ can be factored as $\mathcal{G} = \Phi^* \Phi$ where $\Phi = \{ \varphi_j \}_{j\in [n]}$ is a sequence of equal-norm vectors $ \varphi_j \in \mathbb{C}^d$, $d = \operatorname{rank} \mathcal{G}$.
By collecting the lines spanned by these vectors and eliminating duplicates as necessary, we obtain a sequence $\mathscr{L} = \{ \ell_j \}_{j\in [m]}$ of lines spanning $\mathbb{C}^d$ having representatives $\Phi$ with Gram matrix $\mathcal{G}$.
(Notice that each line may be represented several times in $\Phi$ and in $\mathcal{G}$.)
The resulting lines have been studied in \cite{DelsarteGS:77,IversonJM:19,IversonJM:SPIE:17,IversonJM:17}.

A large class of examples of association schemes arise from permutation groups, as follows.
Let $G$ be a finite group acting transitively on a set $X$ (from the left).
Then $G$ also acts on $X \times X$ through the diagonal action $g\cdot (x,y) := (g\cdot x, g\cdot y)$.
Collect the orbits $\{ R_i \}_{i \in [k]}$ of $G$ on $X \times X$, and let $\{ A_i \}_{i \in [k]}$ be the corresponding indicator matrices in $\mathbb{C}^{X \times X}$, i.e.\ $[A_i]_{x,y} = 1$ if $(x,y) \in R_i$ and $0$ otherwise.
By reindexing if necessary, we may assume that $A_1 = I$.
Then $\{ A_i \}_{i\in [k]}$ is an association scheme, called a \textbf{Schurian} scheme.
The corresponding adjacency algebra $\mathscr{A}$ consists of all $G$-stable matrices, that is, matrices $M \in \mathbb{C}^{X\times X}$ satisfying $M_{g\cdot x, g\cdot y} = M_{x,y}$ for every $g \in G$ and $x,y \in X$.
Any vector sequence $\Phi = \{ \varphi_x \}_{x \in X}$ with $\Phi^* \Phi \in \mathscr{A}$ is called a \textbf{homogeneous frame}, since it inherits symmetries corresponding to the action of $G$ on $X$ (see~\cite{IversonJM:17}).
For an alternative description of the scheme, we may assume without loss of generality that $X = G/H$, where $H \leq G$ is the stabilizer of a point in $X$.
Then the double cosets $\{ H a_i H \}_{i \in [k]}$ of $H$ in $G$ can be indexed in such a way that $[A_i]_{xH,yH} = 1$ if $y^{-1}x \in Ha_i H$ and $0$ otherwise.
We refer to $(G,H)$ as a \textbf{Gelfand pair} when its Schurian scheme is commutative.

\subsection{Roux and roux lines}

Given a finite abelian group $\Gamma$ of order $r$ and a positive integer $n$, we denote $\mathbb{C}[\Gamma]$ for its group algebra over $\mathbb{C}$, and $\lceil \cdot \rfloor \colon \mathbb{C}[\Gamma]^{n\times n} \to \mathbb{C}^{rn \times rn}$ for the injective $\ast$-algebra homomorphism that replaces each element of $\Gamma$ with its $r\times r$ Cayley representation, extended linearly to $\mathbb{C}[\Gamma]$.

\begin{definition}
Given an abelian group $\Gamma$ and a positive integer $n$, an $n\times n$ \textbf{roux} for $\Gamma$ is a matrix $B$ with entries in $\mathbb{C}[\Gamma]$ such that the following hold:
\begin{itemize}
\item[(R1)]
$B_{ii}=0$ for every $i\in[n]$.
\item[(R2)]
$B_{ij}\in\Gamma$ for every $i,j\in[n]$, $i\neq j$.
\item[(R3)]
$B_{ji}=(B_{ij})^{-1}$ for every $i,j\in[n]$, $i\neq j$.
\item[(R4)]
The vector space $\mathscr{A}(B) := \operatorname{span} (\{ gI : g \in \Gamma\} \cup \{ gB : g \in \Gamma \}) \leq \mathbb{C}[\Gamma]^{n\times n}$ is an algebra, i.e., it is closed under matrix multiplication.
\end{itemize}
\end{definition}

Every roux $B \in \mathbb{C}[\Gamma]^{n\times n}$ determines an association scheme (called the \textbf{roux scheme}) with adjacency matrices $\{ \lceil g I \rfloor \}_{g \in \Gamma}$ and $\{ \lceil gB \rfloor \}_{g \in \Gamma}$.

\begin{proposition}[Lemma~2.3 in~\cite{IversonM:18}]
\label{prop.B squared}
Suppose $B\in\mathbb{C}[\Gamma]^{n\times n}$ satisfies (R1)--(R3).
Then $B$ is a roux for $\Gamma$ if and only if
\[
B^2=(n-1)I+\sum_{g\in\Gamma}c_ggB
\]
for some complex numbers $\{c_g\}_{g\in\Gamma}$, called the \textbf{roux parameters} for $B$.
In this case, we necessarily have that $\{c_g\}_{g\in\Gamma}$ are nonnegative integers that sum to $n-2$, with $c_{g^{-1}}=c_g$ for every $g\in\Gamma$.
\end{proposition}

\begin{proposition}[Theorem~2.8 in~\cite{IversonM:18}]
\label{prop.roux primitive idempotents}
Given an $n\times n$ roux $B$ for $\Gamma$ with parameters $\{ c_g \}_{g \in \Gamma}$, the primitive idempotents for the corresponding roux scheme are scalar multiples of
\[
\mathcal{G}_{\alpha}^\epsilon
:=\sum_{g\in\Gamma}\alpha(g)\lceil gI\rfloor+\mu_\alpha^\epsilon\sum_{g\in\Gamma}\alpha(g) \lceil gB \rfloor,
\qquad
\qquad
(\alpha\in\hat\Gamma,~\epsilon\in\{+,-\}),
\]
where $\mu_\alpha^\epsilon$ is defined in terms of the Fourier transform $\hat{c}_\alpha:=\sum_{h\in\Gamma}c_h\overline{\alpha(h)}$ as follows:
\[
\mu_\alpha^\epsilon
:=\frac{\hat{c}_\alpha+\epsilon\sqrt{(\hat{c}_\alpha)^2+4(n-1)}}{2(n-1)}.
\]
The rank of $\mathcal{G}_\alpha^\epsilon$ is
\[
d_\alpha^\epsilon :=\frac{n}{1+(n-1)(\mu_\alpha^\epsilon)^2}.
\]

Furthermore, if $\alpha \in \hat{\Gamma}$, $\mu > 0$, and 
\[ \mathcal{G} := \sum_{g\in\Gamma}\alpha(g)\lceil gI\rfloor+\mu \sum_{g\in\Gamma}\alpha(g) \lceil gB \rfloor \]
satisfies $\mathcal{G}^2 = c \mathcal{G}$ for some $c >0$, then $\mu = \mu_\alpha^\epsilon$ for some $\epsilon \in \{ +, - \}$, and a scalar multiple of $\mathcal{G} = \mathcal{G}_\alpha^\epsilon$ is a primitive idempotent for the roux scheme of $B$.
\end{proposition}

Given a finite abelian group $\Gamma$, we write $\hat{\Gamma}$ for the Pontryagin dual group of characters $\alpha \colon \Gamma \to \mathbb{T}$.
Every $\alpha \in \hat{\Gamma}$ extends by linearity to a $\ast$-algebra homomorphism $\overline{\alpha} \colon \mathbb{C}[\Gamma] \to \mathbb{C}$, which in turn extends to a $\ast$-algebra homomorphism $\hat{\alpha} \colon \mathbb{C}[\Gamma]^{n\times n} \to \mathbb{C}^{n\times n}$ through entrywise application of $\overline{\alpha}$.

\begin{proposition}[Theorem~3.1 in~\cite{IversonM:18}]
\label{prop.roux signature}
Suppose $B\in\mathbb{C}[\Gamma]^{n\times n}$ satisfies (R1)--(R3).
Then $B$ is a roux if and only if for every $\alpha\in\hat{\Gamma}$, $\hat\alpha(B)$ is the signature matrix of an equiangular tight frame.
\end{proposition}

Any sequence of lines arising from an application of Proposition~\ref{prop.roux signature} are called \textbf{roux lines}.
We will see that every sequence of doubly transitive lines is roux.
While it will not play a role in the sequel, we also remark that every \textit{regular abelian distancer-regular antipodal cover of the complete graph} (\textsc{drackn}) can be viewed as an instance of a roux, as explained in~\cite{IversonM:18}, and the resulting roux lines coincide with the construction of ETFs from regular abelian \textsc{drackn}s in \cite{GodsilHensel:92,CoutinhoGSZ:16}. 

We now have two constructions of lines from roux: one set from from the primitive idempotents in Proposition~\ref{prop.roux primitive idempotents}, and another from Proposition~\ref{prop.roux signature}.
However, the resulting lines are identical.
Indeed, for any choice of $\alpha \in \hat{\Gamma}$ and $\epsilon \in \{1,-1\}$, it holds that $[\mathcal{G}_{\alpha^{-1}}^\epsilon]_{(i,1),(j,1)} = \mu_{\alpha}^\epsilon \alpha(B_{ij})=\epsilon \alpha(B_{ij}) |\mu_{\alpha}^\epsilon|$ whenever $i \neq j$.
This yields the following version of Lemma~3.2 in~\cite{IversonM:18}.

\begin{proposition}
\label{prop.two reps of same lines}
For each $\alpha\in\hat\Gamma$ and $\epsilon \in \{1,-1\}$, the signature matrix $\epsilon \hat{\alpha}(B)$ from Proposition~\ref{prop.roux signature} and the Gram matrix $\mathcal{G}_{\alpha^{-1}}^\epsilon$ from Proposition~\ref{prop.roux primitive idempotents} describe the same lines (each line implicated by the former is represented $|\Gamma|$ times in the latter).
\end{proposition}

One can easily determine whether or not roux lines are real with the following result.

\begin{proposition}[Real roux lines detector, Corollary~3.8 in~\cite{IversonM:18}]
\label{prop.real line detector}
Let $B$ be a roux for $\Gamma$ and pick $\alpha\in\hat\Gamma$.
Then $\hat\alpha(B)$ is a signature matrix of real lines if and only if $\alpha(g)$ is real for every $g\in\Gamma$ such that $c_g\neq0$.
\end{proposition}

Finally, there are several trivial ways to deform roux, one of which is described below.

\begin{proposition}[Lemma~2.5(a) in~\cite{IversonM:18}]
\label{prop.switching equivalent roux}
Let $B \in \mathbb{C}[\Gamma]^{n\times n}$ be a roux with parameters $\{c_g \}_{g \in \Gamma}$.
Given any diagonal matrix $D \in \mathbb{C}[\Gamma]^{n\times n}$ with $D_{ii} \in \Gamma$ for every $i \in [n]$, then $DBD^{-1}$ is a roux with parameters $\{c_g\}_{g\in \Gamma}$.
\end{proposition}

The roux $B$ and $DBD^{-1}$ in Proposition~\ref{prop.switching equivalent roux} are called \textbf{switching equivalent}.
Given any such $B$ and $D$ and any $\alpha \in \widehat{\Gamma}$, it is easy to see that the signature matrices $\hat{\alpha}(B)$ and $\hat{\alpha}(DBD^{-1})$ are switching equivalent.
Consequently, switching equivalent roux create the same roux lines.

\subsection{Schurian roux schemes and Higman pairs}

\begin{definition}
Given a finite group $G$ and a proper subgroup $H\leq G$, let $K:=N_G(H)$ be the normalizer of $H$ in $G$.
We say $(G,H)$ is a \textbf{Higman pair} if there exists a \textbf{key} $b\in G\setminus K$ such that
\begin{itemize}
\item[(H1)]
$G$ acts doubly transitively on $G/K$,
\item[(H2)]
$K/H$ is abelian,
\item[(H3)]
$HbH=Hb^{-1}H$,
\item[(H4)]
$aba^{-1}\in HbH$ for every $a\in K$, and
\item[(H5)]
$a \in K$ satisfies $ab \in HbH$ only if $a \in H$.
\end{itemize}
\end{definition}

The relationship between roux and Higman pairs is given below and depicted in Figure~\ref{fig:roux}.

\begin{proposition}[Theorem~2.1 in~\cite{IversonM:18}]
\label{prop.schurian roux}
Let $G$ be a finite group, and pick $H\leq G$.
The Schurian scheme of $(G,H)$ is isomorphic to a roux scheme if and only if $(G,H)$ is a Higman pair.
\end{proposition}

\begin{proposition}[Lemma~2.6 in~\cite{IversonM:18}]
\label{prop.double cosets}
Given a Higman pair $(G,H)$, denote $K:=N_G(H)$, $n:=[G:K]$ and $r:=[K:H]$, and select any key $b\in G\setminus K$.
Then the following hold:
\begin{itemize}
\item[(a)]
$H$ has $2r$ double cosets in $G$: $r$ of the form $aH$, and $r$ of the form $HabH$ for some $a\in K$;
\item[(b)]
for every $a\in K$, we have $HabH=HbaH$; and
\item[(c)]
for every $a\in K$, we have $|HabH|=(n-1)|H|$.
\end{itemize}
\end{proposition}

\begin{proposition}[Roux from Higman pairs, Lemma~2.7 in~\cite{IversonM:18}]
\label{prop.higman's roux}
Given a Higman pair $(G,H)$, denote $K:=N_G(H)$ and $n:=[G:K]$, and select any key $b\in G\setminus K$.
Choose left coset representatives $\{x_j\}_{j\in[n]}$ for $K$ in $G$, and choose coset representatives $\{a_g\}_{g\in K/H}$ for $H$ in $K$.
Define $B\in\mathbb{C}[K/H]^{n\times n}$ entrywise as follows:
Given $i\neq j$, let $B_{ij}$ be the unique $g\in K/H$ for which $x_i^{-1}x_j\in Ha_gbH$, and set $B_{ii}=0$.
Then $B$ is a roux for $K/H$ with roux parameters $\{c_g\}_{g\in K/H}$ given by
\begin{equation}
\label{eq.higman's roux}
c_g
=\frac{n-1}{|H|}\cdot|bHb^{-1}\cap Ha_gbH|.
\end{equation}
Furthermore, the roux scheme generated by $B$ is isomorphic to the Schurian scheme of $(G,H)$.
\end{proposition}

We emphasize that a Higman pair $(G,H)$ does not, in general, carry a unique key, and distinct keys may create distinct roux through Proposition~\ref{prop.higman's roux}.
However, for any choice of key the resulting roux scheme is isomorphic to the Schurian scheme of $(G,H)$.
Since roux lines correspond with primitive idempotents in the roux scheme, their ranks do not depend on the choice of key.
In particular, the ranks of the primitive idempotents for the Schurian scheme of $(G,H)$ can be computed using~\eqref{eq.higman's roux} and Proposition~\ref{prop.roux primitive idempotents}.

\begin{figure}
\begin{center}
\begin{tabular}{ll}
\begin{tikzpicture}[scale=0.8]
\coordinate (0) at (0,0);
\coordinate (1) at (-1.5,0);
\coordinate (2) at (1.5,0);
\coordinate (3) at (0,3.23);
\coordinate (4) at (2.5,0.23);
\coordinate (5) at (-2.5,0.23);
\coordinate (6) at (-2.5,-0.23);
\coordinate (7) at (0,0.23);
\coordinate (8) at (0,-0.23);
\coordinate (9) at (2.5,-0.2);
\draw [thick] (-4.5,-3) -- (-4.5,4) -- (4.5,4) -- (4.5,-3) -- (-4.5,-3);
\draw [thick] (1) circle [radius = 2.5];
\draw [thick] (2) circle [radius = 2.5];
\node at (3){\small{association schemes}};
\node at (4){\small{roux}};
\node at (9){\small{schemes}};
\node at (5){\small{Schurian}};
\node at (6){\small{schemes}};
\node at (7){\small{Higman}};
\node at (8){\small{pairs}};
\end{tikzpicture}
\qquad
\qquad
\begin{tikzpicture}[scale=0.8]
\coordinate (0) at (0,0);
\coordinate (1) at (-1.5,0);
\coordinate (2) at (1.5,0);
\coordinate (3) at (0,3.23);
\coordinate (4) at (2.5,0.23);
\coordinate (5) at (-2.5,0.23);
\coordinate (6) at (-2.5,-0.23);
\coordinate (7) at (0,0.46);
\coordinate (8) at (0,0);
\coordinate (9) at (0,3.1);
\coordinate (10) at (0,-0.46);
\coordinate (11) at (2.5,-0.23);
\draw [thick] (-4.5,-3) -- (-4.5,4) -- (4.5,4) -- (4.5,-3) -- (-4.5,-3);
\draw [thick] (1) circle [radius = 2.5];
\draw [thick] (2) circle [radius = 2.5];
\node at (3){\small{projections in adjacency algebras}};
\node at (4){\small{roux}};
\node at (11){\small{lines}};
\node at (5){\small{homogeneous}};
\node at (6){\small{frames}};
\node at (7){\small{doubly}};
\node at (8){\small{transitive}};
\node at (10){\small{lines}};
\end{tikzpicture}
\end{tabular}
\end{center}
\caption{
\label{figure.venn}
{\small 
(left) The relationship between roux and Higman pairs, as described by Proposition~\ref{prop.schurian roux}.
(right) The dual relationships for lines, as given by Theorem~\ref{thm.Higman Pair Theorem}.
}\normalsize}
\label{fig:roux}
\end{figure}

\section{Radicalization and the Higman Pair Theorem}

As mentioned in the introduction, every sequence of linearly dependent doubly transitive lines is determined by a Higman pair.
This fact will enable us to classify doubly transitive lines that exhibit almost simple symmetries.
The following result makes this pivotal correspondence explicit.

\begin{theorem}[Higman Pair Theorem, Theorem~1.3 in~\cite{IversonM:18}]\
\label{thm.Higman Pair Theorem}
\begin{itemize}
\item[(a)]
Assume $n\geq 3$.
Given $n>d$ doubly transitive lines that span $\mathbb{C}^d$, there exists $r$ such that one may select $r$ equal-norm representatives from each of the $n$ lines that together carry the association scheme of a Higman pair $(G,H)$ with $r=[N_G(H):H]$ and $n=[G:N_G(H)]$.
Moreover, their Gram matrix is a primitive idempotent for this scheme.
\item[(b)]
Every Higman pair $(G,H)$ is a Gelfand pair.
Every primitive idempotent of its association scheme is the Gram matrix of $r:=[N_G(H):H]$ equal-norm representatives from each of $n:=[G:N_G(H)]$ doubly transitive lines that span $\mathbb{C}^d$ with $d<n$, and the phase of each entry is an $r$th root of unity.
Moreover, the automorphism group of the lines contains the doubly transitive action of $G$ on $G/N_G(H)$.
\end{itemize}
\end{theorem}

The proof of Theorem~\ref{thm.Higman Pair Theorem}(b) appeared in \cite{IversonM:18}, but the proof of Theorem~\ref{thm.Higman Pair Theorem}(a) was saved for the current installment.
Our proof technique relies on the following notions.
A \textbf{projective unitary representation} of a finite group $G$ is a mapping $\rho \colon G \to \operatorname{U}(d)$ such that $\rho(1) = I$ and such that there exists $f \colon G \times G \to \mathbb{T}$ satisfying $\rho(x)\rho(y) = f(x,y) \rho(xy)$ for every $x,y \in G$.
If $\rho\colon G \to \operatorname{U}(d)$ is a homomorphism (i.e.$f \equiv 1$), then we call $\rho$ an (honest) \textbf{unitary representation}.
The function $f$ may be formalized as follows.

\begin{definition}
Let $G$ be a finite group.
A \textbf{2-cocycle} is a function $f \colon G \times G \to \mathbb{C}^\times$ such that $f(x,y) f(xy,z) = f(y,z) f(x,yz)$ for all $x,y,z \in G$.
A \textbf{2-coboundary} is a function $f \colon G \times G \to \mathbb{C}^\times$ for which there exists $t \colon G \to \mathbb{C}^\times$ such that $f(x,y) = t(x)t(y)t(xy)^{-1}$ for all $x,y \in G$.
The \textbf{Schur multiplier} of $G$ is the abelian group $M(G)$ of 2-cocycles under pointwise multiplication, modulo the subgroup of 2-coboundaries. 
A \textbf{Schur cover} of $G$ is a group $G^*$ with a subgroup $A$ such that:
\begin{itemize}
\item[(i)] the center of $G^*$ and the commutator subgroup of $G^*$ both contain $A$,
\item[(ii)] $A \cong M(G)$, and
\item[(iii)] there is an epimorphism $\pi \colon G^* \to G$ (a \textbf{Schur covering}) with kernel $A$.
\end{itemize}
\end{definition}

The Schur multiplier of a finite group is finite, and every finite group has a (necessarily finite) Schur cover.
We refer to \cite{Karpilovsky:85,Karpilovsky:87} for background.
The following result explains how Schur covers help convert projective unitary representations into honest unitary representations.

\begin{proposition}
\label{prop: make honest}
Let $G$ be a finite group with a projective unitary representation $\rho \colon G \to \operatorname{U}(d)$.
Given any Schur covering $\pi \colon G^* \to G$, there exists a unitary representation $\rho^* \colon G^* \to \operatorname{U}(d)$ and phases $\{ \omega_{x^*} \}_{x^* \in G^*}$ such that $\rho^*(x^*) = \omega_{x^*} \rho(\pi(x^*))$ for every $x^* \in G^*$.
\end{proposition}

\begin{proof}
A theorem of Schur~\cite[Lemma~3.1]{Karpilovsky:85} provides the existence of a representation $\rho^* \colon G^* \to \operatorname{GL}(d,
\mathbb{C})$ and an ensemble of complex scalars $\{ \omega_{x^*} \}_{x^* \in G^*}$ satisfying $\rho^*(x^*) = \omega_{x^*} \rho( \pi( x^* ) )$ for every $x^* \in G^*$.
In order to verify that $\rho^*$ takes its image in $\operatorname{U}(d) \leq \operatorname{GL}(d,\mathbb{C})$, we need only check that each scalar $\omega_{x^*}$ is unimodular.
This can be seen by considering the eigenvalues of the diagonalizable operators $\rho( \pi( x^*) )$ and $\rho^* ( x^* )$.
The first has unimodular eigenvalues because it is unitary; the second has unimodular eigenvalues since its order is finite.
As $\rho^*(x^*) = \omega_{x^*} \rho( \pi( x^* ) )$, we conclude that $\omega_{x^*}$ is a phase.
\end{proof}

\begin{lemma}
\label{lem:proj rep}
Let $G$ be a doubly transitive group of automorphisms of a sequence $\mathscr{L} = \{ \ell_i \}_{i\in [n]}$ of $n > d$ lines spanning $\mathbb{C}^d$.
For each $\sigma \in G$, choose a unitary $\rho(\sigma) \in \operatorname{U}(d)$ such that $\rho(1) = I$ and $\rho(\sigma)\ell_i = \ell_{\sigma(i)}$ for every $i \in [n]$.
Then $\rho\colon G \to \operatorname{U}(d)$ is a projective unitary representation.
\end{lemma}

\begin{proof}
Choose unit-norm representatives $\Phi = \{ \varphi_i \}_{i \in [n]}$ for $\mathscr{L}$.
Then $\Phi$ is an equiangular tight frame, by Proposition~\ref{prop: 2tran implies ETF}.
Since $n > d$, the common value of $| \langle \varphi_i, \varphi_j \rangle|$ for $i \neq j$ is not zero.
The desired result now follows immediately from~\cite[Lemma~6.5]{ChienWaldron:18}.
\end{proof}

\begin{figure}
\begin{center}

\begin{tabular}{r|cccccccccc}
\multicolumn{1}{r|}{} & \multicolumn{2}{c}{Group} && \multicolumn{2}{c}{Stabilizer} && \multicolumn{1}{c}{Character}&& \multicolumn{2}{c}{Kernel} \\ \hline
Radicalization & $\widetilde{G}^*$ & $=G^* \times C_r$ && $\widetilde{G}^*_1$ &  $=G^*_1 \times C_r$ && $\widetilde{\alpha} \colon \widetilde{G}_1^* \twoheadrightarrow C_r$ && $H$ & $=\ker \widetilde{\alpha}$ \rule{0 pt}{15 pt}  \\
& $\hookuparrow$ &&& $\hookuparrow$ &&& $\dashuparrow$ extends \\
Schur cover & $G^*$ &&& $G^*_1$ & $=\pi^{-1}(G_1)$ && $\alpha \colon G_1^* \twoheadrightarrow C_{r'}$ \\
& $\pi \twoheaddownarrow \phantom{\pi}$ &&& $\twoheaddownarrow$ & \\
Permutation group & $G$ &&& $G_1$
\end{tabular}
\end{center}

\caption{Overview of notation in Example~\ref{ex:Higman pair from 2x4}, the proof of Theorem~\ref{thm.radicalization}, and Section~\ref{sec.roux tech}.
Here, $\pi\colon G^* \to G$ is a Schur covering, $r = 2r'$, and $(\widetilde{G}^*,H)$ is the \textit{radicalization} of $(G^*,G_1^*,\alpha)$ from Definition~\ref{defn:radicalization}.
}
\label{fig:notation}
\end{figure}

The following example demonstrates the fundamental connection between doubly transitive lines and Higman pairs.

\begin{example}
\label{ex:Higman pair from 2x4}
We produce a Higman pair $(\tilde{G}^*,H)$ from four doubly transitive lines in $\mathbb{C}^2$.
Put $\omega := e^{\pi \mathrm{i}/4}$, $a := \sqrt{2 - \sqrt{3}}$, $b:=\sqrt{1+a^2} = \sqrt{3-\sqrt{3}}$, and consider the lines $\mathscr{L} = \{\ell_1,\ell_2,\ell_3,\ell_4\}$ spanned by columns of
\[
\Phi := \tfrac{1}{b}\left[ \begin{array}{rrrr}
\omega & \omega & a & a \\
a & -a & \omega & - \omega
\end{array} \right].
\]
It is easy to check that $\Phi$ is an equiangular tight frame, and it is well known that $\operatorname{Aut} \mathscr{L} = A_4 =: G$.
Indeed, the permutation $\sigma_1 := (1\, 2)(3\, 4)$ is performed by the unitary $\rho(\sigma_1):=\left[ \begin{smallmatrix} 1 & 0 \\ 0 & -1 \end{smallmatrix} \right]$, and the permutation $\sigma_2 := (2\, 3\, 4)$ is performed by the unitary $\rho(\sigma_2):=\tfrac{1}{\sqrt{2}} \left[ \begin{smallmatrix} 1 & 1 \\ -\mathrm{i} & \mathrm{i} \end{smallmatrix} \right]$.
Consequently, $\operatorname{Aut} \mathscr{L} \geq \langle \sigma_1,\sigma_2 \rangle = A_4$.
Equality holds since there do not exist $n=4$ lines in $d=2$ dimensions that have triply transitive automorphism group $S_4$; see Lemma~\ref{lem.3tran}.
By Lemma~\ref{lem:proj rep}, we may extend $\rho$ to a projective unitary representation $\rho \colon G \to \operatorname{U}(2)$ with the property that $\rho(\sigma) \ell_i = \ell_{\sigma(i)}$ for each $\sigma \in G$ and $i \in [4]$; in fact, the extension is unique up to a choice of unimodular constants.
We desire an honest unitary representation, and so we consider the Schur cover $G^*:= \operatorname{SL}(2,3)$.
Denoting $x_1:= \left[ \begin{smallmatrix} 0 & 1 \\ -1 & 0 \end{smallmatrix} \right]$ and $x_2 := \left[ \begin{smallmatrix} 2 & 2 \\ 0 & 2 \end{smallmatrix} \right]$, we have that $G^*=\langle x_1,x_2\rangle$, and a Schur covering $\pi \colon G^* \to G$ is given by $\pi(x_1) = \sigma_1$ and $\pi(x_2) = \sigma_2$.
As in Proposition~\ref{prop: make honest}, the data $\rho^*(x_1):=\mathrm{i} \rho(\sigma_1)$ and $\rho^*(x_2):= -e^{\pi \mathrm{i}/12} \rho(\sigma_2)$ then determine a unitary representation $\rho^* \colon G^* \to \operatorname{U}(2)$ with the property that $\rho^*(x) \ell_i = \ell_{\pi(x)(i)}$ for each $x \in G$ and $i \in [4]$.

Next, we consider how a subgroup of $G^*$ holds a line invariant via $\rho^*$.
The stabilizer $G_1 \leq A_4$ of $1\in[4]$ is generated by $\sigma_2$, and $G_1^*:=\langle x_2\rangle$ is the preimage of $G_1$ in $G^*$.
As such, $G_1^*={\{x\in G^*:}\rho^*(x)\ell_1=\ell_1\}$ is the stabilizer of $\ell_1$.
Since $\ell_1$ is spanned by $\varphi:=\frac{1}{b}\left[\begin{smallmatrix}\omega\\a\end{smallmatrix}\right]$, the action of $G_1^*$ on $\ell_1$ determines a homomorphism $\alpha\colon G_1^*\to\mathbb{T}$ such that $\rho^*(x)\varphi=\alpha(x)\varphi$ for every $x\in G_1^*$.
Explicitly, $\alpha(x_2)=-1$, since $\varphi$ is an eigenvector of $\rho^*(x_2)$ with eigenvalue $-1$.
It follows that the $G^*$-orbit of $\varphi$ consists of the columns of $z\Phi$ for each $z\in\{\pm1\}=C_{r'}$, where $r':=2$.
Meanwhile, the off-diagonal entries of the signature matrix of $\Phi$ are all members of $C_r$ with $r:=2r'=4$.

It will be convenient to extend $G^*$ and $\rho^*$ so that the resulting orbit of $\varphi$ consists of the columns of $z\Phi$ for every $z\in C_r$, as suggested by the signature matrix.
To accomplish this, we put $\tilde{G}^*:=G^*\times C_r$ and define $\tilde{\rho}^*\colon \tilde{G}^*\to\operatorname{U}(2)$ by $\tilde{\rho}^*(x,z):=z\rho^*(x)$.
Then the stabilizer of $\ell_1$ in $\tilde{G}^*$ is $\tilde{G}^*_1:=G^*_1\times C_r$, and the resulting character is $\tilde{\alpha}\colon\tilde{G}^*_1\to C_r$ given by $\tilde{\alpha}(x,z):=z\alpha(x)$.
Taking $H:=\{\tilde{x}\in\tilde{G}^*:\tilde{\rho}^*(\tilde{x})\varphi=\varphi\}=\operatorname{ker}\tilde{\alpha}$, then it turns out that $(\tilde{G}^*,H)$ is a Higman pair.
Furthermore:
\begin{center}
\textit{We can recover the doubly transitive lines from the Higman pair.}
\end{center}
Indeed, the $\tilde{G}^*$-orbit of $\varphi$ consists of $rn=16$ vectors, and their Gram matrix is a scalar multiple of an (easy-to-find) primitive idempotent in the adjacency algebra of $(\tilde{G}^*,H)$.
\end{example}

Following the previous example, we note that when $r = 2r'$, each element of $C_{r'}$ can be obtained by squaring an appropriate element of $C_r$.
In this sense, $C_r$ consists of radicals of $C_{r'}$.
This fact motivates the following terminology, which will play a crucial role in our proofs of the Higman Pair Theorem and the main result.

\begin{definition}
\label{defn:radicalization}
Given a group $G^*$, a subgroup $G^*_1 \leq G^*$, and a linear character $\alpha \colon G_1^* \to \mathbb{T}$, put $r' =|\operatorname{im}\alpha|$, $r=2r'$, and define $\widetilde{\alpha}\colon G^*_1 \times C_r\to C_r$ by $\widetilde{\alpha}(x,z)=\alpha(x)z$.
(Here, $C_r \leq \mathbb{C}^\times$ is the group of $r$th roots of unity under multiplication.)
We say $(G^*\times C_r,\operatorname{ker}\widetilde{\alpha})$ is the \textbf{radicalization} of $(G^*,G^*_1,\alpha)$.
\end{definition}

\begin{theorem}[Lines from permutations]
\label{thm.radicalization}
Let $G \leq S_n$ be a doubly transitive permutation group with $n \geq 3$, and let $G_1\leq G$ be the stabilizer of $1 \in [n]$.
Choose any Schur covering $\pi\colon G^* \to G$, and put $G^*_1:=\pi^{-1}(G_1)$.
If $G$ is a group of automorphisms of a sequence $\mathscr{L}$ of $n > d$ lines spanning $\mathbb{C}^d$, then there is a linear character $\alpha \colon G_1^* \to \mathbb{T}$ such that:
\begin{itemize}
\item[(a)]
the radicalization $(\widetilde{G}^*,H)$ of $(G^*,G^*_1,\alpha)$ is a Higman pair, and
\item[(b)]
the Schurian scheme of $(\widetilde{G}^*,H)$ has a primitive idempotent that is the Gram matrix of $r:=[N_{\widetilde{G}^*}(H):H]$ equal-norm representatives from each of the $n:=[\widetilde{G}^*:N_{\widetilde{G}^*}(H)]$ lines in $\mathscr{L}$, and furthermore, this Gram matrix carries the scheme.
\end{itemize}
\end{theorem}

Our proof of Theorem~\ref{thm.radicalization} uses the following well-known characterizations of doubly transitive actions.
For background on permutation groups, we refer the reader to~\cite{DixonMortimer:96}.

\begin{proposition}
\label{prop: char of 2tran}
Let $G$ be a finite group acting transitively on a set $X$, and let $G_1 \leq G$ be the stabilizer of some point $x_1 \in X$.
Then the following are equivalent:
\begin{itemize}
\item[(a)]
$G$ acts doubly transitively on $X$.
\item[(b)]
$G_1$ acts transitively on $X\setminus \{x_1\}$.
\item[(c)]
For any $\sigma \in G \setminus G_1$, $G = G_1 \cup G_1 \sigma G_1$.
\item[(d)]
For any epimorphism $\pi \colon G^* \to G$, $G^*$ acts doubly transitively on $G^*/G_1^*$, where $G_1^* := \pi^{-1}(G_1)$.
\end{itemize}
When these conditions hold, $G_1 \leq G$ is a maximal subgroup.
\end{proposition}

\begin{proof}[Proof of Theorem~\ref{thm.radicalization}]
Denote $\mathscr{L} =: \{ \ell_j \}_{j \in [n]}$.
Since $G$ consists of automorphisms of $\mathscr{L}$, we can choose unitaries $\{ \rho(\sigma) \}_{\sigma \in G}$ such that $\rho(1) = I$ and $\rho(\sigma) \ell_i = \ell_{\sigma(i)}$ for every $i \in [n]$.
Then $\rho \colon G \to \operatorname{U}(d)$ is a projective unitary representation, by Lemma~\ref{lem:proj rep}.
Applying Proposition~\ref{prop: make honest}, we lift $\rho$ to an honest unitary representation $\rho^* \colon G^* \to \operatorname{U}(d)$ of the Schur cover, where $\rho^*(x) \ell_i = \ell_{\pi(x)(i)}$ for every $x \in G^*$ and $i \in [n]$.

Fix a unit vector $\varphi \in \ell_1$.
Since $G_1^*$ holds $\ell_1$ invariant, every $\xi \in G_1^*$ produces a unimodular constant $\alpha(\xi)$ such that $\rho^*(\xi) \varphi = \alpha(\xi) \varphi$.
Moreover, $\alpha \colon G_1^* \to \mathbb{T}$ is a linear character since $\rho^*$ is a homomorphism.
Put $r' := |\operatorname{im} \alpha|$ and $r := 2r'$, so that $\alpha$ maps $G_1^*$ onto $C_{r'} \leq C_r \leq \mathbb{T}$.
We would like to extend $\rho^*$ in such a way that its image includes all of $\{z I : z \in C_r\}$.
To that end, we define extension groups $\widetilde{G}^* := G^* \times C_r$ and $\widetilde{G}_1^* := G_1^* \times C_r$, and let $\tilde{\rho}^* \colon \widetilde{G}^* \to \operatorname{U}(d)$ be the unitary representation given by $\tilde{\rho}^*(x,z) = z\, \rho^*(x)$.
Then $\widetilde{G}^*$ permutes $\mathscr{L}$ according to the rule $\tilde{\rho}^*(x,z) \ell_i = \ell_{\pi(x)(i)}$, and in particular
\begin{equation}
\label{eq: line stabilizer}
\widetilde{G}_1^* = \bigl\{ \tilde{x} \in \widetilde{G}^* : \tilde{\rho}^*(\tilde{x}) \ell_1 = \ell_1 \bigr\}.
\end{equation}
As before, this implies that $\rho^*(x,z) \varphi = \widetilde{\alpha}(\xi,z) \varphi$ for every $(\xi,z) \in \widetilde{G}_1^*$, where $\widetilde{\alpha} \colon \widetilde{G}_1^* \to C_r$ is the character $\widetilde{\alpha}(\xi,z) = z \alpha(\xi)$.
Let $H \leq \widetilde{G}_1^*$ be the group that stabilizes not only $\ell_1$, but also $\varphi$:
\[ H = \bigl\{ \tilde{x} \in \widetilde{G}^* : \tilde{\rho}^*(\tilde{x}) \varphi = \varphi \bigr\} = \ker \widetilde{\alpha} = \bigl\{ \bigl(\xi,\alpha(\xi)^{-1}\bigr) \in \widetilde{G}_1^* : \xi \in G_1^* \bigr\}. \]
Then $(\widetilde{G}^*, H)$ is the radicalization of $(G^*,G^*_1,\alpha)$.
The notation established thus far is partially summarized in Figure~\ref{fig:notation}.

To prove $(\widetilde{G}^*, H)$ is a Higman pair, first observe that $H = \ker \widetilde{\alpha} \trianglelefteq \widetilde{G}_1^*$, and so $N_{\widetilde{G}^*}(H) \geq \widetilde{G}_1^*$.
By Proposition~\ref{prop: char of 2tran} and~\eqref{eq: line stabilizer}, $\widetilde{G}^*$ acts doubly transitively on $\widetilde{G}^*/\widetilde{G}_1^*$, and $\widetilde{G}_1^*\leq \widetilde{G}^*$ is a maximal subgroup.
We claim that $H$ is not normal in $\widetilde{G}^*$, and therefore $\widetilde{G}_1^* = N_{\widetilde{G}^*}(H)$.
Notice that $G_1^*$ is not normal in $G^*$ since it has $n\geq 3$ left cosets but only two double cosets.
Choose $\xi \in G_1^*$ and $x \in G^*$ such that $x\xi x^{-1} \notin G_1^*$.
Then $(\xi,\alpha(\xi)^{-1}) \in H$, but $(x,1)(\xi,\alpha(\xi)^{-1})(x,1)^{-1} = (x\xi x^{-1},\alpha(\xi)^{-1}) \notin H$.
This proves the claim, and (H1) follows.
We have (H2) since $\widetilde{G}_1^*/\ker \widetilde{\alpha} \cong C_r$.

We now obtain a candidate key for $(\widetilde{G}^*, H)$.
For each $i \in [n]$, choose $x_i \in G^*$ such that $\pi(x_i)(1) = i$.
Then $\{ x_i \}_{i \in [n]}$ is a transversal for $G^* / G_1^*$, and $\Phi := \{ \rho^*(x_i) \varphi \}_{i\in [n]}$ is an ETF by Proposition~\ref{prop: 2tran implies ETF}.
Every $x \in G^*$ can be written uniquely in the form $x = x_i \xi$ for some $i \in [n]$ and $\xi \in G_1^*$, and so
\begin{equation}
\label{eq: phi levels}
| \langle \rho^*(x) \varphi, \varphi \rangle | = \begin{cases}
1 & \text{if }x \in G_1^*, \\
\mu & \text{if }x \notin G_1^*,
\end{cases}
\end{equation}
where $\mu = \sqrt{\frac{n-d}{d(n-1)}} \notin \{0,1\}$ by the Welch bound~\cite{Welch:74}.
We claim that $\mu^{-1} \langle \rho^*(x) \varphi, \varphi \rangle \in C_r$ whenever $x \notin G_1^*$.
Choose such an $x$.
Then $\pi(x) \neq 1$.
Since $G$ acts doubly transitively, there exists $y \in G^*$ such that $\pi(y)[\pi(x)(1)] = 1$ and $\pi(y)(1) = \pi(x)(1)$.
We have $yx,x^{-1}y \in G_1^*$, and so
\begin{equation}
\label{eq:flip inner product}
\langle \rho^*(x) \varphi, \varphi \rangle 
= \bigl\langle \rho^*(yx) \varphi, \rho^*\bigl(xx^{-1}y\bigr) \varphi \bigr\rangle 
= \bigl\langle \alpha(yx) \varphi, \rho^*(x) \alpha\bigl(x^{-1} y \bigr) \varphi \bigr\rangle
= \alpha\bigl(yxy^{-1}x\bigr) \langle \varphi, \rho^*(x) \varphi \rangle.
\end{equation}
Write $\langle \rho^*(x) \varphi, \varphi \rangle =: z \mu$.
Then~\eqref{eq:flip inner product} says that $z \mu = \alpha\bigl(yxy^{-1}x\bigr) z^{-1} \mu$, so that $z^2 = \alpha\bigl(yxy^{-1}x\bigr) \in C_{r'}$.
Therefore, $z \in C_r$, as desired.
In particular, there exists $b = \bigl(x,z^{-1}\bigr) \in \widetilde{G}^* \setminus \widetilde{G}_1^*$ such that $\langle \tilde{\rho}^*(b) \varphi, \varphi \rangle = \mu$.
Fix such $b$ for the remainder of the proof.

Next, we demonstrate the crucial identity
\begin{equation}
\label{eq: double cosets}
\text{for any }\tilde{x},\tilde{y} \in \widetilde{G}^*,
\qquad
H\tilde{x}H = H\tilde{y}H \iff \langle \tilde{\rho}^*(\tilde{x}) \varphi, \varphi \rangle = \langle \tilde{\rho}^*(\tilde{y}) \varphi, \varphi \rangle.
\end{equation}
If $\tilde{x} = a \tilde{y} a'$ for $a,a' \in H$, then
\[ 
\langle \tilde{\rho}^*(\tilde{x}) \varphi, \varphi \rangle 
= \bigl\langle \tilde{\rho}^*(\tilde{y}) \tilde{\rho}^*\bigl(a'\bigr) \varphi, \tilde{\rho}^*\bigl(a^{-1}\bigr)\varphi \bigr\rangle 
= \langle \tilde{\rho}^*(\tilde{y}) \varphi, \varphi \rangle.
\]
Conversely, suppose $\tilde{x} = (x,z)$ and $\tilde{y} = (y,w)$ satisfy $\langle \tilde{\rho}^*(\tilde{x}) \varphi, \varphi \rangle = \langle \tilde{\rho}^*(\tilde{y}) \varphi, \varphi \rangle$, that is, 
\[ z \langle \rho^*(x) \varphi, \varphi \rangle = w \langle \rho^*(y) \varphi, \varphi \rangle. \] 
By~\eqref{eq: phi levels}, either $x,y \in G_1^*$ or $x,y \notin G_1^*$.
In the former case, we have $z \alpha(x) = w \alpha(y)$, so that $\widetilde{\alpha}\bigl(y^{-1}x,w^{-1} z\bigr) = 1$ and $(x,z) = (y,w)\bigl(y^{-1}x,w^{-1}z\bigr) \in H(y,w)H$.
On the other hand, if $x,y \notin G_1^*$ then there exist $\xi, \eta \in G_1^*$ such that $x = \xi y\eta$, by Proposition~\ref{prop: char of 2tran}.
In that case,
\[ 
w \langle \rho^*(y) \varphi, \varphi \rangle 
= z \langle \rho^*(x) \varphi, \varphi \rangle 
= z \bigl\langle \rho^*(y) \alpha\bigl(\eta \bigr) \varphi, \alpha(\xi)^{-1} \varphi \bigr\rangle 
= z \alpha(\eta) \alpha(\xi) \langle \rho^*(y) \varphi, \varphi \rangle
\]
and $z = \alpha(\xi)^{-1} w \alpha(\eta)^{-1}$, so that $(x,z) = \bigl(\xi,\alpha(\xi)^{-1}\bigr) (y,w) \bigl(\eta,\alpha(\eta)^{-1}\bigr) \in H(y,w)H$.
This proves~\eqref{eq: double cosets}.

The remaining Higman pair axioms follow quickly from~\eqref{eq: double cosets}.
Since $\langle \tilde{\rho}^*(b) \varphi, \varphi \rangle = \mu \in \mathbb{R}$, we have $\bigl\langle \tilde{\rho}^*\bigl(b^{-1}\bigr) \varphi, \varphi \bigr\rangle = \langle  \varphi, \tilde{\rho}^*(b) \varphi \rangle = \mu$, and so $b^{-1} \in HbH$.
This is (H3).
For every $a \in \widetilde{G}_1^*$, 
\[ 
\bigl\langle \tilde{\rho}^*\bigl(aba^{-1}\bigr) \varphi, \varphi \bigr\rangle 
= \bigl\langle \tilde{\rho}^*(b) \tilde{\rho}^*\bigl(a^{-1}\bigr) \varphi, \tilde{\rho}^*\bigl(a^{-1}\bigr)\varphi \bigr\rangle 
= \bigl\langle \tilde{\rho}^*(b) \widetilde{\alpha}\bigl(a^{-1}\bigr) \varphi, \widetilde{\alpha}\bigl(a^{-1}\bigr)\varphi \bigr\rangle 
= \langle \tilde{\rho}^*(b) \varphi, \varphi \rangle,
\]
so $aba^{-1} \in HbH$.
This is (H4).
Finally, 
\[ 
\langle \tilde{\rho}^*(ab) \varphi, \varphi \rangle 
= \bigl\langle \tilde{\rho}^*(b) \varphi, \tilde{\rho}^*\bigl(a^{-1}\bigr) \varphi \bigr\rangle 
= \widetilde{\alpha}(a) \langle \tilde{\rho}^*(b) \varphi, \varphi \rangle,
\]
so $ab \in HbH$ only if $a \in \ker \widetilde{\alpha} = H$.
This is (H5), and the proof of (a) is complete.

To prove (b), put $\widetilde{\Phi} = \{ \tilde{\rho}^*(x_i, z) \varphi \}_{i \in [n],\, z \in C_r} = \{ z\cdot \rho^*(x_i) \varphi \}_{i\in [n],\, z \in C_r}$.
As $\operatorname{span}\{ \rho^*(x_i) \varphi \} = \rho^*(x_i) \ell_1 = \ell_i$, $\widetilde{\Phi}$ consists of $r$ representatives from each line in $\mathscr{L}$.
Since $\Phi = \{ \rho^*(x_i) \varphi \}_{i \in [n]}$ is an ETF, Proposition~3.20 in~\cite{IversonJM:17} implies that $\widetilde{\Phi}$ is a tight frame.
In particular, its Gram matrix $\mathcal{G} := \widetilde{\Phi}^* \widetilde{\Phi}$ is a scalar multiple of an orthogonal projection.
Notice that $\{ (x_i, z) \}_{i \in [n],\, z \in C_r}$ is a transversal for $\widetilde{G}^*/H$, since $|\widetilde{G}^*/H| = |\widetilde{G}^*/\widetilde{G}_1^*|\cdot |\widetilde{G}_1^*/H| = nr$ and 
\[ (x_i^{-1} x_j, z^{-1} w) \in H \iff x_i^{-1} x_j \in G_1^*\text{ and } \alpha(x_i^{-1} x_j) = w^{-1} z \iff i=j \text{ and } z=w. \]
It follows from Theorem~3.2 in~\cite{IversonJM:17} that $\mathcal{G}$ lies in the adjacency algebra of $(\widetilde{G}^*,H)$.
Moreover,~\eqref{eq: double cosets} implies that two entries
\[ \mathcal{G}_{(i,z),\, (j,w)} = \langle \tilde{\rho}^*(x_i^{-1} x_j, z^{-1} w) \varphi, \varphi \rangle \qquad \text{and} \qquad \mathcal{G}_{(i',z'),\, (j',w')} = \langle \tilde{\rho}^*(x_{i'}^{-1} x_{j'}, (z')^{-1} {w'}) \varphi, \varphi \rangle \]
of $\mathcal{G}$ are equal if and only if $(x_{i'},z')^{-1} (x_{j'},w') \in H(x_i,z)^{-1} (x_j,w) H$.
Therefore, $\mathcal{G}$ carries the Schurian scheme of $(\widetilde{G}^*,H)$.

It remains only to prove that a scalar multiple of $\mathcal{G}$ is a \emph{primitive} idempotent for $(\widetilde{G}^*,H)$.
To that end, let $B \in \mathbb{C}[C_r]^{n\times n}$ be the roux given by Proposition~\ref{prop.higman's roux} for the Higman pair $(\widetilde{G}^*,H)$.
Specifically, choose $\{ \tilde{x}_i \}_{i\in [n]} := \{ (x_i,1) \}_{i \in [n]}$ as a transversal for $\widetilde{G}^*/\widetilde{G}_1^*$, and $\{ \tilde{a}_z \}_{z\in C_r} := \{ (1,z) \}_{z \in C_r}$ as a transversal for $\tilde{G}_1^* / H$.
For $i \neq j$, \eqref{eq: double cosets} implies that $\tilde{x}_i^{-1} \tilde{x}_j \in H \tilde{a}_z b H$ if and only if $\langle \rho^*(x_i^{-1} x_j) \varphi, \varphi \rangle = z \mu$.
By Proposition~\ref{prop.higman's roux},
\[ B_{ij} 
= \begin{cases}
 \mu^{-1} \langle  \rho^*(x_i^{-1} x_j) \varphi, \varphi \rangle & \text{if }i\neq j, \\
 0 & \text{otherwise}.
 \end{cases}
 \]
 If $\beta \in \hat{C_r}$ is given by $\beta(z) = z^{-1}$, then comparison of the entries shows that 
 \[ \mathcal{G} = \sum_{z \in C_r} \beta(z) \lceil z I \rfloor + \mu \sum_{z\in C_r} \beta(z) \lceil z B \rfloor. \]
 Then Proposition~\ref{prop.roux primitive idempotents} implies that a scalar multiple of $\mathcal{G}$ is a primitive idempotent for $(\widetilde{G}^*,H)$.
 \end{proof}

\begin{proof}[Proof of Theorem~\ref{thm.Higman Pair Theorem}(a)]
Immediate from Theorem~\ref{thm.radicalization}.
\end{proof}

The language of the following corollary intentionally mirrors that of Theorem~3.3 in~\cite{IversonM:18}, which gave a parallel characterization of roux lines.

\begin{corollary}
Let $\mathscr{L}$ be a sequence of linearly dependent complex lines.
Then $\mathscr{L}$ is doubly transitive if and only if all of the following occur simultaneously:
\begin{itemize}
\item[(a)]
$\mathscr{L}$ is equiangular,
\item[(b)]
there exist unit-norm representatives $\{ \varphi_j \}_{j\in [n]}$ of $\mathscr{L}$ whose signature matrix is comprised of $r$th roots of unity for some $r$, and
\item[(c)]
the Gram matrix $\mathcal{G}$ of $\{ z \varphi_j \}_{j \in [n],\, z \in C_r}$ carries a Schurian association scheme.
\end{itemize}
In this case, $\mathcal{G}$ carries the Schurian association scheme of a Higman pair, and a scalar multiple of $\mathcal{G}$ is a primitive idempotent for that scheme.
\end{corollary}

\begin{proof}
If $\mathscr{L}$ is doubly transitively, then (a)--(c) were verified in the proof of Theorem~\ref{thm.radicalization}.
Conversely, if (a)--(c) are satisfied, then Theorem~3.3 in~\cite{IversonM:18} gives that there is a roux having $\mathscr{L}$ as roux lines and the Schurian association scheme carried by $\mathcal{G}$ as its roux scheme, with a scalar multiple of $\mathcal{G}$ as a primitive idempotent.
By Proposition~\ref{prop.schurian roux}, $\mathcal{G}$ carries the Schurian association scheme of a Higman pair.
Then Theorem~\ref{thm.Higman Pair Theorem} implies that $\mathcal{G}$ is the Gram matrix of $r$ unit-norm representatives from a sequence of doubly transitive lines, namely $\mathscr{L}$.
\end{proof}

\section{Roux from radicalization}
\label{sec.roux tech}

The previous section leveraged Theorem~\ref{thm.radicalization} to prove part~(a) of the Higman Pair Theorem, and in the next section, we will again apply Theorem~\ref{thm.radicalization} to classify sequences of doubly transitive lines that exhibit almost simple symmetries.
In particular, for every almost simple doubly transitive permutation group, we will determine whether that group is a subgroup of the automorphism group of some sequence of lines.
By Theorem~\ref{thm.radicalization}, it suffices to find a linear character for which a certain radicalization is a Higman pair.
In the event that such a character exists, we will need to dig further to describe the lines, find the dimension of their span, and determine whether the lines are real.
The purpose of this section is to build up the machinery needed to answer such questions.

\begin{remark}
\label{rem.setting}
Throughout this section, the results we present will assume the following setting without mention.
Let $G$ be a doubly transitive permutation group on $n\geq 3$ points, and let $G_1\leq G$ be the stabilizer of some point.
Given a Schur covering $\pi\colon G^* \to G$, put $G^*_1 = \pi^{-1}(G_1)$.
Fix a linear character $\alpha \colon G_1^* \twoheadrightarrow C_{r'}$ and denote $r := 2r' = 2|\operatorname{im} \alpha|$, $\widetilde{G}^* := G^* \times C_r$, and $\widetilde{G}_1^* := G_1^* \times C_r$.
Let $\widetilde{\alpha}\colon \widetilde{G}_1^* \to C_r$ be the epimorphism given by $\widetilde{\alpha}(x,z) = z\, \alpha(x)$, and put
\begin{equation}
\label{eq.H as kernel}
H := \ker \widetilde{\alpha} = \bigl\{ \bigl(\xi,\alpha(\xi)^{-1}\bigr) : \xi \in G_1^* \bigr\},
\end{equation}
so that $(\widetilde{G}^*,H)$ is the radicalization of $(G^*,G_1^*,\alpha)$.
This notation is summarized in Figure~\ref{fig:notation}.
\end{remark}

The following lemma is proved as in the proof of Theorem~\ref{thm.radicalization}.

\begin{lemma}
\label{lem.normalizer}
$\widetilde{G}^*$ acts doubly transitively on $\widetilde{G}^*/\widetilde{G}_1^*$. Furthermore, $N_{\widetilde{G}^*}(H) = \widetilde{G}_1^*$.
\end{lemma}

\begin{lemma}[Key finder]
\label{lem.key finder}
Assume that $(\widetilde{G}^*, H)$ is a Higman pair.
Given $x \in G^* \setminus G^*_1$, there exist $\xi, \eta \in G_1^*$ such that $x^{-1} = \xi x \eta$.
Pick either of the $z \in C_r$ such that $z^2 = \alpha(\xi\eta)$.
Then $(x,z)$ is a key for $(\widetilde{G}^*,H)$.
\end{lemma}

\begin{proof}
First observe that $x \notin G_1^*$, so $(x,z) \notin \widetilde{G}_1^* = N_{\widetilde{G}^*}(H)$.
To prove (H3) for $b$, notice that $z^{-1} = z\,\alpha(\xi\eta)^{-1}$, so that
\begin{equation}
\label{eq:H3}
(x,z)^{-1}
= \bigl(\xi x \eta, z\, \alpha(\xi \eta)^{-1}\bigr)
=\bigl( \xi, \alpha(\xi)^{-1} \bigr) (x,z) \bigl( \eta, \alpha(\eta)^{-1} \bigr)
\in H(x,z)H.
\end{equation}
For (H4), take any $(\zeta, w) \in \widetilde{G}_1^*$ and compute
\begin{equation}
\label{eq:H4}
(\zeta,w)(x,z)(\zeta,w)^{-1}
= \bigl(\zeta x \zeta^{-1}, z\bigr)
=\bigl( \zeta, \alpha(\zeta)^{-1} \bigr)(x,z)\bigl(\zeta^{-1},\alpha(\zeta) \bigr)
\in H(x,z)H.
\end{equation}

It remains to prove (H5) for $b$. 
To that end, assume $(\zeta, w) \in \widetilde{G}_1^*$ satisfies $(\zeta,w)(x,z) \in H(x,z)H$.
Then there exist $\xi',\eta' \in G_1^*$ such that 
\[ (\zeta,w)(x,z) = \bigl( \xi', \alpha(\xi')^{-1} \bigr) (x,z) \bigl( \eta', \alpha(\eta')^{-1} \bigr). \]
In other words, $\zeta x = \xi' x \eta'$ and $w = \alpha(\xi' \eta')^{-1}$.
Since $(\widetilde{G}^*,H)$ is a Higman pair by assumption, we may select a key $(x_0,z_0)$.
As $x,x_0 \notin G_1^*$, Lemma~\ref{lem.normalizer} and Proposition~\ref{prop: char of 2tran} provide $\xi'', \eta'' \in G_1^*$ such that $x = \xi'' x_0 \eta''$.
Substituting this relation into $\zeta x = \xi' x \eta'$, we obtain $\zeta \xi'' x_0 \eta'' = \xi' \xi'' x_0 \eta'' \eta'$.
Recalling that $w = \alpha(\xi' \eta')^{-1}$, we have
\[
\bigl(\zeta \xi'', w \alpha(\xi'')^{-1}\bigr)(x_0,z_0)
= \bigl(\xi' \xi'' x_0 \eta'' \eta' (\eta'')^{-1}, \alpha(\xi' \eta')^{-1} \alpha(\xi'')^{-1} z_0 \bigr)
\]
\[
= \bigl( \xi' \xi'', \alpha(\xi' \xi'')^{-1} \bigr)
(x_0,z_0)
\bigl(\eta'' \eta' (\eta'')^{-1}, \alpha(\eta')^{-1}\bigr)
\in H(x_0,z_0)H.
\]
Since $(x_0,z_0)$ is a key, (H5) implies that $\bigl(\zeta \xi'', w \alpha(\xi'')^{-1}\bigr) \in H$.
In other words, $\alpha(\zeta \xi'')^{-1} = w \alpha(\xi'')^{-1}$, so that $(\zeta,w) = \bigl( \zeta, \alpha(\zeta)^{-1} \bigr) \in H$.
Therefore (H5) holds for $b$.
\end{proof}

\begin{theorem}[Higman pair detector]
\label{thm.radicalization to Higman}
The following are equivalent for any choice of $x\in G^* \setminus G^*_1$:
\begin{itemize}
\item[(a)]
The radicalization of $(G^*,G^*_1,\alpha)$ is a Higman pair.
\item[(b)]
Every $\xi \in G^*_1$ for which $x\xi x^{-1} \in G^*_1$ satisfies $\alpha(x\xi x^{-1}) = \alpha(\xi)$.
\end{itemize}
\end{theorem}

We emphasize that the equivalence of (a) and (b) holds for \emph{any} choice of $x \in G^* \setminus G_1^*$, despite the fact that $x$ appears in (b) but not in (a).
As we will see, that flexibility makes this tool powerful.

Theorem~\ref{thm.radicalization to Higman} implies that the radicalization of $(G^*,G^*_1,\alpha)$ is a Higman pair only if $\ker \alpha $ is \textbf{strongly closed} in $G_1^*$ with respect to $G^*$.
That is, for any $x \in G^*$, $x(\ker \alpha)x^{-1} \cap G_1^* \leq \ker \alpha$.
Related conditions for doubly transitive two-graphs appear in Theorem~6.1 of~\cite{Taylor:77}, Theorem~6.1 of~\cite{Cameron:77}, and Remark~7.6 of~\cite{Seidel:76}.

\begin{proof}[Proof of Theorem~\ref{thm.radicalization to Higman}]
Assume (a) holds.
By Lemma~\ref{lem.key finder}, there exists $z \in C_r$ such that $b:=(x,z)$ is a key for the Higman pair $(\widetilde{G}^*,H)$.
Choose $\xi \in G_1^*$ such that $x \xi x^{-1} = \eta \in G^*_1$, as in (b).
Then
\[
\bigl(\xi \eta^{-1}, 1\bigr)(x,z)
= \bigl(\xi x \xi^{-1} x, z\bigr)
= \bigl(\xi, \alpha(\xi)^{-1} \bigr) (x,z) \bigl( \xi^{-1}, \alpha(\xi) \bigr)
\in H(x,z)H.
\]
By (H5), $\bigl(\xi \eta^{-1}, 1\bigr) \in H$. Consequently, $\alpha\bigl(\xi \eta^{-1}\bigr) = 1$ and $\alpha\bigl(x\xi x^{-1}\bigr) = \alpha(\eta) = \alpha(\xi)$.

Conversely, suppose (b) holds.
By Lemma~\ref{lem.normalizer} and Proposition~\ref{prop: char of 2tran}, there exist $\eta, \zeta \in G_1^*$ such that $x^{-1} = \eta x \zeta$.
Let $z \in C_r$ be such that $z^2 = \alpha(\eta \zeta)$.
We claim that $(\widetilde{G}^*,H)$ is a Higman pair with key $b := (x,z)$.
Lemma~\ref{lem.normalizer} supplies (H1), while (H2) holds by the first isomorphism theorem $\widetilde{G}^*_1 / H \cong C_r$.
Calculations similar to those in~\eqref{eq:H3} and~\eqref{eq:H4} verify (H3) and (H4).

To prove (H5), suppose $(\zeta',w) \in \widetilde{G}^*_1$ satisfies $(\zeta',w)(x,z) \in H(x,z)H$.
Then there exist $\xi',\eta' \in G_1^*$ such that
\[ (\zeta' x, wz) = \bigl(\xi', \alpha(\xi')^{-1}\bigr) (x,z) \bigl( \eta', \alpha(\eta')^{-1} \bigr). \]
That is, $\zeta' x = \xi' x \eta'$ and $w = \alpha( \xi' \eta')^{-1}$.
We can rewrite the former relation as $x \eta' x^{-1} = (\xi')^{-1} \zeta' \in G_1^*$.
Then (b) gives $\alpha(\eta') = \alpha\bigl( (\xi')^{-1} \zeta')$.
Therefore $\alpha(\zeta')^{-1} = \alpha(\xi' \eta')^{-1} = w$, and $(\zeta', w) \in H$.
\end{proof}

\begin{lemma}[Roux from radicalization]
\label{lem.radicalization's roux}
Suppose the radicalization $(\tilde{G}^*,H)$ of $(G^*,G_1^*,\alpha)$ is a Higman pair with key $(x,z)$.
Then the roux $B\in\mathbb{C}[C_r]^{n\times n}$ constructed in Proposition~\ref{prop.higman's roux} can be found as follows:
Choose left coset representatives $\{x_j\}_{j\in[n]}$ for $G_1^*$ in $G^*$.
Given $i\neq j$, there exist $\xi,\eta\in G_1^*$ for which $x_i^{-1}x_j=\xi x\eta$.
For any such choice of $\xi$ and $\eta$, define $B_{ij}=\alpha(\xi\eta)z^{-1}$, and set $B_{ii}=0$.
Then $B$ is a roux for $C_r$, and the roux scheme generated by $B$ is isomorphic to the Schurian scheme of $(\tilde{G}^*,H)$.
In particular, if $\alpha$ is real valued and $z\in\{\pm1\}$, then all lines arising from $B$ (in the sense of Proposition~\ref{prop.roux signature}) are real.
\end{lemma}

\begin{proof}
We apply Proposition~\ref{prop.higman's roux} with $b:=(x,z)$.
Lemma~\ref{lem.normalizer} gives that $K:=N_{\tilde{G}^*}(H) = G_1^* \times C_r$.
We may therefore select $\{\tilde{x}_j \}_{j\in [n]} := 
\{ (x_j,1) \}_{j \in [n]}$ as left coset representatives of $K$ in $\tilde{G}^* = G^* \times C_r$.
Furthermore, we may select $\{ a_w \}_{w \in C_r} := \{ (1,w) \}_{w \in C_r}$ as coset representatives of $H$ in $K$; indeed, $\tilde{\alpha}$ maps $K$ onto $C_r$ with kernel $H$, and $\tilde{\alpha}(1,w) = w$ for each $w \in C_r$.
Following Proposition~\ref{prop.higman's roux}, we seek $w \in C_r$ for which $\tilde{x}_i^{-1} \tilde{x}_j \in H a_w b H$.
Recalling the definition of $H$ in~\eqref{eq.H as kernel}, then for any choice of $w \in C_r$ and $\xi,\eta \in G_1^*$, we have
\[
(\xi,\alpha(\xi)^{-1}) a_w b (\eta, \alpha(\eta)^{-1})
= (\xi x \eta, \alpha(\xi \eta)^{-1} w z).
\]
This equals $\tilde{x}_i^{-1} \tilde{x}_j$ if and only if $x_i^{-1} x_j = \xi x \eta$ and $w = \alpha(\xi \eta) z^{-1}$.
There exist $\xi,\eta \in G_1^*$ for which this occurs by Proposition~\ref{prop.double cosets}, and the desired result follows immediately from Proposition~\ref{prop.higman's roux}.
For the ``in particular'' part of the result, assume $\alpha$ is real valued and $z\in\{\pm1\}$.
Then every off-diagonal entry of $B$ has order $2$, and evaluating at any character results in a real signature matrix, which in turn determines real lines, by Proposition~\ref{prop:real lines are real}.
\end{proof}

\begin{lemma}[Roux parameters for radicalization]
\label{lem.cgs of radical roux}
Assume $(\widetilde{G}^*,H)$ is a Higman pair, and let $(x,z)$ be any key.
Then the roux for $C_r\cong \widetilde{G}_1^*/H$ constructed in Proposition~\ref{prop.higman's roux} has parameters $\{c_{w}\}_{w\in C_r}$ given by
\[
c_{w}
=\frac{n-1}{|G^*_1|}\cdot\Big|\Big\{\zeta \in G^*_1:
\exists \xi,\eta \in G^*_1
\text{ s.t. }
x\zeta x^{-1}=\xi x\eta
\text{ and }
\alpha\bigl( \xi \eta \zeta^{-1} \bigr) z^{-1} = w
\Big\}\Big|.
\]
\end{lemma}

\begin{proof}
It follows from~\eqref{eq.higman's roux}.
Specifically, we choose $\{ a_w \}_{w \in C_r} := \{ (1,w) \}_{w \in C_r}$ as a transversal for $C_r$ in $\widetilde{G}_1^*$.
Then
\[ H a_w b H = \bigl\{ \bigl( \xi, \alpha(\xi)^{-1} \bigr)(1,w)(x,z)\bigl( \eta, \alpha(\eta)^{-1} \bigr) : \xi, \eta \in G_1^* \bigr\} 
= \bigl\{ \bigl( \xi x \eta, \alpha(\xi \eta)^{-1} wz\bigr) : \xi, \eta \in G_1^* \bigr\},
\]
while
\[ bHb^{-1} = \bigl\{ \bigl( x \zeta x^{-1}, \alpha(\zeta)^{-1} \bigr) : \zeta \in G_1^* \big\}. \]
Consequently,
\[ | bHb^{-1} \cap H a_w b H | = \Big|\Big\{\zeta \in G^*_1:
\exists \xi,\eta \in G^*_1
\text{ s.t. }
x \zeta x^{-1} =  \xi x \eta
\text{ and }
\alpha(\xi \eta)^{-1}wz  = \alpha(\zeta)^{-1}
\Big\}\Big|.
\]
The formula for $c_w$ now follows from~\eqref{eq.higman's roux} since $|H| = |G_1^*|$.
\end{proof}

\begin{lemma}
\label{lem.alpha=1 is dumb}
If $\alpha = 1$, then all roux lines coming from the Higman pair $(\widetilde{G}^*,H)$ span a space of dimension $1$ or $n-1$.
\end{lemma}

\begin{proof}
We have $r'=1$ and $r=2$.
Notice that $(\widetilde{G}^*,H)$ is a Higman pair by Theorem~\ref{thm.radicalization to Higman}.
Given any $x \in G^* \setminus G_1^*$, Lemma~\ref{lem.key finder} implies that $b=(x,1)$ is a key.
There are just two roux parameters $\{ c_1,c_{-1}\}$, and $c_{-1} = 0$ by Lemma~\ref{lem.cgs of radical roux}.
The sum of the roux parameters is therefore $c_1 = n-2$, by Proposition~\ref{prop.B squared}.
In particular, the Fourier transform of $\{ c_w \}_{w \in C_2}$ is given by $\hat{c}_\beta \equiv n-2$, $\beta \in\widehat{C_2}$.
In the notation of Proposition~\ref{prop.roux primitive idempotents}, we have $\mu_\beta^\epsilon \equiv \frac{n-2 + \epsilon n}{2n-2}$.
That is,  $\mu_\beta^+ \equiv 1$ and $\mu_\beta^- \equiv -1/(n-1)$.
Therefore, $d_\beta^+ \equiv 1$ and $d_\beta^- \equiv n-1$.
By Proposition~\ref{prop.roux primitive idempotents}, any sequence of $n$ roux lines coming from the Higman pair $(\widetilde{G}^*,H)$ spans a space of dimension $1$ or $n-1$.
\end{proof}

\begin{lemma}
\label{lem.real radical Higman lines}
Assume $\alpha$ is real valued and $(\widetilde{G}^*,H)$ is a Higman pair.
If $b=(x,z)$ is any key for $(\widetilde{G}^*,H)$ and $z \in \{\pm1\}$, then all roux lines coming from $(\widetilde{G}^*,H)$ are real.
\end{lemma}

\begin{proof}
We apply Proposition~\ref{prop.real line detector}, where $\Gamma = C_r$.
Since $\alpha$ is real valued, we either have $r' = 1$ or $r'=2$.
If $r'=1$, then $r=2$, and the desired result is immediate from Proposition~\ref{prop.real line detector}.
Now assume $r'=2$, so that $r=4$.
Since $z \in \{ \pm 1 \}$, the roux parameters $\{ c_w \}_{w \in C_4}$ satisfy $c_{\mathrm{i}} = c_{-\mathrm{i}} = 0$ by Lemma~\ref{lem.cgs of radical roux}.
For any $\beta \in \widehat{C_4}$, $\beta(1) = 1$ and $\beta(-1) \in \{ \pm 1 \}$.
In particular, both are real.
By Proposition~\ref{prop.real line detector}, the roux lines corresponding to any $\beta \in \widehat{C_4}$ are real.
\end{proof}

Combining Lemmas~\ref{lem.key finder} and~\ref{lem.real radical Higman lines}, we immediately obtain the following.

\begin{lemma}
\label{lem:real involution}
Assume $\alpha$ is real-valued and $(\widetilde{G}^*,H)$ is a Higman pair.
If there exists $x \in G^* \setminus G_1^*$ with $x = x^{-1}$, then all roux lines coming from $(\widetilde{G}^*,H)$ are real.
\end{lemma}

\section{Partial classification of doubly transitive lines}

Having developed the necessary machinery, we now put our tools to use in order to partially classify doubly transitive lines, as summarized in Theorem~\ref{thm:main result}.
In this section, we determine all sequences of linearly dependent doubly transitive lines whose automorphism groups are almost simple.

\subsection{Overview of classification strategy}

Doubly transitive permutation groups have been classified as a consequence of the classification of finite simple groups.
The following list is taken from~\cite{Kantor:85}.
Most of the groups below are not literally permutation groups, but each has a well-known permutation representation whose image is a doubly transitive subgroup of $S_n$.
Some of the groups have two different doubly transitive permutation representations, but when that happens they are exchanged by an outer automorphism of the source, and in particular, they produce the same range in $S_n$.
In that sense, each item below corresponds to a unique permutation group up to permutation equivalence.

\begin{proposition}[\cite{Kantor:85}]
\label{prop.classification}
If $S$ is a doubly transitive permutation group of a set of size $n$, then one of the following cases occurs:
\begin{itemize}
\item[(I)]
$S$ is of affine type.
\item[(II)]
$S$ has a normal subgroup $G \unlhd S\leq\operatorname{Aut}(G)$, where $G$ and $n$ are
\begin{itemize}
\item[(1)]
Alternating.
\begin{itemize}
\item[(i)]
$A_n$, $n\geq5$.
\end{itemize}
\item[(2)]
Lie type.
\begin{itemize}
\item[(ii)]
$\operatorname{PSL}(m,q)$, $m\geq2$, $n=(q^m-1)/(q-1)$, $(m,q)\neq(2,2),(2,3)$.
\item[(iii)]
$\operatorname{PSU}(3,q)$, $n=q^3+1$, $q>2$.
\item[(iv)]
$\operatorname{Sz}(q)$, $n=q^2+1$, $q=2^{2m+1}>2$.
\item[(v)]
$^2G_2(q)$, $n=q^3+1$, $q=3^{2m+1}$.
\end{itemize}
\item[(3)]
Other.
\begin{itemize}
\item[(vi)]
$\operatorname{Sp}(2m,2)$, $m\geq3$, $n=2^{2m-1}\pm2^{m-1}$.
\item[(vii)]
$\operatorname{PSL}(2,11)$, $n=11$.
\item[(viii)]
$A_7$, $n=15$.
\item[(ix)]
$M_n$, $n\in\{11,12,22,23,24\}$.
\item[(x)]
$M_{11}$, $n=12$.
\item[(xi)]
$\operatorname{HS}$ (the Higman--Sims group), $n=176$.
\item[(xii)]
$\operatorname{Co}_3$ (the third Conway group), $n=276$.
\end{itemize}
\end{itemize}
\end{itemize}
\end{proposition}

To prove our main result, we will consider each permutation group $G$ above in turn, and find all instances of doubly transitive lines whose automorphism group contains $G$.

We can quickly cross some groups off our list with the help of the following lemma, which extends Theorem~II.6 in~\cite{King:19}.
Recall that a permutation group $G \leq S_n$ is \textbf{triply transitive} if for any ordered triple of distinct indices $(i,j,k)$ in $[n]$ and any other such ordered triple $(i',j',k')$ there exists $\sigma \in G$ such that $\sigma(l) = l'$ for every $l \in \{ i,j,k\}$.

\begin{lemma}
\label{lem.3tran}
Let $\mathscr{L}$ be a sequence of $n>d$ lines with span $\mathbb{C}^d$.
If $\operatorname{Aut} \mathscr{L}$ is triply transitive, then $d\in\{1,n-1\}$, $\mathscr{L}$ is real, and $\operatorname{Aut} \mathscr{L} =S_n$.
\end{lemma}

\begin{proof}
If $n=2$, then $d=1$ and the result is trivial.
Hence we can assume $n \geq 3$.
Then $\operatorname{Aut} \mathscr{L}$ is doubly transitive, and $\mathscr{L}$ is equiangular by Proposition~\ref{prop: 2tran implies ETF}.
Choose unit-norm representatives $\{ \varphi_j \}_{j\in [n]}$ for $\mathscr{L} = \{ \ell_j \}_{j\in [n]}$, and put $\mu := | \langle \varphi_i, \varphi_j \rangle|$ for $i \neq j$.
By choosing new representatives for $\ell_2,\dotsc,\ell_n$ if necessary, we may assume that $\langle \varphi_1,\varphi_j \rangle = \mu$ whenever $2 \leq j \leq n$.
Given any pair of 3-subsets $\{i,j,k\}, \{i',j',k'\} \subset [n]$ there exists a unitary $U \in \operatorname{U}(d)$ and constants $\omega_l \in \mathbb{T}$ such that $U \varphi_l = \omega_l \varphi_{l'}$ for $l \in \{ i,j,k\}$.
Consequently,
\begin{equation}
\label{eq:triple products are equal}
\langle \varphi_i, \varphi_j \rangle \langle \varphi_j, \varphi_k \rangle \langle \varphi_k, \varphi_i \rangle = \langle \varphi_{i'}, \varphi_{j'} \rangle \langle \varphi_{j'}, \varphi_{k'} \rangle \langle \varphi_{k'}, \varphi_{i'} \rangle =: C.
\end{equation}
Taking $k=1$ shows that $\langle \varphi_i, \varphi_j \rangle = C/\mu^2$ whenever $2 \leq i\neq j \leq n$, and $C = \overline{C}$ since we can switch $i$ and $j$.
Consequently, $\mathscr{L}$ is real.
In the notation of Proposition~\ref{prop:real lines are two-graphs}, \eqref{eq:triple products are equal} implies that we either have $\mathcal{T}_{\mathscr{L}} = \emptyset$ or $\mathcal{T}_{\mathscr{L}} = \{ \text{all 3-subsets of }[n]\}$.
In either case, the two-graph $([n],\mathcal{T}_{\mathscr{L}})$ has automorphism group $S_n$.
By Proposition~\ref{prop:real lines are two-graphs}, $\operatorname{Aut}\mathscr{L} = S_n$, and some choice of unit-norm representatives for $\mathscr{L}$ has signature matrix $J-I$ or $I-J$.
In the former case $d=1$, and in the latter case $d=n-1$.
\end{proof}

For the remaining groups on our list, we apply the following broad strategy.
Let $G \leq S_n$ be a doubly transitive group from Proposition~\ref{prop.classification}.
Our first step is to obtain a Schur covering $\pi \colon G^* \to G$.
(This is the most difficult part in practice.
Thankfully, the multipliers of all the finite simple groups were computed as part of their classification, and a covering group can be found.)
Let $G_1 \leq G$ be the stabilizer of a point, and put $G_1^* := \pi^{-1}(G_1)$.
If there exists a sequence $\mathscr{L}$ of $n > d$ lines spanning $\mathbb{C}^d$ with $G \leq \operatorname{Aut}\mathscr{L}$, then Theorem~\ref{thm.radicalization} explains how to recover $\mathscr{L}$ from the radicalization $(\widetilde{G}^*,H)$ of $(G^*,G_1^*,\alpha)$, for some choice of linear character $\alpha \colon G_1^* \to \mathbb{T}$.
In particular, $(\widetilde{G}^*,H)$ must be a Higman pair.
Our plan is to iterate through all linear characters $\alpha \colon G_1^* \to \mathbb{T}$, find those for which $(\widetilde{G}^*,H)$ is a Higman pair, and describe the resulting lines.

In particular, for any choice of linear character $\alpha \colon G_1^* \to \mathbb{T}$, Theorem~\ref{thm.radicalization to Higman} provides a simple test to determine whether or not the radicalization $(\widetilde{G}^*,H)$ of $(G^*,G_1^*,\alpha)$ is a Higman pair.
When this test is affirmative, a sequence $\mathscr{L}$ of doubly transitive lines satisfying $G \leq \operatorname{Aut} \mathscr{L}$ exists by Theorem~\ref{thm.Higman Pair Theorem}.
For every such $\mathscr{L}$, it then remains to describe $\mathscr{L}$, determine whether or not $\mathscr{L}$ is real, and find the dimension of its span.

In many cases, Lemma~\ref{lem:real involution} provides a fast judgment that all lines coming from $(\widetilde{G}^*,H)$ are real.
In that case, our job is simple.
All real doubly transitive lines derive from Taylor's classification of doubly transitive two-graphs, hence they appear in Proposition~\ref{prop:real2tran}.

For the remaining cases, we dig deeper.
First, we apply Lemma~\ref{lem.key finder} to find a key for $(\widetilde{G}^*,H)$.
Then, we find the resulting roux with Lemma~\ref{lem.radicalization's roux}.
Next, we obtain roux parameters using Lemma~\ref{lem.cgs of radical roux}, and compute ranks of the primitive idempotents in the roux scheme with Proposition~\ref{prop.roux primitive idempotents}.
Finally, with the roux parameters in hand, Propositions~\ref{prop.two reps of same lines} and~\ref{prop.real line detector} team up to give a final determination of which primitive idempotents correspond to real lines, and which do not.

\subsection{Linear groups}
\label{subsec:linear}

We first consider $\operatorname{PSL}(m,q)$ and its doubly transitive action on the $1$-dimensional subspaces of $\mathbb{F}_q^m$.

\begin{example}
\label{ex.Paley conference}
Let $q$ be an odd prime power, take $X:=\mathbb{F}_q\cup\{\infty\}$, and consider the vectors $\{t_i\}_{i\in X}$ in $\mathbb{F}_q^2$ defined by $t_a:=[a,1]^\top$ for $a\in\mathbb{F}_q$
and $t_\infty:=[1,0]^\top$.
We will construct doubly transitive lines in $\mathbb{C}^{(q+1)/2}$ that are indexed by $X$ and have $\operatorname{PSL}(2,q)$ in their automorphism group.
The signature matrix can be described by composing a multiplicative character with a symplectic form.
Let $Q\leq\mathbb{F}_q^\times$ denote the multiplicative subgroup of index $2$, let $\chi\colon\mathbb{F}_q\to\mathbb{R}$ denote the quadratic character defined by
\[
\chi(a)
:=
\begin{cases}
1 & \text{if }a \in Q, \\
0 & \text{if }a = 0, \\
-1 & \text{if }a \not\in Q \cup\{0\},
\end{cases}
\]
and consider the symplectic form $[\cdot,\cdot]\colon\mathbb{F}_q^2\times\mathbb{F}_q^2\to\mathbb{F}_q$ defined by
\[
\Bigl[ [ a,b]^\top, [c,d]^\top \Bigr]:=ad-bc=\operatorname{det}\left(\left[\begin{array}{cc}a&c\\b&d\end{array}\right]\right).
\]
Finally, we define $A\in\mathbb{R}^{X\times X}$ by $A_{ij}=\chi([t_i,t_j])$, and put $\mathcal{S}:=z A\in\mathbb{C}^{X\times X}$, where
\[
z
:=
\begin{cases}
1 & \text{if } q \equiv 1 \bmod 4, \\
\mathrm{i} & \text{if } q \equiv 3 \bmod 4.
\end{cases}
\]
Then $\mathcal{S}$ is the signature matrix of an ETF of $q+1$ vectors in $\mathbb{C}^{(q+1)/2}$; see \cite{Renes:07} or \cite[p67]{Zauner:99}, for example.

To see that $\operatorname{PSL}(2,q)$ is in the automorphism group of these lines, first observe that each member of $\operatorname{PSL}(2,q)$ permutes the lines in $\mathbb{F}_q^2$ spanned by $\{t_i\}_{i\in X}$, which in turn determines a permutation of $X$.
Select any such permutation $\sigma$, and let $C$ denote a corresponding member of $\operatorname{SL}(2,q)$.
Then there exist scalars $\{\omega_i\}_{i\in X}$ in $\mathbb{F}_q^\times$ such that $t_{\sigma(i)}=\omega_iCt_i$.
Since $[\cdot,\cdot]$ is invariant under the action of $\operatorname{SL}(2,q)$, it follows that
\[
A_{\sigma(i),\sigma(j)}
=\chi([t_{\sigma(i)},t_{\sigma(j)}])
=\chi([\omega_iCt_i,\omega_jCt_j])
=\chi(\omega_i\omega_j[Ct_i,Ct_j])
=\chi(\omega_i) A_{ij} \chi(\omega_j).
\]
Letting $P\in\mathbb{R}^{X\times X}$ denote the matrix representation of $\sigma$ and $D\in\mathbb{R}^{X\times X}$ denote the diagonal matrix whose $i$th diagonal entry equals $\chi(\omega_i)$, this implies $PAP^{-1}=DAD$.
Multiplying both sides by $z$ then gives that $\mathcal{S}$ is switching equivalent to $P\mathcal{S}P^{-1}$.
Now Proposition~\ref{prop:switching equiv is unitary equiv} shows that $\sigma\in \operatorname{Aut} \mathscr{L}$.
(See Theorem~\ref{thm: PSL automorphism group} for the full automorphism group of $\mathscr{L}$.)
\end{example}

\begin{theorem}
\label{thm.linear groups}
Let $\mathscr{L}$ be a sequence of $n\geq 2d > 2$ lines with span $\mathbb{C}^d$ and doubly transitive automorphism group containing $\operatorname{PSL}(m,q)$ as in Proposition~\ref{prop.classification}(ii).
Then $m=2$, $d = (q+1)/2$, $n = q+1$, and $\mathscr{L}$ is unitarily equivalent to the lines of Example~\ref{ex.Paley conference}, where either
\begin{itemize}
\item[(a)]
$\mathscr{L}$ is real and $q\equiv1\bmod4$, or
\item[(b)]
$\mathscr{L}$ is not real and $3 < q\equiv3\bmod4$.
\end{itemize}
\end{theorem}

In general, the Schur cover of $\operatorname{PSL}(m,q)$ is $\operatorname{SL}(m,q)$, with a few exceptions, which we address with the following lemmas.

\begin{lemma}[Computer-assisted result]
\label{lem.PSL29}
Let $\mathscr{L}$ be a sequence of $n=10 \geq 2d > 2$ lines with span $\mathbb{C}^d$ and $\operatorname{Aut} \mathscr{L} \geq \operatorname{PSL}(2,9)$ as in Proposition~\ref{prop.classification}(ii).
Then $\mathscr{L}$ is real and unitarily equivalent to the lines of Example~\ref{ex.Paley conference} with $q = 9$ and $d = 5$.
\end{lemma}

\begin{proof}[Computer-assisted proof]
We used GAP~\cite{GAP} to find a Schur covering $\pi \colon G^* \to G$ of $G=\operatorname{PSL}(2,9)$.
Denote $G_1 \leq G$ for the stabilizer of a point, and $G_1^* = \pi^{-1}(G_1)$.
For each nontrivial linear character $\alpha \colon G_1^* \to \mathbb{T}$ we applied Theorem~\ref{thm.radicalization to Higman} to determine if the radicalization of $(G^*,G_1^*,\alpha)$ was a Higman pair (the trivial character being handled by Lemma~\ref{lem.alpha=1 is dumb}).
Exactly one choice of $\alpha$ led to a Higman pair in this way.
For that choice of $\alpha$, we found a key using Lemma~\ref{lem.key finder} and then computed the roux parameters using Proposition~\ref{prop.higman's roux}.
An application of Proposition~\ref{prop.real line detector} showed that the resulting lines were real.
Proposition~\ref{prop:real2tran} then implied that these lines were unique up to equivalence, from which it followed that they were constructed in Example~\ref{ex.Paley conference}.
(Our code is available online~\cite{github}.)
By Theorem~\ref{thm.radicalization}, this procedure captures every sequence $\mathscr{L}$ of $n=10 > d$ lines spanning $\mathbb{C}^d$ for which $\operatorname{Aut} \mathscr{L}$ contains $\operatorname{PSL}(2,9)$ in its doubly transitive action on 1-dimensional subspaces of $\mathbb{F}_9^2$.
\end{proof}

\begin{lemma}[Computer-assisted result]
\label{lem.PSL weirdos}
For each of
\[
(m,q)\in\{(2,4),(3,2),(4,2),(3,3),(3,4)\},
\]
there does not exist a sequence $\mathscr{L}$ of $n = \frac{q^m-1}{q-1} \geq 2d > 2$ doubly transitive lines with span $\mathbb{C}^d$ and $\operatorname{Aut} \mathscr{L} \geq\operatorname{PSL}(m,q)$ as in Proposition~\ref{prop.classification}(ii).
\end{lemma}

\begin{proof}[Computer-assisted proof]
The action of $\operatorname{PSL}(2,4)\cong A_5$ on $n=5$ points is triply transitive, so this case was eliminated by Lemma~\ref{lem.3tran}.
For every other choice of $(m,q)$ above, we proceeded as in the proof of Lemma~\ref{lem.PSL29}.
In every such case, there were no nontrivial characters $\alpha \colon G_1^* \to \mathbb{T}$ for which the radicalization of $(G^*,G_1^*,\alpha)$ was a Higman pair.
(Our code is available online~\cite{github}.)
\end{proof}

While the following result may be known, we were not able to locate a reference.
Hence, we supply our own proof.

\begin{proposition}
\label{prop.SL stabilizer}
Consider the action of $\operatorname{SL}(m,q)$ on the one-dimensional subspaces of $\mathbb{F}_q^m$, where matrices act on column vectors from the left.
Then the stabilizer of the line spanned by $[1,0,\ldots,0]^\top\in\mathbb{F}_q^m$ is
\[
G^*_1
=\left\{
\begin{bmatrix}
(\det M)^{-1}&u^\top\\0&M
\end{bmatrix}
: u\in\mathbb{F}_q^{m-1},M\in\operatorname{GL}(m-1,q)
\right\}.
\]
Furthermore, if $(m,q) \notin \{(2,2),(2,3), (3,2)\}$ then the linear characters $\alpha \colon G_1^* \to \mathbb{T}$ are in one-to-one correspondence with the characters $\alpha' \in\widehat{\mathbb{F}_q^\times}$ through the relation
\[
\alpha\left(\begin{bmatrix}
a&u^\top\\0&M
\end{bmatrix}\right)
=\alpha'(a).
\]
\end{proposition}

\begin{proof}
It is easy to see that $G_1^*$ is the stabilizer of the line spanned by $[1,0,\ldots,0]^\top\in\mathbb{F}_q^m$.
For the other statement, consider the epimorphism $\varphi \colon G_1^* \to \mathbb{F}_q^\times$ given by $\varphi\left( \begin{bmatrix} a & u^\top \\ 0 & M \end{bmatrix} \right) = a = (\det M)^{-1}$,
with
\[ \ker \varphi =\left\{
\begin{bmatrix}
1&u^\top\\0&M
\end{bmatrix}
:u\in\mathbb{F}_q^{m-1},M\in\operatorname{SL}(m-1,q)
\right\}.
\]
Every $\alpha' \in \widehat{ \mathbb{F}_q^\times }$ produces a character $\alpha := \alpha' \circ \varphi$ of $G_1^*$, and the mapping $\alpha' \mapsto \alpha$ is injective since $\varphi$ is surjective.
To complete the proof, it suffices to show that the commutator subgroup $[G_1^*,G_1^*]$ contains (hence equals) $\ker \varphi$;
in that case the kernel of every character $\alpha \colon G_1^* \to \mathbb{T}$ contains $[G_1^*, G_1^*] \geq \ker \varphi$, and so $\alpha$ factors through $\varphi$.

We may assume that $m \geq 2$.
To begin, take any $M,N \in \operatorname{SL}(m-1,q)$ and observe that $\xi = \begin{bmatrix} 1 & 0^\top \\ 0 & M \end{bmatrix}$ and $\eta = \begin{bmatrix} 1 & 0^\top \\ 0 & N \end{bmatrix}$ have commutator
$\xi \eta \xi^{-1} \eta^{-1} = \begin{bmatrix} 1 & 0^\top \\ 0 & MNMN^{-1} \end{bmatrix}$.
Since $(m,q) \neq (3,2)$, the group $\operatorname{SL}(m-1,q)$ is perfect (cf.\ Ch.~XIII of~\cite{Lang:02}).
Consequently, $[G_1^*,G_1^*]$ contains every matrix $\begin{bmatrix} 1 & 0^\top \\ 0 & M \end{bmatrix}$ with $M \in \operatorname{SL}(m-1,q)$.

Since $\begin{bmatrix} 1 & 0^\top \\ 0 & M \end{bmatrix} \begin{bmatrix} 1 & u^\top \\ 0 & I \end{bmatrix} = \begin{bmatrix} 1 & u^\top \\ 0 & M \end{bmatrix}$, it remains only to show that $[G_1^*,G_1^*]$ contains every matrix $\begin{bmatrix} 1 & u^\top \\ 0 & I \end{bmatrix}$ with $u \in \mathbb{F}_q^{m-1}$.
To that end, we claim there exists $M \in \operatorname{GL}(m-1,q)$ such that $\det M = a^{-1}$ is not an eigenvalue of $M^{-1}$.
Assume this is true for the moment.
Then $M^{-1} - a^{-1} I$ is invertible.
Given $u \in \mathbb{F}_q^{m-1}$, we can find $v \in \mathbb{F}_q^{m-1}$ such that 
\[ u^\top = av^\top(M^{-1} - a^{-1} I) = av^\top M^{-1} - v^\top. \]
Then $\xi = \begin{bmatrix} a & 0^\top \\ 0  & M \end{bmatrix}$ and $\eta = \begin{bmatrix} 1 & v^\top \\ 0 & I \end{bmatrix}$ have commutator 
\[ \xi \eta \xi^{-1} \eta^{-1} = \begin{bmatrix} 1 & -v^\top + av^\top M^{-1} \\ 0 & I \end{bmatrix} = \begin{bmatrix} 1 & u^\top \\ 0 & I \end{bmatrix}, \]
as desired.

To prove the claim, we find an invertible matrix $N = M^{-1}$ not having $(\det N)^{-1} = \det M$ as an eigenvalue.
If $m-1 = 1$, then $q > 3$ by assumption, and so we can choose $N = b \in \mathbb{F}_q^\times \setminus \{ \pm 1\}$.
If $m-1 = 2$, we take $N = \begin{bmatrix} 0 & 1 \\ 1 & -1 \end{bmatrix}$.
Finally, if $m-1 \geq 3$, we take $N$ to have entries $N_{ij} = 1$ if $j \geq i$ or $(i,j) = (m-1,1)$, and $N_{ij} = 0$ otherwise.
In other words, $N$ has all 1's on and above the diagonal, and all 0's below the diagonal, except for a 1 in the bottom-left corner.
For example, $N = \begin{bmatrix} 1 & 1 & 1 \\ 0 & 1 & 1 \\ 1 & 0 & 1 \end{bmatrix}$ when $m-1=3$.
A cofactor expansion along the bottom row shows that $\det N = 1$ (when the first column and bottom row are deleted, the remaining matrix has a repeated column).
Meanwhile, the columns of $N-I$ are linearly independent, and so $1$ is not an eigenvalue of $N$.
\end{proof}

\begin{lemma}
For each $m > 2$ and prime power $q$, there does not exist a sequence $\mathscr{L}$ of $n = \frac{q^m-1}{q-1} \geq 2d > 2$ doubly transitive lines with span $\mathbb{C}^d$ and $\operatorname{Aut} \mathscr{L} \geq \operatorname{PSL}(m,q)$ as in Proposition~\ref{prop.classification}(ii).
\end{lemma}

\begin{proof}
We may assume that $(m,q)$ is not among the cases handled by Lemma~\ref{lem.PSL weirdos}.
Then the quotient map $\pi \colon \operatorname{SL}(m,q) \to \operatorname{PSL}(m,q)$ is a Schur covering, by Theorem~7.1.1 of~\cite{Karpilovsky:87}.
As in Proposition~\ref{prop.SL stabilizer}, the pre-image of a point stabilizer $G_1 \leq \operatorname{PSL}(m,q)$ is
\[ G_1^* := \pi^{-1}(G_1) =\left\{
\begin{bmatrix}
(\det M)^{-1}&u^\top\\0&M
\end{bmatrix}
: u\in\mathbb{F}_q^{m-1},M\in\operatorname{GL}(m-1,q)
\right\}.
\]
Choose any nontrivial linear character $\alpha \colon G_1^* \to \mathbb{T}$.
We will apply Theorem~\ref{thm.radicalization to Higman} to show that the radicalization of $(\operatorname{SL}(m,q),\, G_1^*,\, \alpha)$ is not a Higman pair.

By Proposition~\ref{prop.SL stabilizer} there is a nontrivial character $\alpha' \in \widehat{ \mathbb{F}_q^\times}$ such that $\alpha\left(\begin{bmatrix}
a&u^\top\\0&M
\end{bmatrix}\right)
=\alpha'(a)$ for every $\begin{bmatrix}
a&u^\top\\0&M
\end{bmatrix} \in G_1^*$.
Choose any $a \in \mathbb{F}_q^\times$ for which $\alpha'(a) \neq 1$, and any $N \in \operatorname{GL}(m-2,q)$ such that $\det N = a^{-1}$.
Put
\[ M := \begin{bmatrix} a & 0^\top \\ 0 & N \end{bmatrix} \in \operatorname{SL}(m-1,q), \qquad 
\xi := \begin{bmatrix} 1 & (1-a)\delta_1^\top \\ 0 & M \end{bmatrix} \in G_1^*, \qquad
x := \begin{bmatrix} 1 & 0^\top \\ \delta_1 & I \end{bmatrix} \in \operatorname{SL}(m,q) \setminus G_1^*,
\]
where $\delta_1 := [1,0,\dotsc,0]^\top \in \mathbb{F}_q^{m-1}$.
Then a straightforward calculation shows that
\[ x\xi x^{-1} 
=\begin{bmatrix} 1 & 0^\top \\ \delta_1 & I \end{bmatrix}
\begin{bmatrix} 1 & (1-a)\delta_1^\top \\ 0 & M \end{bmatrix}
\begin{bmatrix} 1 & 0^\top \\ -\delta_1 & I \end{bmatrix}
=\begin{bmatrix} a & (1-a)\delta_1^\top \\ 0 & (1-a)\delta_1^\top \delta_1 + M \end{bmatrix}
\in G_1^*,
\]
and yet $\alpha(x\xi x^{-1}) = \alpha'(a) \neq 1 = \alpha(\xi)$.
By Theorem~\ref{thm.radicalization to Higman}, the radicalization of $(\operatorname{SL}(m,q),\, G_1^*,\, \alpha)$ is not a Higman pair.
The desired result now follows from Theorem~\ref{thm.radicalization} and Lemma~\ref{lem.alpha=1 is dumb}.
\end{proof}

\begin{proof}[Proof of Theorem~\ref{thm.linear groups}]
We may assume that $m=2$ and $q \notin \{ 4,9 \}$, since all other cases were handled above.
Then the quotient map $\pi \colon \operatorname{SL}(2,q) \to \operatorname{PSL}(2,q)$ is a Schur covering, by Theorem~7.1.1 of~\cite{Karpilovsky:87}.
According to Proposition~\ref{prop.SL stabilizer}, the pre-image of a point stabilizer $G_1 \leq \operatorname{PSL}(2,q)$ is
\[ G_1^* := \pi^{-1}(G_1) = \left\{ \begin{bmatrix} a & b \\ 0 & a^{-1} \end{bmatrix} : a \in \mathbb{F}_q^\times,\, b \in \mathbb{F}_q \right\}, \]
and the linear characters $\alpha \colon G_1^* \to \mathbb{T}$ are in one-to-one correspondence with the characters $\alpha' \in \widehat{ \mathbb{F}_q^\times }$ through the relation $\alpha\left( \begin{bmatrix} a & b \\ 0 & a^{-1} \end{bmatrix} \right) = \alpha'(a)$.

Choose any $\alpha' \in \widehat{ \mathbb{F}_q^\times }$, and let $\alpha \colon G_1^* \to \mathbb{T}$ be the corresponding character of $G_1^*$.
To begin, we apply the Higman pair detector (Theorem~\ref{thm.radicalization to Higman}) to show that the radicalization of $( \operatorname{SL}(2,q),\, G_1^*,\, \alpha)$ is a Higman pair if and only if $\alpha'$ is real valued.
Take $x := \begin{bmatrix} 0 & 1 \\ -1 & 0 \end{bmatrix} \in \operatorname{SL}(2,q) \setminus G_1^*$.
To identify which choices of $\xi := \begin{bmatrix} a & b \\ 0 & a^{-1} \end{bmatrix} \in G_1^*$ satisfy the condition $x \xi x^{-1} \in G_1^*$, we compute $x \xi x^{-1} = \begin{bmatrix} a^{-1} & 0 \\ -b & a \end{bmatrix}$.
Thus, $x\xi x^{-1} \in G_1^*$ if and only if $b = 0$, and for such $\xi$ we have $\alpha(x\xi x^{-1}) = \alpha'(a^{-1}) = \overline{ \alpha'(a) } = \overline{\alpha(\xi)}$.
Comparing $\alpha(x \xi x^{-1})$ and $\alpha(\xi)$, we find that condition~(b) of Theorem~\ref{thm.radicalization to Higman} holds if and only if $\alpha$ (or equivalently, $\alpha'$) is real valued.
The claim now follows from Theorem~\ref{thm.radicalization to Higman}.

If $q$ is even, the only real-valued character of $\mathbb{F}_q^\times$ is trivial.
By Lemma~\ref{lem.alpha=1 is dumb} and Theorem~\ref{thm.radicalization}, we conclude that no sequence of $n = q+1 \geq 2d > 2$ lines spanning $\mathbb{C}^d$ has automorphism group containing $\operatorname{PSL}(2,q)$ when $q$ is even.

For the remainder of the proof, we assume that $q$ is odd.
Let $\alpha' \in \widehat{ \mathbb{F}_q^\times }$ be the unique nontrivial real-valued character.
Explicitly, $\alpha'(a) = 1$ if $a$ is a quadratic residue and $-1$ otherwise.
Let $\alpha \colon G_1^* \to \mathbb{T}$ be the linear character corresponding to $\alpha'$.
We begin by finding a key for the radicalization of $( \operatorname{SL}(2,q),\, G_1^*,\, \alpha)$ via Lemma~\ref{lem.key finder}.
Let $x := \begin{bmatrix} 0 & 1 \\ -1 & 0 \end{bmatrix}$, as above.
Then $\xi := I$ and $\eta := -I$ satisfy $\xi x \eta = -x = x^{-1}$ and
\[ \alpha(\xi \eta) = \alpha'(-1)=
\begin{cases}
1 & \text{if } q \equiv 1 \bmod 4, \\
-1 & \text{if } q \equiv 3 \bmod 4.
\end{cases}
\]
Put $z :=1$ when $q \equiv 1 \bmod 4$, and $z:=\mathrm{i}$ when $q \equiv 3 \bmod 4$.
Then $(x,z)$ is a key for the radicalization of $( \operatorname{SL}(2,q),\, G_1^*,\, \alpha)$, by Lemma~\ref{lem.key finder}.

When $q\equiv1\bmod4$, the resulting lines are real by Lemma~\ref{lem.radicalization's roux}.
By Proposition~\ref{prop:real2tran}, these lines are unique, and so they are constructed in Example~\ref{ex.Paley conference}.
For the case $q\equiv3\bmod4$, the resulting lines are not real by Proposition~\ref{prop:real2tran}.
We proceed by constructing the roux $B\in\mathbb{C}[C_4]^{n\times n}$ using Lemma~\ref{lem.radicalization's roux}.
Following the notation in Example~\ref{ex.Paley conference}, let $[\cdot,\cdot] : \mathbb{F}_q^2 \times \mathbb{F}_q^2 \to \mathbb{F}_q$ be the symplectic form given by $\Bigl[ [ a,b]^\top, [c,d]^\top \Bigr] = ad - bc$, and take $t_a:=[a,1]^\top$ for each $a\in\mathbb{F}_q$ and $t_\infty:=[1,0]^\top$.
We denote $X:=\mathbb{F}_q \cup \{ \infty \}$, so that $\{ t_i \}_{i \in X}$ is a full set of representatives for lines through the origin of $\mathbb{F}_q^2$.
Here, $G^*_1\leq G^*$ is the stabilizer of $\operatorname{span}\{t_\infty\}$.
For each $i\in X$, we select a representative $x_i$ of the coset of $G^*_1$ that sends $\operatorname{span}\{t_\infty\}$ to $\operatorname{span}\{t_i\}$.
To this end, it suffices to take $x_a:=\left[\begin{array}{cc}a&-1\\1&0\end{array}\right]$ for each $a\in\mathbb{F}_q$, and $x_\infty:=I$.
Then  $\{x_i\}_{i\in X}$ is a full set of left coset representatives for $G^*_1$ in $G^*$.
Choose any $i,j\in X$ with $i\neq j$.
We apply Lemma~\ref{lem.radicalization's roux} in cases to show $B_{ij} = z \alpha'([t_i,t_j])$.
In the case where $i=\infty$, then $\xi:=\left[\begin{array}{cc}-1&-j\\0&-1\end{array}\right]$ and $\eta:=I$ satisfy $\xi x \eta = x_j = x_i^{-1}x_j$.
Then $\alpha(\xi \eta) = -1$, and so Lemma~\ref{lem.radicalization's roux} gives $B_{\infty j}=\alpha(\xi\eta)z^{-1}=z = z \alpha'([t_\infty, t_j])$.
Next, if $j=\infty$, then $B_{i\infty}=B_{\infty i}^{-1}=z^{-1} = z \alpha'([t_i, t_\infty])$.
Otherwise, we have $i,j\in\mathbb{F}_q$, in which case $\xi:=\left[\begin{array}{cc}(j-i)^{-1}&-1\\0&j-i\end{array}\right]$ and $\eta:=\left[\begin{array}{cc}1&(i-j)^{-1}\\0&1\end{array}\right]$ satisfy $\xi x \eta = \left[\begin{array}{cc}1&0\\i-j&1\end{array}\right] = x_i^{-1}x_j$.
Then $\alpha(\xi\eta)=-\alpha'(i-j)$, and so Lemma~\ref{lem.radicalization's roux} gives $B_{ij}=\alpha(\xi\eta)z^{-1}=z\alpha'(i-j)=z\alpha'([t_i,t_j])$.
This proves the claim.

A signature matrix of our lines is given by $\hat\beta(B)$ for some character $\beta\colon C_4\to\mathbb{T}$.
Furthermore, $\beta$ is not real valued since the lines are not real.
When $\beta$ is the identity character, $\hat\beta(B)$ is exactly the signature matrix $\mathcal{S}$ given in Example~\ref{ex.Paley conference}.
Otherwise, $\beta$ is defined by $\beta(w)=\overline{w}$, in which case $\hat\beta(B)=\overline{\mathcal{S}}$.
In what follows, we show that $\overline{\mathcal{S}}$ is equivalent to $\mathcal{S}$ up to switching and permutation; then Proposition~\ref{prop:switching equiv is unitary equiv} implies the underlying line sets are equivalent, too.
Let $\sigma\colon X\to X$ be the permutation given by $\sigma(\infty)=\infty$ and $\sigma(a)=-a$ for $a\in\mathbb{F}_q$, define $\omega_\infty:=-1$ and $\omega_a:=1$ for $a\in\mathbb{F}_q$, and take $F:=\left[\begin{array}{cc}-1&0\\0&1\end{array}\right]$.
Then for any $i,j\in X$ with $i\neq j$, we have
\[
\omega_i\omega_j\mathcal{S}_{\sigma(i)\sigma(j)}
=\omega_i\omega_jz\alpha'([t_{\sigma(i)},t_{\sigma(j)}])
=z\alpha'([Ft_{i},Ft_{j}])
=z\operatorname{det}(F)\alpha'([t_{i},t_{j}])
=\overline{z\alpha'([t_{i},t_{j}])}
=\overline{\mathcal{S}_{ij}}.
\]
This completes the proof.
\end{proof}

\subsection{Unitary groups}
\label{subsec:unitary}

In this subsection, we fix a prime power $q>2$, which may be even or odd.
Throughout, matrices act on column vectors from the left.
Take $\operatorname{GU}(3,q) \leq \operatorname{GL}(3,q^2)$ to stabilize the Hermitian form $(u,v) = u_1 v_3^q + u_2 v_2^q + u_3 v_1^q$ on $\mathbb{F}_{q^2}^3$, let $\operatorname{SU}(3,q) \leq \operatorname{GU}(3,q)$ be the subgroup of matrices with determinant~$1$, and let $\operatorname{PSU}(3,q)$ be the quotient of $\operatorname{SU}(3,q)$ by its scalar subgroup.
We next consider the doubly transitive action of $\operatorname{PSU}(3,q)$ on isotropic lines in $\mathbb{F}_{q^2}^3$, that is, on the lines spanned by $u \in \mathbb{F}_{q^2}^3 \setminus \{0 \}$ with $(u,u) = 0$.
There are $n = q^3 + 1$ such lines.
Explicitly, take $X:=T\cup\{\infty\}$, where
\[
T:=\{[a,b]^\top\in\mathbb{F}_{q^2}^2:a^{q+1}+b+b^q=0\},
\]
and define the vectors $t_\infty:=[1,0,0]^\top$ and 
$t_{a,b}:=[b,a,1]^\top$ 
for $[a,b]^\top\in T$.
Then $\{t_i\}_{i\in X}$ gives a full set of representatives for the $q^3+1$ isotropic lines in $\mathbb{F}_{q^2}^3$.
The theorem below constructs doubly transitive lines in $\mathbb{C}^{q^2-q+1}$ that are indexed by $X$ and have $\operatorname{PSU}(3,q)$ in their automorphism group.
These lines can be described in terms of a roux over the unique subgroup $\mathbb{T}_q\leq\mathbb{F}_{q^2}^\times$ of order $q+1$.

\begin{theorem}
\label{thm.updated unitary groups}
With notation as above, define $B\in\mathbb{C}[\mathbb{T}_q]^{X\times X}$ by
\[
B_{ij}=
\begin{cases}
-(t_i,t_j)^{q-1} & \text{if } i\neq j,\\
0 & \text{if }i=j.
\end{cases}
\]
Then the following hold:
\begin{itemize}
\item[(a)]
$B$ is a roux with parameters $c_1=q-1$ and $c_a=q^2-1$ for  $a\in\mathbb{T}_q\setminus\{1\}$.
\item[(b)]
For each nontrivial character $\beta\colon\mathbb{T}_q\to\mathbb{T}$, it holds that $-\hat\beta(B)\in\mathbb{C}^{X\times X}$ is the signature matrix of a $d\times n$ ETF, where $d=q^2-q+1$ and $n=q^3+1$.
Furthermore, the lines $\mathscr{L}_\beta$ spanned by these ETF vectors have $\operatorname{Aut}\mathscr{L}_\beta\geq\operatorname{PSU}(3,q)$, and exactly one of the following holds:
\begin{itemize}
\item[(i)]
The lines are real and $\beta$ is real valued.
\item[(ii)]
The lines are not real and $\beta$ takes non-real values.
\end{itemize}
\item[(c)]
For any dimension $d'$, any sequence $\mathscr{L}$ of $n\geq 2d'>2$ lines in $\mathbb{C}^{d'}$ for which $\operatorname{Aut}\mathscr{L}\geq\operatorname{PSU}(3,q)$ 
arises as in (b).
In particular, $d' = d$.
\end{itemize}
\end{theorem}

We prove Theorem~\ref{thm.updated unitary groups} at the end of this subsection.
Different choices of the character $\beta$ in Theorem~\ref{thm.updated unitary groups}(b) may or may not give equivalent lines, as detailed in Theorem~\ref{thm:unitary equivalences} below.

\begin{remark}
As explained in~\cite{IversonM:18}, every roux $B$ yields a \textbf{roux graph} with adjacency matrix $\lceil B \rfloor$.
In Theorem~\ref{thm.updated unitary groups}, the roux graph of $B$ is a (cyclic) distance regular antipodal $(q+1)$-fold cover of $K_{q^3+1}$ with the property that vertices at distance two have exactly $q^2-1$ neighbors in common.
This follows from Theorem~4.2 in~\cite{IversonM:18}; in the terminology of that paper, the graph is a $(q^3+1,q+1,q^2-1)$-\textsc{drackn}.
Furthermore, one can show that this graph admits an arc-transitive group of automorphisms isomorphic to $\operatorname{SU}(3,q)$.
By a recent theorem of Tsiovkina~\cite[Theorem~5.11]{Tsiovkina:22}, there is only one such graph up to isomorphism.
It was constructed by Godsil~\cite{Godsil:92} and rediscovered by Fickus, et.~al, in the context of equiangular lines~\cite{FickusJMPW:17}.
When $q$ is odd, the construction of this graph in terms of Theorem~\ref{thm.updated unitary groups} is new, as far as the authors know.
When $q$ is even, Cameron gave an equivalent construction in Proposition~5.1 of~\cite{Cameron:91}.
For any choice of $m \in [q]$, the roux graph of $B^{\circ m}$ is also distance regular, and when $q$ or $m$ is even, our construction also amounts to that of Cameron~\cite{Cameron:91}.
When $q$ and $m$ are both odd, the difference between our construction and Cameron's amounts to the leading minus sign in the expression $-(t_i,t_j)^{q-1}$.
\end{remark}

The following is surely known.
We include a proof for the sake of completeness.

\begin{proposition}
The quotient mapping $\pi \colon \operatorname{SU}(3,q) \to \operatorname{PSU}(3,q)$ is a Schur covering.
\end{proposition}

\begin{proof}
The Schur multiplier of $\operatorname{PSU}(3,q) = \operatorname{SU}(3,q)/Z(\operatorname{SU}(3,q))$ is isomorphic to $Z(\operatorname{SU}(3,q))$, by Theorem~3 in~\cite{Griess:73}.
Since $q > 2$, $\operatorname{SU}(3,q)$ is perfect (cf.\ Theorem~11.22 in~\cite{Grove:02}), and in particular $Z(\operatorname{SU}(3,q)) \leq [\operatorname{SU}(3,q),\operatorname{SU}(3,q)]$.
\end{proof}

Our proof of Theorem~\ref{thm.updated unitary groups} requires detailed knowledge about $G_1^* := \pi^{-1}(G_1)$, where $G_1 \leq \operatorname{PSU}(3,q)$ is the stabilizer of $\operatorname{span}\{ t_\infty \}$.
In particular, we must understand all of its linear characters, as well as their interactions with the double coset structure of $G_1^*$ in $G^*:= \operatorname{SU}(3,q)$.
This information is provided below.

\begin{proposition}
\label{prop:PSU stabilizer}
Consider the action of $\operatorname{SU}(3,q)$ on isotropic lines in $\mathbb{F}_{q^2}^3$.
The stabilizer of the line spanned by $t_\infty \in\mathbb{F}_{q^2}^3$ is $G_1^* = K \ltimes N$, where
\[ 
K := \{ \eta_e : e \in \mathbb{F}_{q^2}^\times \},
\qquad
N := \{ \xi_{a,b} : a,b \in \mathbb{F}_{q^2},\, a^{q+1} + b + b^q = 0\},
\]
with $\eta_{e}$ and $\xi_{a,b}$ given by
\[
\eta_e:=\begin{bmatrix}e & 0 & 0 \\ 0 & e^{q-1} & 0 \\ 0 & 0 & e^{-q} \end{bmatrix},
\qquad
\xi_{a,b}:=\begin{bmatrix}1 & a & b \\ 0 & 1 & -a^q \\ 0 & 0 & 1 \end{bmatrix}
\qquad
(e \in \mathbb{F}_{q^2}^\times, a^{q+1}+b+b^q = 0).
\]
Furthermore, the linear characters $\alpha\colon G_1^* \to \mathbb{T}$ are in one-to-one correspondence with the characters $\alpha'\in\widehat{\mathbb{F}_{q^2}^\times}$ through the relation $\alpha(\eta_e \xi_{a,b})=\alpha'(e)$.
\end{proposition}

We were not able to locate a reference for Proposition~\ref{prop:PSU stabilizer}, so we supply our own proof.

\begin{proof}
Notice that $Z(\operatorname{SU}(3,q)) = \{ \eta_e : e \in \mathbb{F}_{q^2},\, e^{q+1} = e^3 = 1 \}$.
O'Nan~\cite{ONan:72} establishes that $G_1^* \leq \operatorname{SU}(3,q)$ is the stabilizer of the line spanned by $t_\infty$,
and that $G_1^* = K \ltimes N$.
The latter are subgroups, and the group operations for $K$, $N$, and $G_1^*$ follow the basic identities 
\begin{equation}
\label{eq:PSUrels}
\eta_e \eta_h = \eta_{eh},
\qquad
\xi_{a,b} \xi_{f,g} = \xi_{a+f, b+g-af^q},
\qquad
\xi_{a,b}^{-1} = \xi_{-a,b^q},
\qquad
\eta_e \xi_{a,b} \eta_e^{-1} = \xi_{e^{2-q}a,e^{q+1}b}.
\end{equation}
Identifying $G_1^*/N \cong K$ with $\mathbb{F}_{q^2}^\times$, we obtain the epimorphism $\varphi \colon G_1^* \to \mathbb{F}_{q^2}^\times$, $\varphi(\eta_{e}\xi_{a,b}) = e$.
Consequently, every $\alpha' \in \widehat{ \mathbb{F}_{q^2}^\times }$ defines a character $\alpha := \alpha' \circ \varphi$ of $G_1^*$, and the mapping $\alpha' \mapsto \alpha$ is injective.
It remains to show that every linear character of $G_1^*$ factors through $\varphi$, or equivalently, that $[G_1^*,G_1^*]$ contains $\ker \varphi = N$.

To begin, we consider the $\mathbb{F}_q$-linear subspace $V \leq \mathbb{F}_{q^2}$ of all $b$ for which $b + b^q = 0$.
The $\mathbb{F}_q$-linear transformation $\varphi\colon \mathbb{F}_{q^2} \to \mathbb{F}_{q^2}$ given by $\varphi(a) = a - a^q$ maps into $V$, with one-dimensional kernel $\mathbb{F}_q \leq \mathbb{F}_{q^2}$.
Consequently, $\varphi$ maps $\mathbb{F}_{q^2}$ onto $V$.
Similarly, the mapping $a \mapsto a + a^q$ is a surjection of $\mathbb{F}_{q^2}$ onto $\mathbb{F}_q$.

Given $b \in V$, choose $a \in \mathbb{F}_{q^2}$ such that $a-a^q = b$.
Then find $f \in \mathbb{F}_{q^2}$ such that $f + f^q = -a^{q+1}$, and $g \in \mathbb{F}_{q^2}$ satisfying $ g+g^q = -1$.
A simple computation shows that $[\xi_{a,f}, \xi_{-1,g}] = \xi_{0,a-a^q} = \xi_{0,b}$.
Therefore, $[G_1^*,G_1^*]$ contains the group $\{ \xi_{0,b} : b \in V\}$.

Finally, let $\xi_{a,b} \in N$ be arbitrary.
Notice that $\lambda^{2-q} - 1 \neq 0$ since $q^2 - 1 > 2q - 1$ by the assumption $q > 2$.
Put $f := a(\lambda^{2-q} -1)^{-1}$, and find $g \in \mathbb{F}_{q^2}$ such that $g+g^q = -f^{q+1}$.
Then one computes $[\eta_\lambda, \xi_{f,g}] = \xi_{f(\lambda^{2-q}-1),h} = \xi_{a,h}$ for some appropriate choice of $h \in \mathbb{F}_{q^2}$, which necessarily satisfies 
\[ a^{q+1} + h + h^q = 0 = a^{q+1} + b + b^q. \]
Hence $b-h \in V$, and $[G_1^*,G_1^*]$ contains $\xi_{a,h} \xi_{0,b-h} = \xi_{a,b}$.
\end{proof}

\begin{lemma}
\label{lem:PSUdcsts}
Adopt the notation of Proposition~\ref{prop:PSU stabilizer}, and put $x := \begin{bmatrix}
0 & 0 & 1 \\
0 & -1 & 0 \\
1 & 0 & 0
\end{bmatrix}$.
Let $\zeta \in G_1^*$ be arbitrary.
If $\zeta \in K$, then $x\zeta x^{-1} \in K \leq G_1^*$.
If $\zeta \notin K$, then $x \zeta x^{-1} \in G_1^* x G_1^*$.

Furthermore, suppose $\alpha' \in \widehat{ \mathbb{F}_{q^2}^\times}$ satisfies $\operatorname{im} \alpha' \leq C_{q+1}$, and let $\alpha$ be the corresponding linear character of $G_1^*$.
If $\zeta = \eta_e \xi_{a,b} \notin K$, then every decomposition $x \zeta x^{-1} = \xi x \eta$ with $\xi,\eta \in G_1^*$ satisfies $\alpha(\xi \eta \zeta^{-1}) = \alpha'(b)$.
\end{lemma}

\begin{proof}
First, we compute
\begin{equation}
\label{eq:PSUrelx}
x \eta_e x^{-1}
=
\eta_{e^{-q}},
\qquad
x \xi_{a,b} x^{-1} = \begin{bmatrix} 1 & 0 & 0 \\ a^q & 1 & 0 \\ b & -a & 1 \end{bmatrix}
\qquad
(e \neq 0, a^{q+1} + b + b^q = 0).
\end{equation}
Therefore $xKx^{-1} = K$, and $x(\eta_e \xi_{a,b})x^{-1} = (x\eta_ex^{-1}) (x\xi_{a,b}x^{-1}) \in G_1^*$ if and only if $a=b=0$.
Since $G^* := \operatorname{SU}(3,q)$ acts doubly transitively on $G^*/G_1^*$, there are just two double cosets of $G_1^*$ in $G^*$.
When $\zeta \notin K$ we have $x \zeta x^{-1} \notin G_1^*$, and so $x \zeta x^{-1} \in G_1^* x G_1^*$.

Now let $\alpha' \in \mathbb{F}_{q^2}^\times$ be such that $\operatorname{im} \alpha' \leq C_{q+1}$.
Suppose $\zeta = \eta_e \xi_{a,b} \notin K$, and $x\zeta x^{-1} = (\eta_{h} \xi_{f,g} ) x (\eta_{k} \xi_{l,m})$ for some $\eta_h \xi_{f,g}$, $\eta_k \xi_{l,m} \in G_1^*$.
Using~\eqref{eq:PSUrels} and~\eqref{eq:PSUrelx}, we rewrite
\[
\eta_{h^{-q} k^{-1}} \zeta
=
x^{-1} \xi_{k^{2q-1} f, k^{q+1} g} x \xi_{l,m} x.
\]
The left side of this equation is
\[ \eta_{h^q k^{-1}} \zeta 
=
\eta_{h^q k^{-1} e} \xi_{a,b}
=
\begin{bmatrix}
h^q k^{-1} e & h^q k^{-1} a & h^q k^{-1} e b \\
0 & h^{1-q} k^{1-q} e^{q-1} & -h^{1-q} k^{1-q} e^{q-1} a^q \\
0 & 0 & h^{-1} k^q e^{-q}
\end{bmatrix},
\]
while the right side takes the form
\[
\begin{bmatrix} 0 & 0 & 1 \\ 0 & -1 & 0 \\ 1 & 0 & 0 \end{bmatrix}
\begin{bmatrix} 1 & \ast & \ast \\ 0 & 1 & \ast \\ 0 & 0 & \ast \end{bmatrix}
\begin{bmatrix} 0 & 0 & 1 \\ 0 & -1 & 0 \\ 1 & 0 & 0 \end{bmatrix}
\begin{bmatrix} 1 & \ast & \ast \\ 0 & 1 & \ast \\ 0 & 0 & \ast \end{bmatrix}
\begin{bmatrix} 0 & 0 & 1 \\ 0 & -1 & 0 \\ 1 & 0 & 0 \end{bmatrix}
=
\begin{bmatrix} \ast & \ast & 1 \\ \ast & \ast & \ast \\ \ast & \ast & \ast \end{bmatrix},
\]
where $\ast$ marks (possibly different) entries we have not computed.
Comparing the top-right entries, we see that
$b = h^{-q} k e^{-1}$.
We have $\alpha'(h^{-q}) = \alpha'(h)$ since $\operatorname{im} \alpha' \leq C_{q+1}$, and therefore
\[ \alpha( \eta_h \xi_{f,g} \eta_k \xi_{l,m} \zeta^{-1} ) = \alpha'(hke^{-1}) = \alpha'(h^{-q} k e^{-1}) = \alpha'(b). \qedhere \]
\end{proof}

\begin{proof}[Proof of Theorem~\ref{thm.updated unitary groups}]
To begin, we find all characters $\alpha \colon G_1^* \to \mathbb{T}$ for which the radicalization of $(G^*, G_1^*, \alpha)$ is a Higman pair.
Fix a linear character $\alpha' \in \widehat{ \mathbb{F}_{q^2}^\times }$, and let $\alpha$ be the corresponding character of $G_1^*$.
Put $x := \begin{bmatrix} 0 & 0 & 1 \\ 0 & -1 & 0 \\ 1 & 0 & 0 \end{bmatrix}$.
Given $\xi \in G_1^*$, we have $x \xi x^{-1} \in G_1^*$ if and only if $\xi = \eta_e$ for $e \neq 0$, in which case $x \eta_e x^{-1} = \eta_{e^{-q}}$.
By Theorem~\ref{thm.radicalization to Higman}, the radicalization $(\widetilde{G}^*,H)$ of $(G^*, G_1^*, \alpha)$ is a Higman pair if and only if $\alpha'(e) = \alpha'(e)^{-q}$ for every $e \neq 0$, if and only if $\operatorname{im} \alpha' \leq C_{q+1}$.

Now fix a choice of $\alpha$ for which $\operatorname{im}\alpha'\leq C_{q+1}$, and put $r':=|\operatorname{im}\alpha'|$, $r:=2r'$.
Then the radicalization of $(G^*,G_1^*,\alpha)$ is a Higman pair.
To find a key for this Higman pair, first observe that $\mathbb{F}_q^\times \leq \operatorname{ker}\alpha'$.
Consequently, there is a unique character $\alpha''\colon\mathbb{T}_q\to C_{r'}$ for which $\alpha'(a)=\alpha''(a^{q-1})$ for each $a\in\mathbb{F}_{q^2}^\times$.
Put $z:=\alpha''(-1)\in\{\pm1\}$.
Then since $x=x^{-1}$, Lemma~\ref{lem.key finder} implies that $(x,z)$ is a key for the radicalization of $(G^*,G_1^*,\alpha)$.
With this key, Proposition~\ref{prop.higman's roux} produces a roux $\tilde{B}\in\mathbb{C}[C_r]^{X\times X}$.

We now construct the entries of $\tilde{B}$ using Lemma~\ref{lem.radicalization's roux}.
To begin, we fix a choice of left coset representatives for $G_1^*$ in $G^*$.
Since $G^*$ acts doubly transitively on isotropic lines in $\mathbb{F}_{q^2}^3$, where $G_1^*$ stabilizes $\operatorname{span}\{t_\infty\}$, it suffices to find $x_i\in G^*$ for which $x_it_\infty=t_i$ for each $i\in X$.
We take $x_\infty:=\left[\begin{array}{ccc}1&0&0\\0&1&0\\0&0&1\end{array}\right]$, and for $[a,b]^\top\in T$ we define 
$x_{a,b}:=\left[\begin{array}{ccc}b&-a^q&1\\a&1&0\\1&0&0\end{array}\right]$, 
where 
$x_{a,b}^{-1}=\left[\begin{array}{ccc}0&0&1\\0&1&-a\\1&a^q&b^q\end{array}\right]$.
Then $\{x_i\}_{i\in X}$ is a full set of left coset representatives for $G_1^*$ in $G^*$.
For any choice of $i,j\in X$ with $i\neq j$, a straightforward calculation gives that the $(3,1)$ entry of the matrix $x_i^{-1}x_j$ equals $(t_j,t_i)=(t_i,t_j)^q$.
For any $\xi,\eta\in G_1^*$ such that $\xi x \eta = x_i^{-1}x_j$, one may examine the (3,1) entry to conclude that $\alpha(\xi \eta) = \alpha( (t_i,t_j)^q) = \overline{ \alpha( (t_i,t_j) ) }$.
By Lemma~\ref{lem.radicalization's roux}, it follows that $\tilde{B}_{ij} = z \overline{ \alpha( (t_i,t_j) ) } = \overline{ \alpha''\bigl(- (t_i,t_j)^{q-1} \bigr)}$.
At this point, we observe that the entries of $\tilde{B}$ reside in the subring $\mathbb{C}[C_{r'}] \leq \mathbb{C}[C_{r}]$, and so we may view $\tilde{B}$ as a roux for either $C_{r'}$ or $C_{r}$.
Furthermore, the lines given by $\tilde{B}$ do not change when we view it as a roux over $C_{r'}$ instead of $C_r$, in light of Proposition~\ref{prop.two reps of same lines} and the fact that every character of $C_{r'}$ extends to a character of $C_r$.

To prove (a) and (b), we focus on a special choice of $\alpha$.
Namely, we fix a character $\alpha_0 \colon \mathbb{F}_{q^2}^\times \to C_{q+1}$ for which $\alpha_0'' \colon \mathbb{T}_q \to C_{q+1}$ is an isomorphism.
Take $\alpha = \alpha_0$ above, and let $\tilde{B}\in \mathbb{C}[C_{2(q+1)}]^{X \times X}$ be the resulting roux with entries $\tilde{B}_{ii} = 0$ and $\tilde{B}_{ij} =  \overline{\alpha_0''(-(t_i,t_j))}$ for $i \neq j$ in $X$.
For the moment, we view $\tilde{B}$ as a roux over $C_{2(q+1)}$, and we compute its roux parameters $\{c_w\}_{w\in C_{2(q+1)}}$.
For any $w \in C_{2(q+1)}$, Lemmas~\ref{lem.cgs of radical roux} and~\ref{lem:PSUdcsts}
combine to show that 
\begin{equation}
\label{eq:PSUcw}
c_w = 
\frac{q^3}{|G_1^*|}\cdot  \left| \{ \eta_e \xi_{a,b} \in G_1^* : b\neq 0 \text{ and } \alpha_0'(b) = z w \} \right|
= \frac{q^3(q^2 - 1)}{|K||N|}\cdot \sum_{ \substack{ b \in \mathbb{F}_{q^2}^\times, \\ \alpha'(b) = z w }} N_b
= \sum_{ \substack{ b \in \mathbb{F}_{q^2}^\times, \\ \alpha'(b) = z w }} N_b,
\end{equation}
where 
\[ N_b
=
| \{ a \in \mathbb{F}_{q^2} : a^{q+1} + b + b^q = 0 \} |
\qquad
(b \in \mathbb{F}_{q^2}^\times).
\]
Since $a\mapsto a^{q+1}$ is a surjective homomorphism of $\mathbb{F}_{q^2}^\times$ onto $\mathbb{F}_q^\times$, and since $-(b+b^q)\in\mathbb{F}_q$ for every $b\in\mathbb{F}_{q^2}$, we deduce that
\[
N_b
=\left\{\begin{array}{cl}
1&\text{if }b+b^q=0 \\
q+1&\text{else}
\end{array}\right\}
=\left\{\begin{array}{cl}
1&\text{if }b^{q-1}=-1 \\
q+1&\text{else}
\end{array}\right\}
\qquad
(b \in \mathbb{F}_{q^2}^\times).
\] 
Therefore,
\[
c_w
=\Big|\Big\{b\in\mathbb{F}_{q^2}^\times:b^{q-1}=-1\text{ and }\alpha_0'(b)=zw\Big\}\Big|
+(q+1)\Big|\Big\{b\in\mathbb{F}_{q^2}^\times:b^{q-1}\neq-1\text{ and }\alpha_0'(b)=zw\Big\}\Big|.
\]
Our choice of $z=\alpha_0''(-1)$ ensures that $\alpha_0'(b)=\alpha_0''(b^{q-1})=z$ whenever $b^{q-1}=-1$.
There are exactly $q-1$ choices of $b$ for which $b^{q-1}=-1$, and so 
\[
c_w
=(q-1)\delta_{w,1}
+(q+1)\Big|\Big\{b\in\mathbb{F}_{q^2}^\times:b^{q-1}\neq-1\text{ and }\alpha_0'(b)=zw\Big\}\Big|.
\]
For $w\not\in C_{q+1}=\operatorname{im}\alpha_0'$, we have $w\neq 1$, and there does not exist $b\in\mathbb{F}_{q^2}^\times$ for which $w=z\alpha_0'(b)$ since $z\in\operatorname{im}\alpha_0'$.
Hence, $c_w=0$ whenever $w\not\in C_{q+1}$.
Given $w\in C_{q+1}$, we have $zw\in\operatorname{im}\alpha_0'$, and there are exactly $|\operatorname{ker}\alpha_0'|=q-1$ choices of $b\in\mathbb{F}_{q^2}^\times$ for which $\alpha_0'(b)=zw$.
For $w=1$, any such $b$ satisfies $b^{q-1} = -1$ by our choice of $z$, and so
$c_1 = q-1$.
For $w \in C_{q+1} \setminus \{1\}$, any such $b$ satisfies $b^{q-1}\neq-1$, and so $c_w =q^2-1$.

To prove (a), view $\tilde{B}$ as a roux over $C_{q+1}$, and observe that its roux parameters $\{c_w\}_{w\in C_{q+1}}$ are as given above.
Indeed, $c_w = 0$ when $w \in C_{2(q+1)}\setminus C_{q+1}$, and so $\tilde{B}^2 = (n-1) I + \sum_{w\in C_{2(q+1)}} c_w w \tilde{B} = (n-1) I + \sum_{w\in C_{q+1}} c_w w \tilde{B}$.
Next, define $\gamma \colon C_{q+1} \to \mathbb{T}_q$ to be the inverse mapping of $\alpha_0''$, so that $\gamma\bigl( \overline{\alpha_0''(a)} \bigr) = a$ for every $a \in \mathbb{T}_q$.
Then $\gamma$ extends to an isomorphism $\tilde{\gamma} \colon \mathbb{C}[C_{q+1}]^{X \times X} \to \mathbb{C}[\mathbb{T}_q]^{X\times X}$, and $B:= \tilde{\gamma}(\tilde{B}) \in \mathbb{C}[\mathbb{T}_q]^{X\times X}$ is a roux with entries $B_{ii} = 0$ and $B_{ij} = \gamma( \tilde{B}_{ij} ) = - (t_i,t_j)^{q-1}$ for every $i \neq j$.
Considering the parameters for $\tilde{B}$ we found above, then $B$ has roux parameters $\{ c_a \}_{a\in \mathbb{T}_q}$ with $c_1 = q-1$ and $c_a = q^2 - 1$ for $a \in \mathbb{T}_q \setminus \{1\}$.
This proves (a).

For (b), we apply Proposition~\ref{prop.roux primitive idempotents} to compute the ranks $\{ d_\beta^\epsilon \}_{\beta \in \hat{\mathbb{T}_q}, \epsilon \in \{1,-1\}}$ of the primitive idempotents in the roux scheme given by $B$.
Given a character $\beta \colon \mathbb{T}_q \to \mathbb{T}$, we compute
\[
\hat{c}_{\beta} = \sum_{a \in \mathbb{T}_q} c_a \overline{\beta(a)}
= c_1 \overline{\beta(1)} + \sum_{a \in \mathbb{T}_q \setminus \{1\} } c_a \overline{\beta(a)}
= (q-1) + (q^2-1) \cdot \overline{ \sum_{a\in \mathbb{T}_q \setminus\{1\} } \beta(a) }.
\]
When $\beta=1$ is the trivial character, this yields $\hat{c}_1 = (q-1) + (q^2-1)q = q^3-1 = n-2$.
In Proposition~\ref{prop.roux primitive idempotents}, we obtain $\mu_1^+ = 1$ and $\mu_1^- = (1-n)^{-1}$, so that $d_1^+ = 1$ and $d_1^- = n-1$.
Now assume $\beta$ is nontrivial.
Then $\hat{c}_{\beta} = (q-1) - (q^2-1) = q - q^2$.
In Proposition~\ref{prop.roux primitive idempotents}, we obtain $\mu_\beta^+ = q^{-2}$ and $\mu_\beta^- = -q^{-1}$, so that $d_\beta^+ = q(q^2-q+1)$ and $d_\beta^- = q^2-q+1$.
In particular, Proposition~\ref{prop.two reps of same lines} implies that $- \hat{\beta}(B)$ is the signature matrix of a $d\times n$ ETF with $d = d_{\beta^{-1}}^- = q^2-q+1$ and $n = |X| = q^3+1$.
Furthermore, the lines $\mathscr{L}_\beta$ spanned by the vectors in this ETF are real if and only if $\beta$ takes real values, by Proposition~\ref{prop.real line detector}.
To see that $\operatorname{Aut} \mathscr{L}_\beta \geq \operatorname{PSU}(3,q)$, let $\sigma \colon X \to X$ be a permutation of isotropic line indices that is implemented by a unitary $U \in \operatorname{SU}(3,q)$.
Then for each $i \in X$ there exists $a_i \in \mathbb{F}_{q^2}^\times$ such that $U a_i t_i = t_{\sigma(i)}$.
Given $i,j \in X$ with $i \neq j$, it holds that 
\[
\hat{\beta}(B)_{\sigma(i),\sigma(j)} = \beta \bigl( - (t_{\sigma(i)},t_{\sigma(j)})^{q-1} \bigr)
= \beta \bigl( - (U a_i t_i, U a_j t_j)^{q-1} \bigr)
= \beta \bigl( - a_i^{q-1} a_j^{q(q-1)} (t_i, t_j)^{q-1} \bigr)
\]
\[
= \beta\bigl(a_i^{q-1}\bigr) \overline{\beta\bigl(a_j^{q-1}\bigr)} \hat{\beta}(B)_{ij}.
\]
Overall, the result of permuting the entries of $\hat{\beta}(B)$ with $\sigma$ is the same as that of applying a switching equivalence to $\hat{\beta}(B)$.
Then $\sigma \in \operatorname{Aut} \mathscr{L}_\beta$, by Proposition~\ref{prop:switching equiv is unitary equiv}.
This proves (b).

It remains to prove (c).
Let $\mathscr{L}$ be a sequence of $n \geq 2d > 2$ lines for $\mathbb{C}^d$ for which $\operatorname{Aut} \mathscr{L} \geq \operatorname{PSU}(3,q)$.
Then we can take $X$ as an indexing set for $\mathscr{L}$, and $n = |X| = q^3+1$.
By Theorem~\ref{thm.radicalization}, there is a character $\alpha \colon G_1^* \to \mathbb{T}$ for which the radicalization of $(G^*,G_1^*,\alpha)$ is a Higman pair whose Schurian scheme has a primitive idempotent that represents $\mathscr{L}$.
By the above, $\alpha(\eta_e \xi_{a,b}) = \alpha''(e^{q-1})$ for some character $\alpha'' \colon \mathbb{T}_q \to C_{q+1}$, and $\mathscr{L}$ can be obtained from a primitive idempotent in the roux scheme of $\tilde{B} \in \mathbb{C}[C_{r'}]^{X \times X}$ with entries $\tilde{B}_{ii} = 0$ and $\tilde{B}_{ij} = \overline{ \alpha''\bigl( - (t_i,t_j)^{q-1} \bigr) }$ for $i \neq j$, where $r' := | \operatorname{im} \alpha''|$.
Considering Proposition~\ref{prop.two reps of same lines}, there exists a character $\beta' \colon C_{r'} \to \mathbb{T}$ and $\epsilon \in \{1,-1\}$ such that $\epsilon\hat{\beta'}(\tilde{B})$ is the signature matrix of an ETF whose vectors span $\mathscr{L}$.
To see this has the form of (b), take $\beta \colon \mathbb{T}_q \to \mathbb{T}$ to be the character given by $\beta(a) = \beta'\bigl( \overline{ \alpha''(a) } \bigr)$ for every $a \in \mathbb{T}_q$.
Then $\epsilon \hat{\beta'}(\tilde{B}) = \epsilon \hat{\beta}(B)$.
By Proposition~\ref{prop.two reps of same lines}, $d = d_{\beta^{-1}}^\epsilon$ is as computed above.
Then the condition $n \geq 2d > 2$ ensures that $\beta$ is nontrivial and $\epsilon = -1$.
This completes the proof.
\end{proof}

\subsection{Suzuki groups}
We now consider the doubly transitive action of a Suzuki group $\operatorname{Sz}(q)$, where $q = 2^{2m+1} > 2$ as in Proposition~\ref{prop.classification}(iv).
Our primary reference is \S XI.3 in~\cite{HuppertB:82:III}.
Fix the following notation for the duration of this subsection.
Set $t = 2^{m+1}$.
Given $a,b \in \mathbb{F}_q$ and $e \in \mathbb{F}_q^\times$, put
\[
\xi_{a,b}
:=
\begin{bmatrix}
1 & 0 & 0 & 0 \\
a & 1 & 0 & 0 \\
b & a^t & 1 & 0 \\
a^{2+t} + ab + b^t & a^{t + 1} + b & a & 1
\end{bmatrix},
\qquad
\eta_e
:=
\begin{bmatrix}
e^{1+2^m} & 0 & 0 & 0 \\
0 & e^{2^m} & 0 & 0 \\
0 & 0 & e^{-2^m} & 0 \\
0 & 0 & 0 & e^{-1-2^m}
\end{bmatrix}.
\]
Also define $x \in \operatorname{GL}(4,q)$ to be the anti-diagonal matrix with entries $x_{ij} = \delta_{i,5-j}$.
By definition,
\[ \operatorname{Sz}(q) := \langle \xi_{a,b}, \eta_e, x : a,b \in \mathbb{F}_q, e \in \mathbb{F}_q^\times \rangle \leq \operatorname{GL}(4,q). \]
It has a doubly transitive permutation representation with point stabilizer
\[ G_1 := \langle \xi_{a,b}, \eta_e : a,b \in \mathbb{F}_q, e \in \mathbb{F}_q^\times \rangle. \]

\begin{theorem}
\label{thm.Suzuki groups}
There does not exist a sequence $\mathscr{L}$ of $n = q^2+1 \geq 2d > 2$ doubly transitive lines with span $\mathbb{C}^d$ and $\operatorname{Aut} \mathscr{L} \geq \operatorname{Sz}(q)$ as in Proposition~\ref{prop.classification}(iv).
\end{theorem}

By~\cite{AlperinGorenstein:66}, the Schur multiplier of $\operatorname{Sz}(q)$ is trivial with a single exception, addressed below.

\begin{lemma}[Computer-assisted result]
\label{lem.Sz8}
There does not exist a sequence $\mathscr{L}$ of $n \geq 2d > 2$ doubly transitive lines with span $\mathbb{C}^d$ and $\operatorname{Aut} \mathscr{L} \geq \operatorname{Sz}(8)$ as in Proposition~\ref{prop.classification}(iv).
\end{lemma}

\begin{proof}[Computer-assisted proof]
The Schur multiplier of $G := \operatorname{Sz}(8)$ is $\mathbb{Z}_2 \times \mathbb{Z}_2$, by Theorem~2 in~\cite{AlperinGorenstein:66}.
We used the AtlasRep~\cite{AtlasRep} package in GAP~\cite{GAP} to obtain a Schur covering $\pi \colon G^* \to G$.
Then we proceeded as in the proof of Lemma~\ref{lem.PSL29}.
No nontrivial character $\alpha \colon G_1^* \to \mathbb{T}$ produced a Higman pair through radicalization.
Consequently, any sequence of lines whose automorphism group contains $\operatorname{Sz}(8)$ spans a space of dimension $d \in \{1,n-1\}$.
(Our code is available online~\cite{github}.)
\end{proof}

We supply our own proof for the following proposition, which we could not locate elsewhere.

\begin{proposition}
\label{prop:Szstab}
The point stabilizer $G_1 \leq \operatorname{Sz}(q)$ is a semidirect product $G_1 = N \rtimes K$, where 
\[ N := \{ \xi_{a,b} : a,b \in \mathbb{F}_q \}, \qquad K := \{ \eta_e : e \in \mathbb{F}_q^\times \}. \]
Furthermore, the linear characters $\alpha\colon G_1^* \to \mathbb{T}$ are in one-to-one correspondence with the characters $\alpha'\in\widehat{\mathbb{F}_{q}^\times}$ through the relation $\alpha(\xi_{a,b} \eta_e )=\alpha'(e)$.
\end{proposition}

\begin{proof}
In~\cite[Lemma~XI.3.1]{HuppertB:82:III} it is shown that $G_1 = N \rtimes K$.
Moreover, the group operation for $G_1$ satisfies the basic identities
\[
\xi_{a,b} \xi_{f,g} = \xi_{a+f, b + g + a^t f},
\qquad
\eta_e \eta_h = \eta_{eh},
\qquad
\eta_e^{-1} \xi_{a,b} \eta_e = \xi_{ea, e^{t + 1} b}
\qquad
(a,b,f,g \in \mathbb{F}_q,\, e, h \in \mathbb{F}_q^\times). 
\]
Consequently, there is a well-defined epimorphism $\varphi \colon G_1 \to \mathbb{F}_q^\times$ given by $\varphi(\xi_{a,b} \eta_e) = e$.
Each $\alpha' \in \widehat{ \mathbb{F}_q^\times }$ determines a linear character $\alpha := \alpha' \circ \varphi \colon G_1 \to \mathbb{T}$, and the mapping $\alpha' \mapsto \alpha$ is injective.
It remains to prove that $[G_1,G_1]$ contains $\ker \varphi = N$, so that every linear character $\alpha \colon G_1 \to \mathbb{T}$ factors through $\varphi$.

Since $m > 0$ we have $2^{m+1}+1 < 2^{2m+1} - 1 = q-1$, and so there exists $e \in \mathbb{F}_q^\times$ such that $e^{t+1} \neq 1$.
Fix such an $e$.
Given any $a,b \in \mathbb{F}_q$, define 
\[
f := (e-1)^{-1}a,
\qquad
g := (e^{t + 1} - 1)^{-1}[b-(1-e^t)f^{t+1}].
\]
Then
\[
\eta_{e}^{-1} \xi_{f,g} \eta_e \xi_{a,b}^{-1}
=
\xi_{ef, e^{t + 1} g} \xi_{-f,-g+f^{t+1}}
=
\xi_{ (e-1)^{-1} f, (e^{t+1}-1)g+(1-e^t)f^{t+1}}
=
\xi_{a,b},
\]
so that $\xi_{a,b} \in [G_1,G_1]$.
\end{proof}

\begin{proof}[Proof of Theorem~\ref{thm.Suzuki groups}]
We can assume $q > 8$. 
Then the Schur muliplier is trivial by~\cite[Theorem~1]{AlperinGorenstein:66}, and so $G^* = G$ is its own Schur cover.
Let $\alpha \colon G_1 \to \mathbb{T}$ be any nontrivial linear character of $G_1$, and let $\alpha' \in \widehat{ \mathbb{F}_q^\times }$ be the corresponding character in Proposition~\ref{prop:Szstab}.
Choose any $e \in \mathbb{F}_q^\times$ such that $\alpha'(e) \neq 1$.
Since $\mathbb{F}_q^\times \cong C_{q-1}$ and $q$ is even, there exists $h \in \mathbb{F}_q^\times$ with $h^2 = e$.
Then $x \eta_h x^{-1} = \eta_{h^{-1}} \in G_1$, and yet $\alpha'(h) \neq \alpha'(h^{-1})$.
By Theorem~\ref{thm.radicalization to Higman}, the radicalization of $(G,G_1, \alpha)$ is not a Higman pair.
From Theorem~\ref{thm.radicalization} and Lemma~\ref{lem.alpha=1 is dumb} we conclude that any sequence of lines admitting $\operatorname{Sz}(q)$ as a group of automorphisms spans a space of dimension $d \in \{1,n-1\}$.
\end{proof}

\subsection{Ree groups}

Next, we find all line sequences whose automorphism group contains the Ree group ${}^2G_2(q)$, where $q = 3^{2m+1}$.
Our primary reference is~\cite[Section 3]{Baarnhielm:14}, which compiles information about Ree groups taken from \cite{HuppertB:82:III,KemperLM:01,Levchuk:85,Ward:66,Wilson:09}.
We use the following notation in this subsection.
Set $t = 3^m$.
Given $a,b,c \in \mathbb{F}_q$, put $\xi_{a,b,c} := \xi_{a,0,0} \xi_{0,b,0} \xi_{0,0,c}$, where
\[
\xi_{a,0,0}
:=
\left[ \begin{array}{ccccccc}
1 & a^t & 0 & 0 & -a^{3t+1} & -a^{3t+2} & a^{4t+2} \\
0 & 1 & a & a^{t+1} & -a^{2t+1} & 0 & -a^{3t+2} \\
0 & 0 & 1 & a^t & -a^{2t} & 0 & a^{3t+1} \\
0 & 0 & 0 & 1 & a^t & 0 & 0 \\
0 & 0 & 0 & 0 & 1 & -a & a^{t+1} \\
0 & 0 & 0 & 0 & 0 & 1 & -a^t \\
0 & 0 & 0 & 0 & 0 & 0 & 1
\end{array} \right],
\]
\[
\xi_{0,b,0}
:=
\left[ \begin{array}{ccccccc}
1 & 0 & -b^t & 0 & -b & 0 & -b^{t+1} \\
0 & 1 & 0 & b^t & 0 & -b^{2t} & 0 \\
0 & 0 & 1 & 0 & 0 & 0 & b \\
0 & 0 & 0 & 1 & 0 & b^t & 0 \\
0 & 0 & 0 & 0 & 1 & 0 & b^t \\
0 & 0 & 0 & 0 & 0 & 1 & 0 \\
0 & 0 & 0 & 0 & 0 & 0 & 1
\end{array} \right],
\qquad
\xi_{0,0,c}
:=
\left[ \begin{array}{ccccccc}
1 & 0 & 0 & -c^t & 0 & -c & -c^{2t} \\
0 & 1 & 0 & 0 & -c^t & 0 & c \\
0 & 0 & 1 & 0 & 0 & c^t & 0 \\
0 & 0 & 0 & 1 & 0 & 0 & -c^t \\
0 & 0 & 0 & 0 & 1 & 0 & 0 \\
0 & 0 & 0 & 0 & 0 & 1 & 0 \\
0 & 0 & 0 & 0 & 0 & 0 & 1
\end{array} \right].
\]
For $e \in \mathbb{F}_q^\times$, put
\[
\eta_e := \operatorname{diag}(e^t, e^{1-t}, e^{2t-1}, 1, e^{1-2t}, e^{t-1}, e^{-t}).
\]
Finally, define $x \in \operatorname{GL}(7,q)$ to be the anti-diagonal matrix with entries $x_{ij} = - \delta_{i,8-j}$.
By definition,
\[
{}^2G_2(q) := \langle \xi_{a,b,c}, \eta_e, x : a,b,c \in \mathbb{F}_q, e \in \mathbb{F}_q^\times \rangle \leq \operatorname{GL}(7,q).
\]
It acts doubly transitively on $n := q^3 + 1$ points with point stabilizer
\[
G_1 := \langle \xi_{a,b,c}, \eta_e : a,b,c \in \mathbb{F}_q, e \in \mathbb{F}_q^\times \rangle
=
N \rtimes K,
\]
where
\[
N := \left\{ \xi_{a,b,c} : a,b,c \in \mathbb{F}_q \right\},
\qquad
K := \left\{ \eta_e : e \in \mathbb{F}_q^\times \right\}.
\]
Multiplication in $G_1$ follows the basic identities
\[
\xi_{a,b,c} \xi_{f,g,h}
=
\xi_{a+f,b+g-af^{3t},c+h-fb+af^{3t+1}-a^2f^{3t}},
\qquad
\xi_{a,b,c}^{-1}
=
\xi_{-a, -b-a^{3t+1}, -c-ab+a^{3t+2} },
\]
\[
\eta_e^{-1} \xi_{a,b,c} \eta_e
=
\xi_{e^{3t-2} a, e^{1-3t} b, e^{-1} c},
\qquad
\eta_e \eta_k = \eta_{ek}.
\]

\begin{theorem}
\label{thm:Ree}
Let $\mathscr{L}$ be a sequence of $n > d$ doubly transitive lines with span $\mathbb{C}^d$ and $\operatorname{Aut} \mathscr{L} \geq {}^2G_2(q)$ as in Proposition~\ref{prop.classification}(v).
Then $\mathscr{L}$ is real, hence described by Proposition~\ref{prop:real2tran}(iv).
\end{theorem}

Our technique for proving Theorem~\ref{thm:Ree} requires $q>3$.
We dispose of the case $q=3$ by computer, as below.

\begin{lemma}
[Computer-assisted result]
Assume $q=3$.
Let $\mathscr{L}$ be a sequence of $n > d$ doubly transitive lines with span $\mathbb{C}^d$ and $\operatorname{Aut} \mathscr{L} \geq {}^2G_2(3)$ as in Proposition~\ref{prop.classification}(v).
Then $\mathscr{L}$ is real.
\end{lemma}

\begin{proof}[Computer-assisted proof]
We proceeded as in the proof of Lemma~\ref{lem.PSL29}, using GAP to compute a Schur cover $G^*$.
For every linear character $\alpha \colon G_1^* \to \mathbb{T}$, we checked to see if the radicalization of $(G^*,G_1^*,\alpha)$ was a Higman pair.
If it was, we computed roux parameters.
Then we applied the real roux lines detector of Propositions~\ref{prop.two reps of same lines} and~\ref{prop.real line detector}.
In every case, we found that the resulting lines were real.
(Our code is available online~\cite{github}.)
\end{proof}

\begin{proposition}
\label{prop:ReeStab}
Assume $q > 3$.
Then the linear characters $\alpha \colon G_1 \to \mathbb{F}_q$ of the point stabilizer $G_1 \leq {}^2G_2(q)$ are in one-to-one correspondence with the characters $\alpha' \in \mathbb{F}_q^\times$ through the relation $\alpha(\xi_{a,b,c} \eta_e) = \alpha'(e)$.
\end{proposition}

We provide a proof since we could not locate a reference.

\begin{proof}
The quotient map of $G$ onto $G/N \cong K \cong \mathbb{F}_q^\times$ gives an epimorphism $\varphi \colon G_1 \to \mathbb{F}_q^\times$, $\varphi(\xi_{a,b,c} \eta_e) = e$.
Therefore every $\alpha' \in \widehat{ \mathbb{F}_q^\times }$ produces a linear character $\alpha := \alpha' \circ \varphi \colon G_1 \to \mathbb{T}$, and the mapping $\alpha' \mapsto \alpha$ is injective.
It suffices to prove that $[G_1,G_1]$ contains $\ker \varphi = N$, so that every linear character $\alpha \colon G_1 \to \mathbb{T}$ factors through $\varphi$.

Since $q = 3^{2m+1} > 3$, a simple counting argument shows there exists $e \in \mathbb{F}_q^\times$ such that $1 \notin \{e,e^{3t-2},e^{3t-1}\}$.
Fix such an $e$.
Given $a,b,c \in \mathbb{F}_q$, put
\[
f := (e^{3t-2} - 1)^{-1} a,
\qquad
g := (e^{1-3t} - 1)^{-1}[ b - (e^{3t-2} - 1)f^{3t+1} ],
\]
\[
h := (e^{-1} - 1)^{-1} [ c - (e^{1-3t} - 1)fg - (1 - e^{3t-2} + e^{6t-4})f^{3t+2} ].
\]
Then a direct calculation shows that
\[
\eta_e^{-1} \xi_{f,g,h} \eta_e \xi_{f,g,h}^{-1}
=
\xi_{
(e^{3t-2}  - 1)f,
(e^{1-3t} - 1)g + (e^{3t-2} - 1)f^{3t+1},
(e^{-1} - 1)h + (e^{1-3t} - 1)fg + (1-e^{3t-2} + e^{6t-4})f^{3t+2}
},
\]
which equals $\xi_{a,b,c}$ by construction.
\end{proof}

\begin{proof}[Proof of Theorem~\ref{thm:Ree}]
We may assume that $q>3$.
The Schur multiplier of $G= {}^2G_2(q)$ is trivial, by~\cite[Theorem~1]{AlperinGorenstein:66}.
Hence $G^*=G$ is its own Schur cover.
Choose any linear character $\alpha \colon G_1 \to \mathbb{T}$, and let $\alpha' \in \widehat{ \mathbb{F}_q }$ be the corresponding character from Proposition~\ref{prop:ReeStab}.
Given any $e \in \mathbb{F}_q^\times$, we have $x\eta_e x^{-1} = \eta_{e^{-1}} \in G_1$.
By Theorem~\ref{thm.radicalization to Higman}, the radicalization $(\widetilde{G}^*,H)$ of $(G,G_1,\alpha)$ is a Higman pair only if $\alpha'(e) = \alpha'(e)^{-1}$ for every $e \in \mathbb{F}_q$, only if $\alpha$ is real valued.
In that case, Lemma~\ref{lem:real involution} shows that the resulting doubly transitive lines are real.
We conclude from Theorem~\ref{thm.radicalization} that every sequence of linearly dependent lines $\mathscr{L}$ satisfying ${}^2G_2(q) \leq \operatorname{Aut}\mathscr{L}$ is real.
\end{proof}

\subsection{Symplectic groups}
\label{subsec:symplectic}
Next we consider the doubly transitive actions of symplectic groups $\operatorname{Sp}(2m,2)$, summarized below.
For details, see~\cite[Section~7.7]{DixonMortimer:96}, \cite[Chapters 12--14]{Grove:02}, \cite[Chapter~2]{KleidmanLiebeck:90}.
Fix $m \geq 3$.
We consider $\mathbb{F}_2^{2m}$ to consist of column vectors $u = \left[ \begin{array}{c} u_{(1)} \\ u_{(2)} \end{array} \right]$ with $u_{(i)} \in \mathbb{F}_2^m$, $i\in\{1,2\}$.
Then $\operatorname{Sp}(2m,2) \leq \operatorname{GL}(2m,2)$ is the group of all invertible operators that stabilize the symplectic form $[\cdot ,\cdot] \colon \mathbb{F}_2^{2m} \times \mathbb{F}_2^{2m} \to \mathbb{F}_2$ given by
\[ [u,v] = u_{(1)}^\top v_{(2)} + u_{(2)}^\top v_{(1)}. \]
Put $\delta_1 = [1,0,\dotsc,0]^\top \in \mathbb{F}_2^m$.
For each $\epsilon \in \{ \pm \}$, we define a quadratic form $Q^\epsilon \colon \mathbb{F}_2^{2m} \to \mathbb{F}_2$ given by
\[
Q^+(u) = u_{(1)}^\top u_{(2)},
\qquad
\qquad
Q^-(u) = u_{(1)}^\top u_{(2)} + (u_{(1)} + u_{(2)})^\top \delta_1.
\]
Thus $Q^\epsilon(u+v) + Q^\epsilon(u) + Q^\epsilon(v) = [u,v]$ for every $u,v \in \mathbb{F}_2^{2m}$.
We take the orthogonal group $\operatorname{O}^\epsilon(2m,2) \leq \operatorname{Sp}(2m,2)$ to be the group of all linear operators stabilizing $Q^\epsilon$.
Then $\operatorname{Sp}(2m,2)$ acts doubly transitively on $n = 2^{m-1}(2^m + \epsilon)$ points, and $\operatorname{O}^\epsilon(2m,2)$ is the stabilizer of a point.

\begin{theorem}
\label{thm:Sp2m2}
Let $\mathscr{L}$ be a sequence of $n > d$ doubly transitive lines with span $\mathbb{C}^d$ and $\operatorname{Aut} \mathscr{L} \geq \operatorname{Sp}(2m,2)$ as in Proposition~\ref{prop.classification}(vi).
Then $\mathscr{L}$ is real, hence described by Proposition~\ref{prop:real2tran}(v,vi).
\end{theorem}

The Schur multiplier of $\operatorname{Sp}(2m,2)$ is trivial for $m \geq 4$.
We dispose of the case $m=3$ by computer.

\begin{lemma}[Computer-assisted result]
Assume $m=3$.
Let $\mathscr{L}$ be a sequence of $n > d$ doubly transitive lines with span $\mathbb{C}^d$ and $\operatorname{Aut} \mathscr{L} \geq \operatorname{Sp}(6,2)$ as in Proposition~\ref{prop.classification}(vi).
Then $\mathscr{L}$ is real.
\end{lemma}

\begin{proof}[Computer-assisted proof]
We used GAP~\cite{GAP} to verify that the Schur multiplier of $G=\operatorname{Sp}(6,2)$ is $\mathbb{Z}_2$.
(Note that in some places it is falsely reported as being trivial.)
Using the AtlasRep package~\cite{AtlasRep}, we obtained a double cover $\pi \colon G^* \to G$, which we checked was a Schur covering.
Then we proceeded as in the proof of Lemma~\ref{lem.PSL29}, computing the cases $\epsilon = +$ and $\epsilon = -$ separately.
For each linear character $\alpha \colon G_1^* \to \mathbb{T}$, we checked to see if the radicalization $(\widetilde{G}^*,H)$ of $(G^*,G_1^*,\alpha)$ was a Higman pair.
When it was, we computed roux parameters.
In every such case, the real roux lines detector of Propositions~\ref{prop.two reps of same lines} and~\ref{prop.real line detector} verified that all lines coming from $(\widetilde{G}^*,H)$ were real.
(Our code is available online~\cite{github}.)
\end{proof}

\begin{proof}[Proof of Theorem~\ref{thm:Sp2m2}]
We may assume that $q \geq 4$.
Then the Schur multiplier of $G = \operatorname{Sp}(2m,2)$ is trivial~\cite{Steinberg:62}, so that $G^*=G$ is its own Schur cover.
Fix $\epsilon \in \{ \pm \}$, and let $G_1 = \operatorname{O}^\epsilon(2m,2)$.
By~\cite[Proposition~14.23]{Grove:02}, $[G_1,G_1] \leq G_1$ is a subgroup of index 2.
Hence $G_1$ has only real-valued linear characters.

Put $w := \left[ \begin{array}{c} \delta_1 \\ 0 \end{array} \right]$ when $\epsilon = +$, and $w := \left[ \begin{array}{c} \delta_1 \\ \delta_1 \end{array} \right]$ when $\epsilon = -$.
Define the transvection $\tau \colon \mathbb{F}_{2m}^2 \to \mathbb{F}_{2m}^2$ by $\tau(u) = u + [u,w] w$.
Then $\tau \in \operatorname{Sp}(2m,2)$ satisfies $\tau^2 = 1$, but one easily checks that $\tau \notin G_1$.
If $\alpha \colon G_1 \to \{ \pm 1 \}$ is any linear character for which the radicalization $(\widetilde{G},H)$ of $(G,G_1,\alpha)$ is a Higman pair, then Lemma~\ref{lem:real involution} implies that all lines coming from $(\widetilde{G},H)$ are real.
By Theorem~\ref{thm.radicalization}, every sequence of linearly dependent lines $\mathscr{L}$ such that $\operatorname{Sp}(2m,2) \leq \operatorname{Aut} \mathscr{L}$ is real.
\end{proof}

\subsection{Proof of the main result}

So far, we have applied our technology to all but four of the almost simple groups in Proposition~\ref{prop.classification}.
The remaining groups can be ticked off one by one.
We found it convenient to use a computer for this task in the following lemma, but all the relevant computations could also be done by hand.

\begin{lemma}[Computer-assisted result]
\label{lem: misc cases}
Suppose $\mathscr{L}$ is a sequence of $n > d$ lines spanning $\mathbb{C}^d$, and assume that $\operatorname{Aut} \mathscr{L}$ contains one of the groups from cases (vii), (viii), (xi), (xii) of~Proposition~\ref{prop.classification}.
Then $\mathscr{L}$ is real, hence described by Proposition~\ref{prop:real2tran}.
\end{lemma}

\begin{proof}[Computer-assisted proof]
For each group $G$ with point stabilizer $G_1 \leq G$, we found a Schur covering $\pi \colon G^* \to G$ and put $G_1^* := \pi^{-1}(G_1)$.
Then we used the Higman pair detector (Theorem~\ref{thm.radicalization to Higman}) to find every linear character $\alpha \colon G_1^* \to \mathbb{T}$ for which the radicalization of $(G^*,G_1^*,\alpha)$ was a Higman pair.
In every case, we were able to apply Lemma~\ref{lem.alpha=1 is dumb} or Lemma~\ref{lem:real involution} to deduce that the corresponding lines were real.
It follows by Theorem~\ref{thm.radicalization} that $\mathscr{L}$ is real whenever $G \leq \operatorname{Aut} \mathscr{L}$.

Specifically, for case (vii) the quotient mapping $\pi \colon \operatorname{SL}(2,11) \to \operatorname{PSL}(2,11)$ is a Schur covering by \cite[Theorem~7.1.1]{Karpilovsky:87}.
Here $G_1^* \cong \operatorname{SL}(2,5)$ has no nontrivial linear characters at all.

For (viii), GAP~\cite{GAP} provided a Schur covering of $A_7$.
No nontrivial linear character produced a Higman pair through radicalization.

In case (xi), the Schur multiplier of the Higman--Sims group $G$ is $\mathbb{Z}_2$~\cite{McKayWales:71}.
We obtained a double cover $G^*$ of $G$ through the AtlasRep package~\cite{AtlasRep} of GAP, and verified that it was a Schur cover of $G$.
Only real-valued characters of $G_1^*$ produced Higman pairs through radicalization, and there exists an involution $x \in G^* \setminus G_1^*$.
By Lemma~\ref{lem:real involution}, these Higman pairs produce only real lines.

Finally, for (xii) the Schur multiplier of $G=\operatorname{Co}_3$ is trivial~\cite{Griess:74}, and so $G^* = G$ is its own Schur cover.
Here, $G_1$ has only real-valued characters to begin with.
Once again, there exists an involution $x \in G \setminus G_1$, and by Lemma~\ref{lem:real involution} any Higman pair that arises from radicalization can only create real lines.

For all cases, our code is available online~\cite{github}.
\end{proof}

\begin{proof}[Proof of Theorem~\ref{thm:main result}]
Let $\mathscr{L}$ be a sequence of $n \geq 2d > 2$ lines that span $\mathbb{C}^d$, and assume that $S = \operatorname{Aut} \mathscr{L}$ is doubly transitive and almost simple.
Then $S$ contains a subgroup $G$ such that $G\unlhd S\leq\operatorname{Aut}(G)$ and $G$ satisfies one of the cases (i)--(xii) of Proposition~\ref{prop.classification}.
In Subsections~\ref{subsec:linear}--\ref{subsec:symplectic} and Lemma~\ref{lem: misc cases}, we settled what happens in each instance except (i), (ix), and (x).
In each of those cases, $G$ is triply transitive: for (i), see \cite[Theorem~9.7]{Wielandt:64}; for (ix), see Theorems~XII.1.3 and XII.1.4 in~\cite{HuppertB:82:III}; and for (x) consult~\cite[Theorem~XII.1.9]{HuppertB:82:III}.
By Lemma~\ref{lem.3tran}, each remaining case violates the hypothesis $n \geq 2d > 2$.
\end{proof}

\begin{proof}[Proof of Theorem~\ref{thm:classification}]
By Proposition~\ref{prop: Burnside}, $S$ is either affine (as in case (I)) or almost simple (as in case (II)).
The affine case was settled by Theorem~1.1 and Remark~1.2 of Dempwolff and Kantor~\cite{DempwolffK:22}, and the almost simple case is settled by Theorem~\ref{thm:main result}.
\end{proof}

\section{Equivalence and automorphisms}
\label{sec: equivalence}
In this section, we compute the automorphism groups of complex doubly transitive lines with almost simple symmetry groups, as in Theorem~\ref{thm:classification}(D).
As an application, we sort out equivalences in the unitary case (xiii).
The remaining automorphism groups of doubly transitive lines are known.
Type~(B) appears in~\cite{DempwolffK:22}, while the two-graphs behind types~(A) and (C) are treated in~\cite{Taylor:92}.

\subsection{Linear symmetry}
We begin with the lines of Example~\ref{ex.Paley conference}, as in Theorem~\ref{thm:classification}(xii).
Let $q\equiv 3 \bmod 4$ be a prime power with prime divisor $p$.
We write $\operatorname{\text{$\Gamma$}L}(2,q)$ for the group of all permutations $f\colon \mathbb{F}_q^2 \to \mathbb{F}_q^2$ of the form $f(u) = (Mu)^{\circ p^k}$ for $M \in \operatorname{GL}(2,q)$ and $k \geq 0$, while $\operatorname{\text{$\Sigma$}L}(2,q) \leq \operatorname{\text{$\Gamma$}L}(2,q)$ is the subgroup of all those $f$ with $M \in \operatorname{SL}(2,q)$.
Both $\operatorname{\text{$\Gamma$}L}(2,q)$ and $\operatorname{\text{$\Sigma$}L}(2,q)$ permute the one-dimensional subspaces of $\mathbb{F}_q^2$, and in each case the kernel of the action coincides with the center consisting of scalar functions.
This produces embeddings of $\operatorname{P\text{$\Gamma$}L}(2,q):= \operatorname{\text{$\Gamma$}L}(2,q) / Z(\operatorname{\text{$\Gamma$}L}(2,q))$ and $\operatorname{P\text{$\Sigma$}L}(2,q):= \operatorname{\text{$\Sigma$}L}(2,q) / Z(\operatorname{\text{$\Sigma$}L}(2,q))$ in $S_X$, where $X= \{\infty\} \cup \mathbb{F}_q$ indexes the one-dimensional subspaces of $\mathbb{F}_q^2$ as in Example~\ref{ex.Paley conference}.

\begin{theorem}
\label{thm: PSL automorphism group}
Let $q\equiv 3 \bmod 4$ be a prime power, and let $\mathscr{L}$ be the sequence of lines constructed in Example~\ref{ex.Paley conference}.
Then $\operatorname{Aut} \mathscr{L} = \operatorname{P\text{$\Sigma$}L}(2,q)$.
\end{theorem}

Our proof uses the following detail.

\begin{proposition}
\label{prop: one copy of socle}
Let $S \leq S_n$ be primitive and almost simple with socle $G$.
If $H \leq S$ and $H \cong G$, then $H = G$.
\end{proposition}

\begin{proof}
Suppose instead that $H \neq G$.
Then $H \triangleright (G \cap H) = \{1\}$ since $H$ is simple and normalizes $G$.
Hence, $S \geq \langle G, H \rangle \cong G \rtimes H$, and in particular, $|G|^2 \leq |S|$.
Meanwhile, $C_S(G) = \{1\}$ by Theorem~4.2A of~\cite{DixonMortimer:96}, and so conjugation provides an embedding $S \leq \operatorname{Aut} G$.
This is impossible since every finite simple group satisfies $|\operatorname{Aut} G| \leq \frac{1}{30} |G|^2$, as detailed in Lemma~2.2 of~\cite{Quick:04}.
\end{proof}

\begin{proof}[Proof of Theorem~\ref{thm: PSL automorphism group}]
We follow the notation of Example~\ref{ex.Paley conference}, and we also denote $S:= \operatorname{Aut} \mathscr{L} \geq \operatorname{PSL}(2,q)$.
The doubly transitive group $S$ must appear in Proposition~\ref{prop.classification}.
It also acts as automorphisms of $n = q+1$ complex lines in dimension $d = n/2$.
If $S$ is affine, then Theorem~1.1 of~\cite{DempwolffK:22} reports that $(d,n)$ are listed somewhere in Theorem~\ref{thm:classification}(I).
The only parameters with $d=n/2$ are in case (ii) with $d=2$, $n=4$, and $q=3$.
These lines are unique, and they indeed have $S = A_4 = \operatorname{P\text{$\Sigma$}L}(2,3)$.
Now assume $S$ has simple socle $G$.
If $G = \operatorname{PSU}(3,q_0)$ for some prime power $q_0 > 2$, then $q+1 = n = q_0^3 + 1$, and Theorem~\ref{thm.updated unitary groups}(c) shows $d = q_0^2 - q_0 + 1$.
This is impossible since $d = n/2$.
We ruled out all other possibilities from Proposition~\ref{prop.classification}(II), and so $G \cong \operatorname{PSL}(2,q)$ and $q > 3$.
By Proposition~\ref{prop: one copy of socle}, $G = \operatorname{PSL}(2,q) \leq S_X$ (in its usual embedding).
In particular,
\[
S \leq N_{S_X}\bigl( \operatorname{PSL}(2,q) \bigr) = \operatorname{Aut}\bigl( \operatorname{PSL}(2,q) \bigr) = \operatorname{P\text{$\Gamma$}L}(2,q),
\]
where the equalities follow from~\cite[Theorem~4.2A]{DixonMortimer:96} and~\cite[\S 3.3.4]{Wilson:09}.


Now it suffices to prove $\operatorname{P\text{$\Gamma$}L}(2,q) \cap S = \operatorname{P\text{$\Sigma$}L}(2,q)$.
Select $M \in \operatorname{GL}(2,q)$ and $k \geq 0$, and let $\sigma \in \operatorname{P\text{$\Gamma$}L}(2,q) \leq S_X$ be the corresponding permutation given by $\operatorname{span}\{ t_{\sigma(i)} \} = \operatorname{span}\{ (M t_i)^{\circ p^k} \}$ for $i \in X$.
We use Proposition~\ref{prop:switching equiv is unitary equiv} to show $\sigma \in S$ if and only if $\sigma \in \operatorname{P\text{$\Sigma$}L}(2,q)$.
For each $i \in X$, select $\omega_i \in \mathbb{F}_q^\times$ with $t_{\sigma(i)} = \omega_i (M t_i)^{\circ p^k}$.
Then for $i, j \in X$,
\[
\chi\bigl( [t_{\sigma(i)}, t_{\sigma(j)} ] \bigr) 
= \chi(\omega_i) \, \chi(\omega_j) \, \chi\Bigl( (\det M)^{p^k} [t_i, t_j]^{p^k} \Bigr)
= \chi(\omega_i) \, \chi(\omega_j) \, \chi( \det M) \, \chi\bigl( [t_i, t_j] \bigr),
\]
since the Frobenius automorphism preserves the set $Q$ of quadratic residues.
Thus,
\begin{equation}
\label{eq: PSL auto 1}
\mathcal{S}_{\sigma(i),\sigma(j)} = \chi(\omega_i) \, \chi(\omega_j) \, \chi( \det M) \, \mathcal{S}_{i,j}
\qquad
\text{for every $i,j \in X$.}
\end{equation}
When $M = c M'$ for some $c \in \mathbb{F}_q^\times$ and $M' \in \operatorname{SL}(2,q)$, we have $\det M = c^2$ and $\chi( \det M) = 1$, and then Proposition~\ref{prop:switching equiv is unitary equiv} implies $\sigma \in S$.
That is, $\operatorname{P\text{$\Sigma$}L}(2,q) \leq \operatorname{P\text{$\Gamma$}L}(2,q) \cap S$.

To prove the reverse inclusion, assume for the sake of contradiction that $\sigma \in S$ and yet $M \neq c M'$ for any $c \in \mathbb{F}_q^\times$ and $M' \in \operatorname{SL}(2,q)$.
Then it is easy to show $\chi( \det M ) = -1$.
By Proposition~\ref{prop:switching equiv is unitary equiv} and~\eqref{eq: PSL auto 1}, there are unimodular constants $\{ c_i \}_{i\in X}$ such that
$c_i \overline{c_j} \mathcal{S}_{i,j} 
= \mathcal{S}_{\sigma(i), \sigma(j)}
= - \chi(\omega_i) \, \chi(\omega_j) \, \mathcal{S}_{i,j}$
for every $i,j \in X$.
Equivalently,
\begin{equation}
\label{eq: PSL auto 2}
c_i \overline{c_j} \chi\bigl( [t_i, t_j] \bigr)
= - \chi( \omega_i) \, \chi( \omega_j ) \, \chi\bigl( [t_i, t_j] \bigr)
\qquad
\text{for every $i \neq j$ in $X$}.
\end{equation}
Taking $i = \infty$ and $j = a\in \mathbb{F}_q \subset X$ in~\eqref{eq: PSL auto 2}, we find $c_\infty \overline{c_a} = - \chi( \omega_\infty)\, \chi( \omega_a )$, so that
\begin{equation}
\label{eq: PSL auto 3}
c_a\,  \chi( \omega_a) 
= -  c_\infty \, \chi( \omega_\infty)
\end{equation}
is constant for all $a \in \mathbb{F}_q$.
Taking $i = 0$ and $j = -1$ in~\eqref{eq: PSL auto 2}, we find $c_0 \overline{c_{-1}} = - \chi( \omega_0) \, \chi( \omega_{-1} )$, so that $c_0 \, \chi( \omega_0 )  = -  c_{-1} \, \chi( \omega_{-1} )$.
This contradicts~\eqref{eq: PSL auto 3}.
\end{proof}

\subsection{Unitary symmetry}

Now we handle the lines of Theorem~\ref{thm.updated unitary groups}, as in Theorem~\ref{thm:classification}(xiii).
Let $q>2$ be a prime power with prime divisor~$p$, and let notation be as in the first paragraph of Subsection~\ref{subsec:unitary}.
We write $\operatorname{\text{$\Gamma$}U}(3,q)$ for the group of all functions $f \colon \mathbb{F}_{q^2}^3 \to \mathbb{F}_{q^2}^3$ with $f(v) = (Uv)^{\circ p^k}$ for $U \in \operatorname{GU}(3,q)$ and $k \geq 0$.
Then $\operatorname{\text{$\Gamma$}U}(3,q)$ permutes the isotropic lines in $\mathbb{F}_{q^2}^3$, and the kernel of this action consists of scalar functions, which form the center of $\operatorname{\text{$\Gamma$}U}(3,q)$.
This provides an embedding of $\operatorname{P\text{$\Gamma$}U}(3,q) := \operatorname{\text{$\Gamma$}U}(3,q) / Z\bigl( \operatorname{\text{$\Gamma$}U}(3,q) \bigr)$ in $S_X$.

\begin{lemma}
\label{lem: unitary equiv}
Let $\sigma \in \operatorname{P\text{$\Gamma$}U}(3,q) \leq S_X$, where there is a unitary $U \in \operatorname{GU}(3,q)$ and $k \geq 0$ such that $\operatorname{span}\{t_{\sigma(i)} \} = \operatorname{span}\{(Ut_i)^{\circ p^k}\}$ for all $i \in X$.
For any choice of nontrivial characters $\alpha,\beta \colon \mathbb{T}_q \to \mathbb{T}$, it holds that $\alpha = \beta^{p^k}$ if and only if $\sigma$ induces a unitary equivalence $\mathscr{L}_\beta \to \mathscr{L}_\alpha$ between the corresponding lines of Theorem~\ref{thm.updated unitary groups}(b).
\end{lemma}

\begin{proof}
We follow the notation of Theorem~\ref{thm.updated unitary groups}, and we also write $\mathcal{S}_\beta = \Bigl[ - \beta\bigl(-(t_i,t_j)^{q-1} \bigr) \Bigr]_{i,j \in X}$ for the signature matrix that produces $\mathscr{L}_\beta$.
Select constants $\{ \omega_i \}_{i\in X}$ in $\mathbb{F}_{q^2}^\times$ such that $\omega_i t_{\sigma(i)} = (U t_i)^{\circ p^k}$ for each $i \in X$, and put $c_i := \beta^{p^k}(\omega_i^{q-1})$.
For $i\neq j$ in $X$, we have $(t_{\sigma(i)}, t_{\sigma(j)}) = \omega_i^{p^k} \omega_j^{q p^k} (t_i, t_j)^{p^k}$, so that
\begin{equation}
\label{eq: unitary equiv 1}
[\mathcal{S}_\beta]_{\sigma(i), \sigma(j)}
= -\beta \bigl( - (t_{\sigma(i)}, t_{\sigma(j)} ) \bigr)
= - \beta^{p^k}\bigl( \omega_i^{q-1} \bigr) \, \beta^{p^k}\bigl( \omega_j^{q(q-1)} \bigr) \, \beta^{p^k}\bigl( -(t_i,t_j)^{q-1} \bigr)
= c_i \overline{c_j} [ \mathcal{S}_{\beta^{p^k}} ]_{i,j}.
\end{equation}
By Proposition~\ref{prop:switching equiv is unitary equiv}, $\sigma$ induces a unitary equivalence $\mathscr{L}_\beta \to \mathscr{L}_{\beta^{p^k}}$.

Conversely, suppose $\sigma$ induces a unitary equivalence $\mathscr{L}_\beta \to \mathscr{L}_\alpha$.
By Proposition~\ref{prop:switching equiv is unitary equiv} and~\eqref{eq: unitary equiv 1}, there are unimodular constants $\{c_i' \}_{i \in X}$ such that $c_i' \overline{c_j'} [ \mathcal{S}_{\alpha} ]_{i,j} = [ \mathcal{S}_{\beta} ]_{\sigma(i), \sigma(j)} = c_i \overline{c_j} [ \mathcal{S}_{\beta^{p^k}} ]_{i,j}$ for every $i,j \in X$.
In other words,
\begin{equation}
\label{eq: unitary equiv 2}
c_i' \overline{c_j'} \, \alpha \bigl( - (t_i,t_j)^{q-1} \bigr) = c_i \overline{c_j} \, \beta^{p^k}\bigl( - (t_i, t_j)^{q-1} \bigr)
\qquad
\text{for every $i \neq j$ in $X$}.
\end{equation}
Take $i \in T$ and $j = \infty$ in~\eqref{eq: unitary equiv 2} to obtain $c_i' \overline{ c_\infty' }\, \alpha(-1) = c_i \overline{c_\infty}\, \beta(-1)$, so that
\begin{equation}
\label{eq: unitary equiv 3}
c_i' \overline{c_i} = c_{\infty}' \overline{ c_\infty } \, \alpha(-1)\, \beta(-1) =: C \in \mathbb{T}
\end{equation}
is constant for $i \in T$.
Select any $e \in \mathbb{T}_q$, and find $b \in \mathbb{F}_{q^2}^\times$ with $b^{q-1} = -e$.
Then find $a \in \mathbb{F}_{q^2}$ with $a^{q+1} = -b-b^q$, so that $i:=[a,b]^\top \in T$.
For $j := [0,0]^\top \in T$, we have $(t_i,t_j)^{q-1} = b^{q-1} = -e$.
Then~\eqref{eq: unitary equiv 2} produces $c_i' \overline{c_j'} \, \alpha(e) = c_i \overline{c_j} \beta^{p^k}(e)$, and $\beta^{p^k}(e) = c_i' \overline{c_i} \overline{c_j'} c_j \alpha(e) = C \overline{C} \alpha(e) = \alpha(e)$ by~\eqref{eq: unitary equiv 3}.
\end{proof}

\begin{theorem}
\label{thm:unitary automorphisms}
Let $\beta \colon \mathbb{T}_q \to \mathbb{T}$ be a nontrivial character that takes non-real values, and let $\mathscr{L}_\beta$ be the lines of Theorem~\ref{thm.updated unitary groups}(b).
Then
$
\operatorname{PSU}(3,q) \trianglelefteq \operatorname{Aut} \mathscr{L}_\beta \leq \operatorname{P\text{$\Gamma$}U}(3,q).
$
Specifically, $\operatorname{Aut} \mathscr{L}_\beta$ consists of all $\sigma \in S_X$ for which there exist $U \in \operatorname{GU}(3,q)$ and $k \geq 0$ such that $\operatorname{span}\{ t_{\sigma}(i) \} = \operatorname{span}\{ (U t_i)^{\circ p^k} \}$ for all $i \in X$ and $\beta^{p^k} = \beta$.
\end{theorem}

The case where $\beta$ is real is treated in~\cite{Taylor:92}: $\operatorname{Aut} \mathscr{L}_\beta = \operatorname{P\text{$\Gamma$}U}(3,q)$ unless $q=3$, in which case $\operatorname{Aut} \mathscr{L}_\beta = \operatorname{Sp}(6,2)$.

\begin{proof}[Proof of Theorem~\ref{thm:unitary automorphisms}]
The condition that $\beta$ takes non-real values ensures $\mathscr{L}_\beta$ are not real lines.
Put $S:= \operatorname{Aut} \mathscr{L}_\beta$.
We first show $S$ is almost simple with socle $\operatorname{PSU}(3,q)$.
As in the proof of Theorem~\ref{thm: PSL automorphism group}, $S$ is doubly transitive and must appear in Proposition~\ref{prop.classification}.
If $S$ is affine, then $(d,n)$ appear somewhere in Theorem~\ref{thm:classification}(B).
We have $d=q^2-q+1$ and $n=q^3+1$, which excludes (iii) and~(iv).
The only other possibility is~(ii), with $\frac{n}{d} = \frac{2\sqrt{n}}{\sqrt{n}-1}$ and $n$ a square prime power.
Meanwhile, the true parameters satisfy
\[
\frac{n}{d}=\frac{q^3+1}{q^2-q+1} > \frac{q^3}{q^2-q+q}=q=\sqrt[3]{n-1}.
\]
For $n > 18$, we have $\frac{n}{d} > \sqrt[3]{n-1} > \frac{2\sqrt{n}}{\sqrt{n}-1}$, and after checking all possibilities for $n \leq 18$ we find that case~(ii) does not occur.
Overall, $S$ is not affine, and so it is almost simple.
Theorem~\ref{thm.linear groups} shows its socle cannot be $\operatorname{PSL}(2,q')$ for any prime power $q'$, since $\frac{n}{d} > q > 2$.
We have ruled out all other possibilities, and so the socle $S$ is isomorphic to $\operatorname{PSU}(3,q)$.
By Proposition~\ref{prop: one copy of socle}, the socle equals $\operatorname{PSU}(3,q)\leq S_X$ (in its usual embedding).
In particular,
\begin{equation}
\label{eq: PSU normalizer}
S \leq N_{S_X}\bigl( \operatorname{PSU}(3,q) \bigr) = \operatorname{Aut}\bigl( \operatorname{PSU}(3,q) \bigr) = \operatorname{P\text{$\Gamma$}U}(3,q),
\end{equation}
where the equalities follow from~\cite[Theorem~4.2A]{DixonMortimer:96} and~\cite[\S 3.6.3]{Wilson:09}.
The desired result now follows immediately from Lemma~\ref{lem: unitary equiv}.
\end{proof}

\begin{theorem}
\label{thm:unitary equivalences}
Let $\alpha,\beta \colon \mathbb{T}_q \to \mathbb{T}$ be nontrivial characters, both of which take non-real values.
Then the lines $\mathscr{L}_\alpha$ and $\mathscr{L}_\beta$ constructed in Theorem~\ref{thm.updated unitary groups}(b) are unitarily equivalent if and only if $\alpha = \beta^{p^k}$ for some $k \geq 0$.
\end{theorem}

\begin{proof}
Lemma~\ref{lem: unitary equiv} establishes that $\mathscr{L}_\beta$ is unitarily equivalent with $\mathscr{L}_{\beta^{p^k}}$ for every $k \geq 0$.
Conversely, suppose $\sigma \in S_X$ induces a unitary equivalence $\mathscr{L}_\beta \to \mathscr{L}_\alpha$.
Then $\sigma ( \operatorname{Aut} \mathscr{L}_\beta ) \sigma^{-1} = \operatorname{Aut} \mathscr{L}_\alpha$, and in particular, $\sigma$ conjugates the socle of $\operatorname{Aut} \mathscr{L}_\beta$ to that of $\operatorname{Aut} \mathscr{L}_\alpha$.
Both automorphism groups have socle $\operatorname{PSU}(3,q)$ by Theorem~\ref{thm:unitary automorphisms}, so $\sigma \in N_{S_X}\bigl( \operatorname{PSU}(3,q) \bigr) = \operatorname{P\text{$\Gamma$}U}(3,q)$ as in~\eqref{eq: PSU normalizer}.
Now Lemma~\ref{lem: unitary equiv} shows that $\alpha = \beta^{p^k}$ for some $k \geq 0$.
\end{proof}

\section*{Acknowledgments}

This work was partially supported by NSF DMS 1321779, NSF DMS 1829955, ARO W911NF-16-1-0008, AFOSR F4FGA06060J007, AFOSR FA9550-18-1-0107, AFOSR Young Investigator Research Program award F4FGA06088J001, and an AFRL Summer Faculty Fellowship Program award.
The authors thank Matt Fickus and John Jasper for many stimulating conversations, and they are also pleased to thank Bill Kantor and the anonymous referees for suggestions that substantially improved the quality of the manuscript.
In particular, the equivalences identified in Remark~\ref{rem: coincidences} were made possible thanks to some insightful suggestions provided by one reviewer.
The views expressed in this article are those of the authors and do not reflect the official policy or position of the United States Air Force, Army, Department of Defense, or the U.S.\ Government.

\bibliographystyle{abbrv}
\bibliography{references}

\end{document}